\pdfoutput=1
\documentclass[11pt,a4paper]{article}
\usepackage[utf8]{inputenc}
\usepackage{jheppub}
\hypersetup{unicode=true,bookmarksopen=true}

\usepackage{mathtools}
\mathtoolsset{showonlyrefs=true}
\usepackage{mlmodern}
\usepackage[T1]{fontenc}
\usepackage{bbm}
\DeclareMathAlphabet{\mathbfi}{OML}{cmm}{b}{it}

\usepackage{amsthm}
\newtheorem{theorem}{Theorem}[section]
\newtheorem{lemma}[theorem]{Lemma}
\newtheorem{corollary}[theorem]{Corollary}
\newtheorem{assumption}[theorem]{Assumption}
\newtheorem{proposition}[theorem]{Proposition}

\theoremstyle{definition}
\newtheorem{definition}[theorem]{Definition}

\newtheorem{example}[theorem]{Example}

\let\originalleft\left
\let\originalright\right
\renewcommand{\left}{\mathopen{}\mathclose\bgroup\originalleft}
\renewcommand{\right}{\aftergroup\egroup\originalright}

\makeatletter
\newenvironment{equations}[1][]{\subequations\ifx\relax#1\relax\else\label{#1}\fi\align\ignorespaces}{\endalign\ignorespacesafterend\endsubequations}
\def\@spliteq#1{\begin{equation}\begin{split}#1\end{split}\end{equation}}
\def\splitequation{\collect@body\@spliteq}

\makeatother

\newcommand{\eqend}[1]{\,#1}

\newcommand{\mathe}{\mathrm{e}}
\newcommand{\mathi}{\mathrm{i}}
\let\oldre\Re
\let\oldim\Im
\renewcommand{\Re}{\oldre\mathfrak{e}\,}
\renewcommand{\Im}{\oldim\mathfrak{m}\,}
\newcommand{\total}{\mathop{}\!\mathrm{d}}

\makeatletter
\def\abs{\@ifnextchar[\abs@size\abs@nosize}
\def\abs@size[#1]#2{\mathopen{#1\lvert}{#2}\mathclose{#1\rvert}}
\def\abs@nosize#1{\left\lvert{#1}\right\rvert}
\def\norm{\@ifnextchar[\norm@size\norm@nosize}
\def\norm@size[#1]#2{\mathopen{#1\lVert}{#2}\mathclose{#1\rVert}}
\def\norm@nosize#1{\left\lVert{#1}\right\rVert}
\makeatother

\newcommand{\1}{\mathbbm{1}}

\newcommand{\Tr}{\operatorname{Tr}}
\newcommand{\dom}{\operatorname{dom}}
\newcommand{\Hil}{\mathcal{H}}
\newcommand{\BH}{{\mathcal{B}(\mathcal{H})}}
\newcommand{\M}{\mathcal{M}}
\newcommand{\PM}{\mathcal{P}(\mathcal{M})}
\newcommand{\I}{\mathcal{I}}
\newcommand{\N}{\mathcal{N}}
\newcommand{\sth}{_{*,\mathrm{h}}}
\newcommand{\stp}{_{*,+}}

\newcommand{\half}{{\mathinner{1/2}}}
\newcommand{\pinv}{{\mathinner{1/p}}}

\newcommand{\defeq}{\coloneq}
\newcommand{\eqdef}{\eqcolon}

\makeatletter
\newcommand*{\n@meref}[3][]{#3~\ref{#2}\if\relax\detokenize{#1}\relax\else ~\ref{#1}\fi}
\newcommand*{\secref}[2][]{\n@meref[#1]{#2}{Sec.}}
\newcommand*{\defref}[2][]{\n@meref[#1]{#2}{Def.}}
\newcommand*{\lemmaref}[2][]{\n@meref[#1]{#2}{Lemma}}
\newcommand*{\propref}[2][]{\n@meref[#1]{#2}{Prop.}}
\newcommand*{\thmref}[2][]{\n@meref[#1]{#2}{Thm.}}
\newcommand*{\corref}[2][]{\n@meref[#1]{#2}{Cor.}}
\newcommand*{\exref}[2][]{\n@meref[#1]{#2}{Example}}
\makeatother

\makeatletter
\gdef\@fpheader{\strut}
\makeatother

\allowdisplaybreaks

\begin{document}

\title{Integral representations of $f$-divergences for general von Neumann algebras}

\author[a]{Ricardo Correa da Silva,}
\author[b]{Markus B. Fr{\"o}b,}
\author[b]{Gandalf Lechner}
\author[c]{and Leonardo Sangaletti}

\affiliation[a]{Instituto de F{\'\i}sica, Universidade de S{\~a}o Paulo, Rua do Mat{\~a}o 1371, 05508-090 S{\~a}o Paulo, Brazil}
\affiliation[b]{Department Mathematik, Friedrich-Alexander-Universit{\"a}t Erlangen-N{\"u}rnberg, Cauerstra{\ss}e 11, 91058 Erlangen, Germany}
\affiliation[c]{Dipartimento di Fisica, Universit{\`a} di Genova, Via Dodecaneso 33, 16146 Genova, Italy}

\emailAdd{ricardo.correa.silva@usp.br}
\emailAdd{markus.froeb@fau.de}
\emailAdd{gandalf.lechner@fau.de}
\emailAdd{leonardo.sangaletti@edu.unige.it}

\abstract{We define and analyze hockeystick divergences and $f$-divergences for normal positive functionals on general von Neumann algebras, generalizing and unifying previous work in classical probability and finite-dimensional von Neumann algebras. All the main properties of these state distinguishability measures (including in particular monotonicity, convexity, semicontinuity, bounds, state discrimination, data processing inequality) are derived from properties of the Jordan decomposition of selfadjoint normal functionals. This is done by representing the $f$-divergences as integrals over hockeystick divergences, and their significance in quantum hypothesis testing is reviewed. The $f_0$-divergence given by the information function $f_0(t) = t \ln t$ is shown to coincide with Araki's relative entropy, extending results of Frenkel to general von Neumann algebras.}


\maketitle

\section{Introduction and overview}
\label{sec:intro}

\paragraph{$f$-divergences in classical probability.} A central question in probability and information theory and their applications to physics consists in distinguishing states of interest, and finding operationally meaningful ways of quantifying their distinguishability. In classical probability, classical $f$- or Csisz{\'a}r divergences~\cite{csiszar1963} $D_f$ (distinguishability measures) are among the basic tools for doing so: Given two probability measures $\mu_1,\mu_2$ on a measurable space $X$, which for the purposes of this introduction we may assume to be absolutely continuous with respect to each other ($\mu_1 \ll \mu_2 \ll \mu_1$), and a convex function $f \colon (0,\infty) \to \mathbb{R}_+$, they are defined as \cite[Def.~3]{sasonverdu2016}
\begin{equation}
\label{eq:DfIntro}
D_f(\mu_1 \vert \mu_2) = \int_X f\left( \frac{\total \mu_1}{\total \mu_2} \right) \total \mu_2 \eqend{.}
\end{equation}
Different choices of $f$ lead to divergences with different features. Within the rich class of $f$-divergences, two cases are of particular importance: The Kullback--Leibler divergence or relative entropy and the hockeystick divergences. The Kullback--Leibler divergence $D_{f_0}$ is given by setting $f$ in \eqref{eq:DfIntro} equal to the information function
\begin{equation}
f_0(t) = t \ln t \eqend{,}
\end{equation}
and is distinguished by its additivity properties over tensor products. On the other hand, the hockeystick divergence at likelihood parameter $t > 0$, which owes its name to the shape of the graph of $s \mapsto (s-t)_+ = \max( s-t, 0 )$, is defined as $E_t \defeq D_{(\,\cdot\, -t)_+}$, also plays a special role because the other $f$-divergences can be expressed as integrals over hockeystick divergences. Namely, as shown by Sason and Verd{\'u} \cite{sasonverdu2016}, for twice differentiable convex $f \colon (0,\infty) \to (0,\infty)$ with $f(1) = 0$, it holds that\footnote{This formula follows from straightforward integration. It will also be a special case of our analysis in \secref{sec:hockeystick:fdivergence}.}
\begin{equation}
\label{eq:DIntegralIntro}
D_f(\mu_1 \vert \mu_2) = \int_1^\infty \left[ f''(t) E_t(\mu_1 \vert \mu_2) + t^{-3} f''\left( t^{-1} \right) E_t(\mu_2 \vert \mu_1) \right] \total t \eqend{.}
\end{equation}
Hence many features of a whole family of $f$-divergences can be derived from features of hockeystick divergences, even if the hockeystick divergence at fixed $t$ lacks some desirable features, such as the ability to distinguish arbitrary probability measures.

\paragraph{$f$-divergences in quantum theory.} When passing from this classical setting to the quantum setting, as appropriate for applications in quantum physics, the classical measure spaces $(X,\mu)$ are replaced by noncommutative measure spaces, namely von Neumann algebras~$\M$~\cite{takesaki2001}. This includes the classical situation as the commutative von Neumann algebra $\M = L^\infty(X,\mu)$, with normal states given by integrating against probability measures. In quantum information theory \cite{watrous2018,holevo2019}, one often considers finite-dimensional matrix algebras $\M = M_N(\mathbb{C})$, although also more general systems are of prominent interest from the point of view of entanglement \cite{vanluijkstottmeisterwernerwilming2024,vanluijkstottmeisterwernerwilming2025}. In other physical applications, various types of von Neumann algebras appear anyway, such as infinite-dimensional type I algebras like $\BH$ in quantum mechanics, type II algebras in certain models of quantum statistical mechanics \cite{emch1972}, and type III algebras as local observable algebras in quantum field theory \cite{haag1996}. Because interesting information-theoretic questions arise in all of these scenarios, we would like to define and study $f$-divergences for general von Neumann algebras. This is the subject of the present work. As this is a long article, we now give a rather detailed discussion of more background, relation of our work to the existing literature, and a guideline to the main contents and results.

Even in the simplest case of matrix algebras, states $\varphi_j(\,\cdot\,) = \Tr\left( h_j \,\cdot\, \right)$, $j = 1,2$, will typically have noncommuting density matrices $h_1, h_2$. This leads to many new effects and challenges: for instance, the naive replacement $h_1 h_2^{-1}$ of the classical Radon--Nikod{\'y}m derivative $\total \mu_1/\total \mu_2$ is no longer positive. Among various noncommutative replacements discussed in the literature \cite{puszworonowicz1975,kuboando1980}, the relative modular operator $\Delta_{\varphi_1,\varphi_2}$ \cite{araki1976,araki1977,arakimasuda1982} provides a natural replacement which is available in generality.\footnote{We review the necessary mathematical background in \secref{sec:relentropy}.} For example, Araki's definition of the relative entropy $S^\M_\mathrm{rel}(\varphi_1 \vert \varphi_2)$ of two normal states $\varphi_1,\varphi_2$ on a von Neumann algebra $\M$ \cite{araki1976,araki1977} amounts to replacing $\total \mu_1/\total \mu_2$ by $\Delta_{\varphi_1,\varphi_2}$ in \eqref{eq:DfIntro}, with $f = f_0$, and reading the integral against~$\mu_2$ as evaluation in the functional~$\varphi_2$ (see \defref{def:RelativeEntropy} for details). Similarly, general $f$-divergences have been extended to the von Neumann algebra setting by Petz~\cite{petz1985,petz1986a}, see also Hiai~\cite{hiai2018}.

\paragraph{Hockeystick divergences and quantum hypothesis testing.} However, the approach to quantum $f$-divergences based on relative modular operators does not take into account the significance of hockeystick divergences in Bayesian (quantum) hypothesis testing (see \cite{hayashi2017,jaksicogatapilletseiringer2012,hirche2018} for a small sample of a large body of literature). For a quick review of this perspective, consider a von Neumann algebra $\M$ representing a physical system, two normal states $\varphi_1,\varphi_2$ on $\M$, and the two hypotheses ``the system is in state $\varphi_j$'', $j=1,2$. We use a projection $p \in \M$ as a test for the first hypothesis, namely we interpret $p$ taking the value~$1$ as indicating that the state is $\varphi_1$, and $p$ taking the value $0$ as indicating that the state is~$\varphi_2$. In case prior probabilities $q_1, q_2 \in (0,1)$ with $q_1 + q_2 = 1$ for the two hypotheses are known, the success probability of this test is
\begin{equation}
\label{eq:SuccessProbabilityTest}
\mathbb{P}(q_1 \varphi_1, q_2 \varphi_2; p) \defeq q_1 \varphi_1(p) + q_2 \varphi_2(p^\perp) = q_2 + q_1 \left[ \varphi_1(p) - \frac{q_2}{q_1} \varphi_2(p) \right] \eqend{.}
\end{equation}

We denote by $t \defeq q_2/q_1$ the likelihood ratio and would like to maximize the success probability for fixed $q_1,q_2$. For this, we need to choose $p$ such that the weighted difference $\varphi_1(p) - t \varphi_2(p)$ becomes maximal. The optimal success probability
\begin{equations}
\mathbb{P}_\mathrm{max}(q_1 \varphi_1, q_2 \varphi_2) &\defeq \sup_p \mathbb{P}(q_1 \varphi_1, q_2 \varphi_2; p) = q_2 + q_1 \widetilde{E}_t(\varphi_1 \vert \varphi_2) \eqend{,} \\
\widetilde{E}_t(\varphi_1 \vert \varphi_2) &\defeq \sup_p \Bigl[ \varphi_1(p) - t \varphi_2(p) \Bigr]
\end{equations}
is called the \emph{Helstrom bound} in the classical and quantum information literature \cite{helstrom1969}, \cite[Thm.~3.4]{watrous2018}, and the optimal test $p$ is called a \emph{Neyman--Pearson} test (see, e.g., \cite[Lemma~3]{jencova2010}). As $q_2$ corresponds to the success probability using the trivial test $p = 0$ (meaning always accepting the second hypothesis), the quantity $\widetilde{E}_t(\varphi_1 \vert \varphi_2)$ measures the excess optimal success probability above the trivial strategy, normalized by the prior $q_1$.

It is instructive to consider the classical case of a probability space $(X,\mu)$, where $\M = L^\infty(X,\mu)$, and $\mu_1 \ll \mu_2 \ll \mu$ are two probability measures absolutely continuous with respect to $\mu$, defining the two states $\varphi_j(f) \defeq \int_X f \total \mu_j$, $j=1,2$. In this case, it is easy to check that $\widetilde{E}_t(\varphi_1 \vert \varphi_2)$ coincides with the previously introduced hockeystick divergence:
\begin{equation}
\widetilde{E}_t(\varphi_1 \vert \varphi_2) = E_t(\varphi_1 \vert \varphi_2) = \int_X \left( \frac{\total \mu_1}{\total \mu_2} - t \right)_+ \total \mu_2 = \norm{ (\mu_1 - t \mu_2)_+ } \eqend{.}
\end{equation}
In the last expression, the norm is the total variation norm of the positive part $(\mu_1 - t \mu_2)_+$ of the signed measure $\mu_1 - t \mu_2$.

In von Neumann algebra language, this norm is $\norm{ (\varphi_1 - t \varphi_2)_+ }$, the norm of the positive part of the Jordan decomposition of the selfadjoint normal functional $\varphi_1 - t \varphi_2$ in $L^\infty(X,\mu)_* \cong L^1(X,\mu)$. This observation suggests to define hockeystick divergences of two normal positive selfadjoint functionals $\varphi_1,\varphi_2$ on a general von Neumann algebra $\M$ by
\begin{equation}
\label{eq:EIntro}
E^\M_t(\varphi_1 \vert \varphi_2) \defeq \norm{ (\varphi_1 - t \varphi_2)_+ }_{\M_*} \eqend{,}
\end{equation}
which we will take as our definition of hockeystick divergence, \defref{def:hockeystick}. It directly generalizes the definition of hockeystick divergences for states with density matrices $h_1,h_2$ on matrix algebras, $E^{M_N(\mathbb{C})}_t(h_1 \vert h_2) = \Tr\left[ (h_1 - t h_2)_+ \right]$ \cite{sharmawarsi2012,sharmawarsi2013} (see also \cite{hirchetomamichel2024} for a closely related approach), and has a clearcut interpretation as a distinguishability measure for states in terms of quantum hypothesis testing.

Note that the hockeystick divergences defined as \eqref{eq:EIntro} are independent of modular theory and typically \emph{different} from replacing the Radon--Nikod{\'y}m derivative $\total \mu_1 / \total \mu_2$ in $D_{(\,\cdot\, - t)_+}$ \eqref{eq:DfIntro} by the relative modular operator $\Delta_{\varphi_1,\varphi_2}$. For example, for a matrix algebra and two states with invertible densities $h_1,h_2$, using the relative modular operator would result in $\sum_{\lambda_1,\lambda_2} (\lambda_1 - t \lambda_2)_+ \Tr\left[ E^{h_1}(\lambda_1) E^{h_2}(\lambda_2) \right]$, where $\lambda_j$ and $E^{h_j}(\lambda_j)$ are the eigenvalues and eigenprojections of $h_j$, $j=1,2$. If $h_1$ and $h_2$ do not commute, this is typically different from $\Tr\left[ (h_1 - t h_2)_+ \right]$. Therefore, also the $f$-divergences \eqref{eq:DIntegralIntro} will be typically different from the ones defined using the relative modular operator~\cite{petz1985,petz1986a,hiai2018}, which can also be seen numerically~\cite{hirchetomamichel2024}.

\paragraph{Aims and overview of the article, connection to the literature.} Given the background described up this point, this article has two aims. The first one is the introduction and analysis of hockeystick divergences \eqref{eq:EIntro} and $f$-divergences on general von Neumann algebras, defined in \defref{def:fdivergence} by the analogue of the integral formula \eqref{eq:DIntegralIntro} over hockeystick divergences, unifying and extending the special cases studied before \cite{hirchetomamichel2024,frenkel2023,sasonverdu2016,liuhirchecheng2025,jencova2024,vanluijkwilming2026}.

Our analysis starts in \secref{sec:hockeystick}, where we first recall the Jordan decomposition of normal selfadjoint functionals $\varphi$, discuss typical example scenarios, and formulate the main properties of the \emph{positive variation}, the map $\varphi \mapsto \norm{ \varphi_+ }$ in \secref{sec:hockeystick:jordan}. Based on this, we introduce hockeystick divergences in \secref{sec:hockeystick:hockeystick}, and afterwards $f$-divergences in \secref{sec:hockeystick:fdivergence}.

Our main result regarding this first aim is \thmref{thm:fdivergence_properties}, which collects various properties of $f$-divergences on general von Neumann algebras. These properties include, amongst others, monotonicity, convexity and semicontinuity properties, bounds, state discrimination properties and the data processing inequality. An attractive feature of the approach via hockeystick divergences is that the proofs of all these properties are simple, and can in fact be carried over from the finite-dimensional setting \cite{sharmawarsi2012,sharmawarsi2013,hircherouzefranca2023,liuhirchecheng2025} in various instances. In particular, the proof of the data processing inequality can be carried out in a few lines, in contrast to earlier more involved arguments \cite{beigi2013,muellerhermesreeb2017,jencova2018}. In this regard, $f$-divergences are similar to Kosaki's approach to relative entropy \cite{kosaki1986}, which also allows for simple proofs of various properties, albeit restricted to one particular divergence. However, Kosaki's variational formula has a more complicated expression in comparison to \eqref{eq:DIntegralIntro} and \eqref{eq:Df0Intro}.

As mentioned before, quantum $f$-divergences defined on the basis of the quantum hockeystick divergences \eqref{eq:EIntro} are in general different from the ones defined using the relative modular operator. However, for the special case of the relative entropy, there have recently been various important results indicating a close connection between $D^\M_{f_0}$ and Araki's relative entropy $S^\M_\mathrm{rel}$. For positive normal functionals $\varphi_1,\varphi_2$, the divergence $D^\M_{f_0}$ takes the form
\begin{equation}
\label{eq:Df0Intro}
D^\M_{f_0}(\varphi_1 \vert \varphi_2) = \int_1^\infty \left[ \frac{1}{t} E^\M_t(\varphi_1 \vert \varphi_2) + \frac{1}{t^2}E^\M_t(\varphi_2 \vert \varphi_1) \right] \total t + \varphi_1(\1) - \varphi_2(\1) \eqend{.}
\end{equation}
This line of research began with the work of Frenkel, who proved $D^{M_N(\mathbb{C})}_{f_0} = S^{M_N(\mathbb{C})}_\mathrm{rel}$ for matrix algebras (in an equivalent form), and showed how the data processing inequality follows from this integral representation \cite{frenkel2023}. This result has subsequently been used and extended by several authors: Jen{\v c}ov{\'a} extended Frenkel's formula to trace-class operators on a separable Hilbert space $\Hil$ \cite{jencova2024}, and Friedland extended it to bounded operators, establishing $D^\BH_{f_0} = S^\BH_\mathrm{rel}$ \cite{friedland2026}. Hirche and Tomamichel defined quantum $f$-divergences as described here for $\M = M_N(\mathbb{C})$ \cite{hirchetomamichel2024}. Liu, Hirche, and Cheng proposed an alternative proof of Frenkel's formula via so-called operator layer cake decompositions \cite{liuhirchecheng2025}, generalizing the well-known layer cake representation from measure theory \cite[Thm.~1.13]{liebloss1997}. Van Luijk and Wilming have shown that Frenkel's finite-dimensional result carries over to approximately finite-dimensional von Neumann algebras~\cite{vanluijkwilming2026}.

\paragraph{Proving $\boldsymbol{D^\M_{f_0} = S^\M_\mathrm{rel}}$.} In view of these results, one might wonder whether the $f_0$-diver\-gence $D^\M_{f_0}$ coincides with Araki's relative entropy $ S^\M_\mathrm{rel}$ also for general von Neumann algebras, which would open a new perspective on relative entropy in settings where infinite-dimensional von Neumann algebras are relevant, such as for example quantum field theory. Proving this is the second objective of this paper.

While $f$-divergences and their properties are based entirely on the Jordan decomposition, Araki's approach to relative entropy is based on relative modular theory, and the connection between the two is non-trivial already in the case of matrix algebras. We present two approaches to this question: In \secref{sec:ModularBounds}, we show how hockeystick divergences $E^\M_t(\varphi_1 \vert \varphi_2)$ can be bounded above and below by expressions involving $\Delta_{\varphi_2,\varphi_1}$, partially building on previous work by Jak{\v s}i{\'c}, Ogata, Pillet, and Seiringer in quantum hypothesis testing \cite{jaksicogatapilletseiringer2012} and Ogata's generalization \cite{ogata2011} of the Powers--St{\o}rmer inequality \cite{powersstormer1970}. In integrated form, this leads to bounds on the $f_0$-divergence, including in particular the simple upper bound
\begin{equation}
D^\M_{f_0}(\varphi_1 \vert \varphi_2) \leq S^\M_\mathrm{rel}(\varphi_1 \vert \varphi_2) + \varphi_1(\1) \eqend{,}
\end{equation}
and a more involved lower bound. These bounds are sharp enough to show that when passing to tensor powers $\M^{\otimes n}$, $\varphi_1^{\otimes n},\varphi_2^{\otimes n}$, one obtains (\thmref{thm:ModularBoundsDf0}):
\begin{equation}
\lim_{n \to \infty} \frac{1}{n} D^{\M^{\otimes n}}_{f_0}\left( \varphi_1^{\otimes n} \vert \varphi_2^{\otimes n} \right) = S^\M_\mathrm{rel}(\varphi_1 \vert \varphi_2) \eqend{.}
\end{equation}
In hypothesis testing, this limit is the multi-shot regularization of $D^\M_{f_0}(\varphi_1 \vert \varphi_2)$ and describes asymptotic decay rates of testing errors (see~\cite{chengliu2026} and references therein).

While this approach works directly for arbitrary von Neumann algebras $\M$, it seems not to yield the desired equality $D^\M_{f_0}(\varphi_1 \vert \varphi_2) = S^\M_\mathrm{rel}(\varphi_1 \vert \varphi_2)$ because we currently do not have an independent proof of additivity of $D_{f_0}$ over tensor products. We therefore develop a second approach to comparing the $f_0$-divergence with relative entropy. This approach passes through several steps and relies on various approximation techniques: In \secref{sec:semifinite} we first consider the special case of a semifinite von Neumann algebra $\M$ (a von Neumann algebra with a normal semifinite tracial weight $\tau$), in which the two positive functionals $\varphi_1,\varphi_2$ to be compared are given by densities $h_1,h_2$ with respect to $\tau$. These densities are unbounded operators in general, and for $\varphi_1 \ll \varphi_2$, the relative entropy takes the form
\begin{equation}
S^\M_\mathrm{rel}(\varphi_1 \vert \varphi_2) = \tau\left( h_1 \ln h_1 - h_1 \ln h_2 \right)
\end{equation}
given in Cor.~\ref{cor:semifinite_relative_entropy}. The identification of this expression with $D^\M_{f_0}$ then relies on establishing the equality of the two operator integrals
\begin{equations}
J(h_1,h_2) &= \int_0^\infty \sqrt{ \frac{s}{h_1 + s} } \, \left[ C_{h_1,h_2}(s) \right]^2 f''\left( C_{h_1,h_2}(s) \right) \sqrt{ \frac{s}{h_1 + s} } \total s \eqend{,} \\
I(h_1, h_2) &= \int_0^\infty (h_2 - t h_1)_+ f''(t) \total t \eqend{,}
\end{equations}
where $C_{h_1,h_2}(s) = (h_1 + s)^{-\half} h_2 (h_1 + s)^{-\half}$. For positive numbers $h_1, h_2 > 0$, the integral equality $I(h_1,h_2) = J(h_1,h_2)$ is a simple consequence of the change of variables theorem. For matrix algebras, a similar identity (with the projection $E^{h_2 - t h_1}_{\mathbb{R}_+}$ on the positive part instead of the positive part $(h_2 - t h_1 )_+$ itself) was first derived in~\cite{chengliu2025}, based on a layer cake representation of the derivative of the operator logarithm. It was later shown to be equivalent to Frenkel's formula for the relative entropy~\cite{chenggourlamiliu2025}.

In our general setup, the idea is that for $f = f_0$ the expectation value $\tau\left( J(h_1,h_2) \right)$ in the tracial weight essentially gives $\tau( h_1 \ln h_1 - h_1 \ln h_2) $, and $\tau\left( I(h_1,h_2) \right)$ essentially gives $D^\M_{f_0}$; for details, see \secref{sec:operatorint}. This approach is quite different from Frenkel's original work, but similar to the one taken by Liu, Hirche, and Cheng in \cite{liuhirchecheng2025} for matrix algebras. In order to prove the equality, we need to approximate the densities by spectrally cut off versions, and $f_0$ by a monotonically increasing sequence of $L^1$-functions. Having proven $D^\M_{f_0} = S^\M_\mathrm{rel}$ for semifinite von Neumann algebras, we extend this result in a final step to general von Neumann algebras by passing through finite subalgebras of a discrete crossed product (Haagerup reduction) in \secref{sec:Df=Srel}, and arrive in \thmref{thm:D=Sgeneral} at $D^\M_{f_0} = S^\M_\mathrm{rel}$ for general~$\M$. The underlying operator integral constructions are collected in \secref{sec:operatorint}.

The entropy formula \eqref{eq:Df0Intro} presents a new perspective on relative entropy in general quantum systems, and is expected to be useful in various fields. Our motivation mostly derives from applications in quantum field theory, where infinite von Neumann algebras are ubiquitous. However, the current work is entirely dedicated to a general analysis of $f$-divergences, and such applications will appear elsewhere.

\paragraph{Notation.} Our notation is standard, but briefly recalled here. We consider an arbitrary von Neumann algebra~$\M$, with closed unit ball $\M_1$, positive cone $\M_+ \subset \M$, projection lattice $\PM$, and write $[0,1]_\M \defeq \{ x \in \M \colon 0 \leq x \leq \1_\M \}$ for the unit order interval in $\M$, the effects, where $\1_\M \in \M$ denotes the identity. The dual of $\M$ is denoted $\M^*$ and the predual of normal functionals $\M_*$. We write $\M\stp \subset \M_*$ for the positive normal functionals, and $\M\sth \supset \M\stp$ for the real Banach space $\M\sth \subset \M_*$ consisting of all normal linear functionals that are selfadjoint (or Hermitian) in the sense that $\varphi = \varphi^*$ with $\varphi^*(x) = \overline{\varphi(x^*)}$, $x \in \M$.

We write $s^\M(\varphi) \defeq \inf\{ p \in \PM \colon \varphi(p) = 1 \}$ for the support projection of a positive functional $\varphi \in \M\stp$, and also use the standard notation $\varphi_1 \perp \varphi_2 \iff s^\M(\varphi_1) s^\M(\varphi_2) = 0$ to denote orthogonality, and $\varphi_1 \ll \varphi_2 \iff s^\M(\varphi_1) \leq s^\M(\varphi_2)$ for absolute continuity/support inclusion. For $\varphi_1,\varphi_2 \in \M\stp$, we will often abbreviate support projections as $s_j \defeq s^\M(\varphi_j)$ when the functionals $\varphi_j$ are clear from the context.

Working in a standard representation of $\M \subset \BH$, we also recall that any $\varphi \in \M\stp$ is represented by a unique vector $\Omega_\varphi \in \Hil$ lying in the natural positive cone, namely $\varphi(x) = \langle \Omega_\varphi,  x \, \Omega_\varphi \rangle$ for $x \in \M$, and in particular $\norm{ \Omega_\varphi }^2 = \varphi(\1)$. In this case, the support projection $s^\M(\varphi)$ can be characterized as the smallest projection in $\M$ such that $s^\M(\varphi) \, \Omega_\varphi = \Omega_\varphi$, which is the projection with range $\overline{ \M' \, \Omega_\varphi }$.

Given a selfadjoint operator $X$ on a Hilbert space, we write $E^X$ for its spectral measure. We furthermore set $\mathbb{R}_+ \defeq [0,\infty)$.

We also adopt the usual interpretation of this setting: The selfadjoint elements $x \in \M$ model the observables of a physical system (quantum if $\M$ is noncommutative), and the normalized elements of $\M\stp$ its states. Projections $p \in \PM$ and effects $x \in [0,1]_\M$ correspond to sharp and unsharp measurements, respectively.

\section{Hockeystick divergences and \texorpdfstring{$f$}{f}-divergences}
\label{sec:hockeystick}

In this section we introduce $f$-divergences for general von Neumann algebras. We begin by recalling some properties of the Jordan decomposition of normal functionals in \secref{sec:hockeystick:jordan}, based on which hockeystick divergences are defined in \secref{sec:hockeystick:hockeystick}. Hockeystick divergences are the basic building block of $f$-divergences, which are defined and discussed in \secref{sec:hockeystick:fdivergence}.

\subsection{Jordan decomposition}
\label{sec:hockeystick:jordan}

The Jordan decomposition theorem is essential for our definition of hockeystick divergences:
\begin{theorem}
\label{theorem:JordanDecomposition}
Let $\M$ be a von Neumann algebra and $\varphi \in \M\sth$ a normal selfadjoint functional. Then there exist unique positive $\varphi_\pm \in \M\stp$ such that
\begin{equation}
\varphi = \varphi_+ - \varphi_- \eqend{,} \qquad \norm{ \varphi } = \norm{ \varphi_+ } + \norm{ \varphi_- } \eqend{.}
\end{equation}
The norm equality $\norm{ \varphi } = \norm{ \varphi_+ } + \norm{ \varphi_- }$ is equivalent to $\varphi_+$ and $\varphi_-$ being orthogonal, namely $s^\M(\varphi_+) s^\M(\varphi_-) = 0$.
\end{theorem}
\begin{proof}
See \cite[Thm.~III~4.2~(ii)]{takesaki2001}.
\end{proof}

The following well-known examples are the simplest explicit realizations of Jordan decompositions.
\begin{example}
\label{example:Jordan}
\leavevmode
\begin{enumerate}
\item \label{item:TypeIJordan} Let $\Hil$ be a complex Hilbert space, $\M = \BH$ the type I factor of all bounded linear operators on $\Hil$, and $\Tr_\Hil$ the Hilbert space trace. Let furthermore $L^1(\BH, \Tr_\Hil)$ be the Banach space of trace-class operators with trace norm $\norm{ h }_1 = \Tr_\Hil \abs{h}$. The map
\begin{equation}
\label{eq:TypeIIsomorphism}
L^1(\BH, \Tr_\Hil) \to \BH_* \eqend{,} \quad h \mapsto \varphi_h \eqend{,} \quad \varphi_h(x) \defeq \Tr_\Hil(h x) \eqend{,}
\end{equation}
is an isometric isomorphism \cite[Prop.~2.4.3]{brattelirobinson1}. Through this isomorphism, $\varphi_h^* = \varphi_{h^*}$, so selfadjoint functionals correspond to selfadjoint trace class operators. The Jordan decomposition is carried into the decomposition of $h = h^*$ into its positive and negative parts $h_\pm \defeq \left( \abs{h} \pm h \right)/2 \geq 0$. In particular, $\norm{ (\varphi_h)_+ } = \Tr_\Hil(h_+)$.

\item \label{item:AbelianJordan} Let $(X,\mu)$ be a $\sigma$-finite measure space, and consider the Abelian von Neumann algebra $\M = L^\infty(X,\mu)$, faithfully represented by multiplication operators on $\Hil = L^2(X,\mu)$. The map
\begin{equation}
\label{eq:AbelianIsomorphism}
L^1(X,\mu) \to L^\infty(X,\mu)_* \eqend{,} \quad h \mapsto \varphi_h \eqend{,} \quad \varphi_h(f) \defeq \int_X f h \total \mu \eqend{,}
\end{equation}
is an isometric isomorphism \cite[Thm.~6.16]{rudin1987}. Through this isomorphism, selfadjoint functionals correspond to real densities $h \in L^1(X,\mu)$, and the Jordan decomposition is carried into the decomposition of a real $L^1$-function $h$ into its positive and negative parts $h_\pm \defeq \left( \abs{h} \pm h \right)/2 \geq 0$. In particular, the variational norm of the positive part of the signed measure $h \mu$, the \emph{positive variation}, is given by $\norm{ (\varphi_h)_+ } = \int_X h_+ \total \mu$.

\item \label{item:TypeIIJordan} Let $\M$ be a von Neumann algebra and $\tau \in \M\stp$ a faithful normal finite trace, i.e., $\tau(x^* x) = 0 \iff x = 0$ and $\tau(x y) = \tau(y x)$ for all $x, y \in \M$. Then $\norm{ x }_1 \defeq \tau( \abs{x} )$ is a norm on $\M$. We denote the completion of $\M$ in this norm by $L^1(\M,\tau)$ and note that~$\M$ acts norm-continuously on $L^1(\M,\tau)$ by multiplication from the left and right. The adjoint map $\M \ni h \mapsto h^* \in \M$ and the trace $\M \in h \mapsto \tau(h) \in \mathbb{C}$ extend continuously to $L^1(\M,\tau)$; the extensions are denoted by the same symbols. The map
\begin{equation}
L^1(M,\tau) \to \M_* \eqend{,} \quad h \mapsto \varphi_h \eqend{,} \quad \varphi_h(x) \defeq \tau(h x) \eqend{,}
\end{equation}
is an isometric isomorphism \cite[Sec.~I.6.10]{dixmier1981}. Through this isomorphism, $\varphi_h^* = \varphi_{h^*}$, so selfadjoint functionals correspond to selfadjoint densities. To describe the Jordan decomposition, we here restrict ourselves to densities $h \in \M \subset L^1(\M,\tau)$. In this case, it is clear that the Jordan decomposition of $\varphi_h$ is given by the decomposition of $h$ into its positive and negative parts $h_\pm \defeq \left( \abs{h} \pm h \right) \geq 0$. In particular, $\norm{ (\varphi_h)_+ } = \tau(h_+)$. The case of general densities works similarly and will be recalled in \secref{sec:semifinite}.
\end{enumerate}
\end{example}

All of these examples fit into the general setting of semifinite von Neumann algebras that will be recalled in \secref{sec:semifinite}. For now, we collect several properties of the generalization of the different maps $\varphi_h \mapsto \Tr_\Hil( h_+ )$, $\varphi_h \mapsto \int_X h_+ \total \mu$, and $\varphi_h \mapsto \tau(h_+)$ to arbitrary von Neumann algebras, generalizing in particular the positive variation of a signed measure in \exref[item:AbelianJordan]{example:Jordan}. We formulate these facts as corollaries because they are direct consequences of the Jordan decomposition theorem.
\begin{corollary}
\label{cor:Jordan}
Let $\M$ be a von Neumann algebra, and define the \emph{positive variation}
\begin{equation}
\label{eq:P}
P^\M \colon \M\sth \to \mathbb{R}_+ \eqend{,} \quad \varphi \mapsto \norm{ \varphi_+ } \eqend{.}
\end{equation}
Then it holds that
\begin{equation}
\label{eq:varphiSup}
P^\M(\varphi) = \varphi_+(\1) = \varphi(s_+) = \frac{1}{2} \Bigl( \varphi(\1) + \norm{ \varphi } \Bigr) = \sup_{x \in [0,1]_\M} \varphi(x) = \sup_{p \in \PM} \varphi(p) \eqend{,}
\end{equation}
where $s_+ \defeq s^\M(\varphi_+)$ is the support projection of $\varphi_+$. Furthermore, for $\varphi, \psi \in \M\sth$ the positive variation $P^\M$ has the following properties:
\begin{enumerate}
\item \label{item:PContractive} \emph{(Continuity)} $\norm{ \varphi_+ } \leq \norm{ \varphi }$, and $P^\M$ is continuous in the norm topology of $\M^*$. It is also lower semicontinuous in the $\sigma(\M_*,\M)$-topology on $\M\sth$.
\item \label{item:PMonotone} \emph{(Monotonicity)} $\varphi \leq \psi \implies \norm{ \varphi_+ } \leq \norm{ \psi_+ }$.
\item \label{item:PSubadditive} \emph{(Subadditivity)} $\norm{ (\varphi + \psi)_+ } \leq \norm{ \varphi_+ } + \norm{ \psi_+ }$.
\item \label{item:PHomogeneous} \emph{(Positive homogeneity)} $\norm{ (\lambda \varphi)_+ } = \lambda \norm{ \varphi_+ }$ for $\lambda \geq 0$.
\item \label{item:PConvex} \emph{(Convexity)} $\norm{ (\lambda \varphi + (1-\lambda) \psi)_+ } \leq \lambda \norm{ \varphi_+ } + (1-\lambda) \norm{ \psi_+ }$ for $\lambda \in [0,1]$.
\end{enumerate}
\end{corollary}
\begin{proof}
We begin with proving the five representations of $P^\M$ claimed in \eqref{eq:varphiSup}. The first equality in \eqref{eq:varphiSup}, $\norm{ \varphi_+ } = \varphi_+(\1)$, holds by positivity of $\varphi_+$. The second equality, $\varphi_+(\1) = \varphi(s_+)$, follows by orthogonality of the support projections $s_\pm$ of $\varphi_\pm$, namely $\varphi(s_+) = \varphi_+(s_+) - \varphi_-(s_+) = \varphi_+(\1) - \varphi_-(s_- s_+ s_-) = \varphi_+(\1)$.

Since $\varphi(\1) = \varphi_+(\1) - \varphi_-(\1) = \norm{ \varphi_+ } - \norm{ \varphi_- }$ and $\norm{ \varphi } = \norm{ \varphi_+ } + \norm{ \varphi_- }$, the third equality follows. To prove the fourth equality, the variational formula $\varphi_+(\1) = \sup_{x \in[0,1]_\M} \varphi(x)$, note that for $x \in [0,1]_\M$ we have $\varphi(x) = \varphi_+(x) - \varphi_-(x) \leq \varphi_+(x) \leq \varphi_+(\1)$, hence $\varphi_+(\1) \geq \sup_{x \in [0,1]_\M}\varphi(x)$. The opposite inequality follows by noting $s_+ \in [0,1]_\M$ and $\varphi(s_+) = \varphi_+(\1)$. The last equality is now clear because $s_+ \in \PM \subset [0,1]_\M$.

The properties \ref{item:PContractive}--\ref{item:PConvex} are also easy to prove: The positive variation $P^\M$ is contractive because $\norm{ \varphi } = \norm{ \varphi_+ } + \norm{ \varphi_- } \geq \norm{ \varphi_+ }$. In $P^\M(\varphi) = \left( \varphi(\1) + \norm{ \varphi } \right)/2$, the map $\varphi \mapsto \varphi(\1)$ is clearly $\sigma(\M_*,\M)$-continuous, and $\varphi \mapsto \norm{ \varphi }$ is $\sigma(\M_*,\M)$-lower semicontinuous as the supremum over the continuous functions $\varphi \mapsto \varphi(x)$, $x \in [0,1]_\M$. Hence $P^\M$ is $\sigma(\M_*,\M)$-lower semicontinuous.

Let $\varphi \leq \psi$. Then
\begin{equation}
P^\M(\varphi) = \sup_{x \in [0,1]_\M}\varphi(x) \leq \sup_{x \in [0,1]_\M} \psi(x) = P^\M(\psi) \eqend{,}
\end{equation}
which shows the monotonicity of $P^\M$. Similarly, for $\varphi,\psi \in \M\sth$ and $\lambda \in \mathbb{R}_+$,
\begin{equations}
\begin{split}
\norm{ (\varphi + \psi)_+ } &= \frac{1}{2} \Bigl( \varphi(\1) + \psi(\1) + \norm{ \varphi + \psi } \Bigr) \\
&\leq \frac{1}{2} \Bigl( \varphi(\1) + \psi(\1) + \norm{ \varphi } + \norm{ \psi } \Bigr) = \norm{ \varphi_+ } + \norm{ \psi_+ } \eqend{,}
\end{split} \\
\norm{ (\lambda \varphi)_+ } &= \frac{1}{2} \Bigl( \lambda \varphi(\1) + \norm{ \lambda \varphi } \Bigr) = \lambda \norm{ \varphi_+ } \eqend{,}
\end{equations}
which shows that $P^\M$ is subadditive and positively homogeneous. Convexity is a consequence of these properties.
\end{proof}

We also need to establish how the Jordan decomposition behaves on direct sums and tensor products of von Neumann algebras ($\N \otimes \M$ will denote the von Neumann algebraic tensor product), and how it behaves with respect to maps between two different von Neumann algebras. These properties are collected next:
\begin{corollary}
\label{cor:Jordan2}
Let $\M$, $\N$ be von Neumann algebras, and $\varphi \in \M\sth$, $\psi \in \N\sth$. Then the positive variation \eqref{eq:P} has the following properties:
\begin{enumerate}
\item \label{item:DirectSums} \emph{(Direct sums)} $P^{\M \oplus \N} = P^\M \oplus P^\N$, i.e., $\varphi \oplus \psi \in (\M \oplus \N)\sth$, and $\norm{ (\varphi \oplus \psi)_+ } = \norm{ \varphi_+ } + \norm{ \psi_+ }$.
\item \label{item:TensorProducts} \emph{(Tensor products)} $\varphi \otimes \psi \in (\M \otimes \N)\sth$, and $\norm{ (\varphi \otimes \psi)_+ } = \norm{ \varphi_+ \otimes \psi_+ + \varphi_- \otimes \psi_- }$. In particular, if $\psi$ is positive, then $\norm{ (\varphi \otimes \psi)_+ } = \psi(\1) \norm{ \varphi_+ }$.
\item \label{item:PositiveContractions} \emph{(Positive contractions)} Let $T \colon \N \to \M$ be a normal positive linear contraction and $T^* \colon \M_* \to \N_*$, $T^* \varphi \defeq \varphi \circ T$ its pullback. Then $P^\N \circ T^* \leq P^\M$, i.e.,
\begin{equation}
\norm{ (\varphi \circ T)_+ } \leq \norm{ \varphi_+ } \eqend{.}
\end{equation}
\item \label{item:CondExp} \emph{(Conditional expectations)} Let $\M \subset \N$, and let $C \colon \N \to \M$ be a normal conditional expectation. Then $P^\N \circ C^* = P^\M$, i.e.,
\begin{equation}
\norm{ (\varphi \circ C)_+ } = \norm{ \varphi_+ } \eqend{.}
\end{equation}
\item \label{item:Compressions} \emph{(Compressions)} Let $\N = p \M p$ for a projection $p \in \PM$. Then $P^\N(\varphi) = P^\M(\varphi)$ for all $\varphi$ with $s^\M(\varphi) \leq p$.
\item \label{item:Martingale} \emph{(Martingale convergence)} Let $(\M_i)_{i \in I}$ with $M_i \subset \M$ for all $i \in I$ be an increasing net of von Neumann algebras such that $\bigcup_{i \in I} \M_i$ is $\sigma(\M, \M_*)$-dense in $\M$. Then it holds that
\begin{equation}
\lim_{i \in I} P^{\M_i}\left( \varphi \rvert_{\M_i} \right) = P^\M(\varphi) \eqend{.}
\end{equation}
\end{enumerate}
\end{corollary}
\begin{proof}
\ref{item:DirectSums} $\varphi \oplus \psi \in (\M \oplus \N)\sth$ is clear, and $\norm{ (\varphi \oplus \psi)_+ } = \norm{ \varphi_+ } + \norm{ \psi_+ }$ follows from the third equality in \eqref{eq:varphiSup} because the identity in $\M \oplus \N$ and the norm on $(\M \oplus \N)^*$ split over direct sums.

\ref{item:TensorProducts} Tensor products of normal hermitian functionals are normal and hermitian. For the claimed formula, we write
\begin{equation}
\label{eq:TensorProductJordan}
\varphi \otimes \psi = ( \varphi_+ \otimes \psi_+ + \varphi_- \otimes \psi_- ) - ( \varphi_+ \otimes \psi_- + \varphi_- \otimes \psi_+ ) \eqend{,}
\end{equation}
and note that the two terms in brackets are normal positive functionals on $\M \otimes \N$. Regarding their norms, we have $\norm{ \varphi \otimes \psi } = \norm{ \varphi } \norm{ \psi } = \left( \norm{ \varphi_+ } + \norm{ \varphi_- } \right) \left( \norm{ \psi_+ } + \norm{ \psi_- } \right)$, and by positivity
\begin{equations}
\norm{ \varphi_+ \otimes \psi_+ + \varphi_- \otimes \psi_- } &= \left( \varphi_+ \otimes \psi_+ + \varphi_- \otimes \psi_- \right)(\1_\M \otimes \1_\N) = \norm{ \varphi_+ } \norm{ \psi_+ } + \norm{ \varphi_- } \norm{ \psi_- } \eqend{,} \\
\norm{ \varphi_+ \otimes \psi_- + \varphi_- \otimes \psi_+ } &= \left( \varphi_+ \otimes \psi_- + \varphi_- \otimes \psi_+ \right)(\1_\M \otimes \1_\N) = \norm{ \varphi_+ } \norm{ \psi_- } + \norm{ \varphi_- } \norm{ \psi_+ } \eqend{.}
\end{equations}
This shows that the sum of the norm of the terms in brackets equals the norm of $\norm{ \varphi \otimes \psi }$, which shows that \eqref{eq:TensorProductJordan} is the Jordan decomposition of $\varphi \otimes \psi$. This implies the claim.

\ref{item:PositiveContractions} $\varphi \circ T \in \N\sth$ by normality and positivity of $T$. Recall that a positive contraction $T \colon \M \to \N$ satisfies $T(\1_\N) \leq \1_\M$. This implies that $T$ maps $[0,1]_\N$ into $[0,1]_\M$, and consequently
\begin{equation}
\norm{ (\varphi \circ T)_+ } = \sup_{x \in [0,1]_\N} \varphi(T(x)) \leq \sup_{y \in [0,1]_\M} \varphi(y) = \norm{ \varphi_+ } \eqend{.}
\end{equation}

\ref{item:CondExp} As a normal conditional expectation is in particular a normal positive linear contraction, we have $\norm{ (\varphi \circ C)_+ } \leq \norm{ \varphi_+ }$ from the previous item. On the other hand, $[0,1]_\M$ is pointwise invariant under $C$, and contained in $[0,1]_\N$. Hence
\begin{equation}
\norm{ \varphi_+ } = \sup_{x \in [0,1]_\M} \varphi(x) = \sup_{x \in [0,1]_\M} \varphi(C(x)) \leq \sup_{x \in [0,1]_\N} \varphi(C(x)) = \norm{ (\varphi \circ C)_+ } \eqend{.}
\end{equation}

\ref{item:Compressions} is obvious.

\ref{item:Martingale} We use the third equality in \eqref{eq:varphiSup}: $P^{\M_i}\left( \varphi \rvert_{\M_i} \right) = \left( \varphi\left( \1_{\M_i} \right) + \norm{ \varphi \rvert_{\M_i} } \right)/2$. By Vigier's theorem~\cite[App.~II]{dixmier1981}, we have $\lim_{i \in I} \1_{\M_i} = \1_\M$ in the strong operator topology. Since the net is bounded, this limit also holds in $\sigma(\M, \M_*)$-topology. Hence $\lim_{i \in I} \varphi\left( \1_{\M_i} \right) = \varphi(\1_\M)$, and with the same arguments $\lim_{i \in I} \norm{ \varphi \rvert_{\M_i} } = \sup_{x \in [0,1]_{\M_i}} \abs{ \varphi(x) } = \norm{ \varphi }$.
\end{proof}
Note that in item~\ref{item:PositiveContractions}, we may in particular consider inclusion maps $\iota \colon \N \hookrightarrow \M$. In case $T \in \operatorname{Aut}(\M)$ is an automorphism, \ref{item:PositiveContractions} of course implies $P^\M(\varphi \circ T) = P^\M(\varphi)$.

We remark that the Jordan decomposition also exists for continuous selfadjoint functionals on a general $C^*$-algebra \cite[Prop.~III~2.1]{takesaki2001}. By passing to its bidual, the properties listed in \corref{cor:Jordan} and~\ref{cor:Jordan2} carry over to $C^*$-algebras. As we will not need this version in this article, we refrain from giving the details.

\subsection{Hockeystick divergences}
\label{sec:hockeystick:hockeystick}

We now define hockeystick divergences as described in the introduction.
\begin{definition}
\label{def:hockeystick}
Let $\M$ be a von Neumann algebra. The \emph{hockeystick divergence} is the map
\begin{splitequation}
\label{eq:HockeystickFormula}
&E^\M \colon \mathbb{R}_+ \times \M\stp \times \M\stp \to \mathbb{R}_+ \eqend{,} \\
&E^\M(t, \varphi_1, \varphi_2) \defeq E^\M_t(\varphi_1 \vert \varphi_2) \defeq P^\M(\varphi_1 - t \varphi_2) = \norm{ (\varphi_1 - t \varphi_2)_+ } \eqend{.}
\end{splitequation}
\end{definition}

\begin{example}
\leavevmode
\begin{enumerate}
\item Let $\M = \BH$ be a type I factor on a Hilbert space $\Hil$ as in \exref[item:TypeIJordan]{example:Jordan}. Then we may identify $\varphi_1, \varphi_2 \in \M\stp$ with their positive densities $0 \leq h_1, h_2 \in L^1(\BH, \Tr_\Hil)$ via \eqref{eq:TypeIIsomorphism}, and the hockeystick divergence becomes
\begin{equation}
E^\BH_t(\varphi_1 \vert \varphi_2) = \Tr_\Hil\left[ (h_1 - t h_2)_+ \right] \eqend{.}
\end{equation}

\item Let $\M = L^\infty(X,\mu)$ be an Abelian von Neumann algebra on a $\sigma$-finite measure space $(X,\mu)$ as in \exref[item:AbelianJordan]{example:Jordan}. Then we may identify $\varphi_1, \varphi_2 \in \M\stp$ with their positive densities $0 \leq h_1, h_2 \in L^1(X,\mu)$ via \eqref{eq:AbelianIsomorphism}, and the hockeystick divergence becomes
\begin{equation}
E^{L^\infty(X,\mu)}_t(\varphi_1 \vert \varphi_2) = \int_X \left( h_1 - t h_2 \right)_+(x) \total \mu(x) \eqend{.}
\end{equation}
\end{enumerate}
\end{example}
The first example includes in particular the case of the finite-dimensional matrix algebras $\M = M_n(\mathbb{C})$. We now summarize the main properties of hockeystick divergences that relate to a single von Neumann algebra, and in particular establish the link to the formula \eqref{eq:SuccessProbability} motivated from hypothesis testing.
\begin{proposition}
\label{prop:hockeystick_properties}
The hockeystick divergence of a von Neumann algebra $\M$ can be expressed as
\begin{equations}[eq:EFormulae]
E^\M_t(\varphi_1 \vert \varphi_2) &= \left( \varphi_1 - t \varphi_2 \right)_+(\1) \\
&= \frac{1}{2} \Bigl( \varphi_1(\1) - t \varphi_2(\1) + \norm{ \varphi_1 - t \varphi_2 } \Bigr) \label{eq:EFormulaDiffNorm} \\
&= \sup_{x \in [0,1]_\M} \Bigl[ \varphi_1(x) - t \varphi_2(x) \Bigr] \\
&= \sup_{p \in \PM} \Bigl[ \varphi_1(p) - t \varphi_2(p) \Bigr] \label{eq:SuccessProbability}
\end{equations}
and for $t \geq 0$ and $\varphi_1,\varphi_2,\varphi_3,\psi_1,\psi_2 \in \M\stp$ satisfies the following properties:
\begin{enumerate}
\item \label{item:EContinuity} \emph{(Continuity)} $E^\M$ is jointly continuous when $\M\stp$ is given the norm topology of $\M^*$, and jointly lower semicontinuous when $\M\stp$ is given the $\sigma(\M_*,\M)$-topology.
\item \label{item:EMonotonicity} \emph{(Monotonicity)} $E^\M$ is monotone in the following sense:
\begin{equation}
\varphi_1 \geq \psi_1 \eqend{,} \ t_1 \leq t_2 \eqend{,} \ \varphi_2 \leq \psi_2 \implies E^\M_{t_1}(\varphi_1 \vert \varphi_2) \geq E^\M_{t_2}(\psi_1 \vert \psi_2) \eqend{.}
\end{equation}
\item \label{item:EBoundsLimit} \emph{(Bounds and limit)} Let $s_2 \in \PM$ be the support projection of $\varphi_2$. Then
\begin{equation}
0 \leq \max\Bigl( \varphi_1(\1) - t \varphi_2(\1) \eqend{,} \ \varphi_1(s_2^\perp) \Bigr) \leq E^\M_t(\varphi_1 \vert \varphi_2) \leq \varphi_1(\1)
\end{equation}
and
\begin{equation}
\lim_{t \to \infty} E^\M_t(\varphi_1 \vert \varphi_2) = \varphi_1(s_2^\perp) \eqend{.}
\end{equation}
\item \label{item:EConvexity} \emph{(Convexity)} For fixed $t \geq 0$, the hockeystick divergence is jointly convex in the functionals, namely
\begin{equation}
E^\M_t\left( \lambda \varphi_1 + (1-\lambda) \varphi_2 \vert \lambda \psi_1 + (1-\lambda) \psi_2 \right) \leq \lambda E^\M_t(\varphi_1 \vert \psi_1) + (1-\lambda) E^\M_t(\varphi_2 \vert \psi_2)
\end{equation}
for all $\lambda \in [0,1]$, and for fixed $\varphi_1,\varphi_2$, it is convex in the parameter, namely
\begin{align}
E^\M_{\lambda t_1 + (1-\lambda) t_2}(\varphi_1 \vert \varphi_2) \leq \lambda E^\M_{t_1}(\varphi_1 \vert \varphi_2) + (1-\lambda) E^\M_{t_2}(\varphi_1 \vert \varphi_2)
\end{align}
for all $\lambda \in [0,1]$.
\item \label{item:ETriangleInequality} \emph{(``Triangle inequality'')} For $t_1, t_2 \geq 0$, we have
\begin{align}
E^\M_{t_1 t_2}(\varphi_1 \vert \varphi_3) \leq E^\M_{t_1}(\varphi_1 \vert \varphi_2) + t_1 E^\M_{t_2}(\varphi_2 \vert \varphi_3) \eqend{.}
\end{align}
\item \label{item:EStateExchange} \emph{(State exchange and rescaling)}
\begin{equations}
E^\M_t(\varphi_2 \vert \varphi_1) &= \varphi_2(\1) - t \varphi_1(\1) + t \, E^\M_{1/t}(\varphi_1 \vert\varphi_2) \eqend{,} \quad t > 0 \eqend{,} \\
E^\M_t(c_1 \varphi_1 \vert c_2 \varphi_2) &= c_1 E^\M_{t c_2 / c_1}(\varphi_1 \vert \varphi_2) \eqend{,} \quad c_1, c_2 > 0 \eqend{.}
\end{equations}
\end{enumerate}
\end{proposition}
\begin{proof}
The expressions \eqref{eq:EFormulae} for $E^\M$ follow directly from \eqref{eq:varphiSup}.

\ref{item:EContinuity} The continuity claims are consequences of the continuity of $P^\M$ from \corref[item:PContractive]{cor:Jordan} and of the vector space operations on $\M\sth$.

\ref{item:EMonotonicity} Monotonicity follows directly from \corref[item:PMonotone]{cor:Jordan}.

\ref{item:EBoundsLimit} Expressing the hockeystick divergence as the supremum over all $p \in \PM$, the lower bound follows from choosing the projections $p = \1$ and $p = s_2^\perp$. The upper bound follows from monotonicity: $E^\M_t(\varphi_1 \vert \varphi_2) \leq E^\M_0(\varphi_1 \vert \varphi_2) = \varphi_1(\1)$.

The limit $l \defeq \lim\limits_{t \to \infty} E^\M_t(\varphi_1 \vert \varphi_2)$ exists and satisfies $l \geq \varphi_1(s_2^\perp)$ due to monotonicity and lower boundedness. For the converse inequality, we consider $t = n \in \mathbb{N}$ and choose $x_n \in [0,1]_\M$ such that $\varphi_1(x_n) - n \varphi_2(x_n) \geq E^\M_n(\varphi_1 \vert \varphi_2) - \frac{1}{n}$, i.e.,
\begin{equation}
\label{eq:ELimitEstimate}
0 \leq n \varphi_2(x_n) \leq \varphi_1(x_n) + \frac{1}{n} - E^\M_n(\varphi_1 \vert \varphi_2) \eqend{.}
\end{equation}
Since $[0,1]_\M$ is ultraweakly compact and the functionals are ultraweakly continuous, we may pass to a convergent subsequence $x_{n_j} \to x \in [0,1]_\M$ such that the right-hand side of \eqref{eq:ELimitEstimate} converges to $\varphi_1(x) - l$. It follows that $\varphi_2(x) = 0$, i.e., the limit satisfies $l \leq \varphi_1(x)$, and $0 = \varphi_2(x) = \varphi_2(s_2 x s_2)$, which implies $x s_2 = s_2 x = 0$. Hence $x = s_2^\perp x \, s_2^\perp \leq s_2^\perp$, and the reverse inequality $l \leq \varphi_1(s_2^\perp)$ follows.

\ref{item:EConvexity} Both convexity properties follow directly from \corref[item:PConvex]{cor:Jordan}.

\ref{item:ETriangleInequality} The ``triangle inequality'' follows by evaluating $\varphi_1 - t_1 t_2 \varphi_3 = (\varphi_1 - t_1 \varphi_2) + t_1 (\varphi_2 - t_2 \varphi_3)$ in~$P^\M$ and use its subadditivity and positive homogeneity from \corref{cor:Jordan}.

\ref{item:EStateExchange} State exchange and rescaling properties follow similarly from
\begin{splitequation}
\left( \varphi_2 - t \varphi_1 \right)_+ &= \varphi_2 - t \varphi_1 + \left( \varphi_2 - t \varphi_1 \right)_- = \varphi_2 - t \varphi_1 + t \left( t^{-1} \varphi_2 - \varphi_1 \right)_- \\
&= \varphi_2 - t \varphi_1 + t\left( \varphi_1 - t^{-1} \varphi_2 \right)_+ \eqend{,}
\end{splitequation}
and
\begin{equation}
\left( c_1 \varphi_1 - t c_2 \varphi_2 \right)_+ = c_1 \left( \varphi_1 - \frac{t c_2}{c_1} \varphi_2 \right)_+ \eqend{.}
\end{equation}
\end{proof}
We note that the proof of the ``triangle inequality'' \ref{item:ETriangleInequality} is analogous to the one for matrix algebras, \cite[Prop.~II.5]{hircherouzefranca2023}. We will obtain more specific upper and lower bounds on hockeystick divergences with the help of relative modular operators in \secref{sec:ModularBounds}.

As the main motivation underlying hockeystick divergences is the discrimination of states, we also characterize some special state configurations in terms of their hockeystick divergences.
\begin{corollary}
\label{cor:EStateComparison}
Let $\M$ be a von Neumann algebra and $\varphi_1, \varphi_2 \in \M\stp$.
\begin{enumerate}
\item \label{item:orthogonality} \emph{(Orthogonality)} $\varphi_1 \perp \varphi_2 \iff E^\M_t(\varphi_1 \vert \varphi_2) = \varphi_1(\1)$ for all $t \geq 0$ $\iff E^\M_{t_0}(\varphi_1 \vert \varphi_2) = \varphi_1(\1)$ for some $t_0 > 0$.

\item \label{item:phi12} \emph{(Absolute continuity)} $\varphi_1 \ll \varphi_2 \iff \lim\limits_{t \to \infty} E^\M_t(\varphi_1 \vert \varphi_2) = 0$.

\item \label{item:phi121} \emph{(Equal supports)} $s^\M(\varphi_1) = s^\M(\varphi_2) \iff \lim\limits_{t \to \infty} E^\M_t(\varphi_1 \vert \varphi_2) = \lim\limits_{t \to \infty} E^\M_t(\varphi_2 \vert \varphi_1) = 0$.

\item \label{item:multiples} \emph{(Multiples)} $\varphi_2 = c \, \varphi_1$ for some $c \geq 0$ $\iff E^\M_t(\varphi_1 \vert \varphi_2) = \Bigl( \varphi_1(\1) - t \varphi_2(\1) \Bigr)_+$ for all $t \geq 0$.

\item \label{item:equality} \emph{(Equality)} $\varphi_1 = \varphi_2 \iff \varphi_1(\1) = \varphi_2(\1)$ and $E^\M_t(\varphi_1 \vert \varphi_2) = \varphi_1(\1) (1-t)_+$ for all $t \geq 0$.

\item \label{item:submultiples} \emph{(Domination)} $\varphi_1 \leq c \, \varphi_2$ for some $c > 0$ $\iff E^\M_t(\varphi_1 \vert \varphi_2) = 0$ for all $t \geq c$. Equivalently, $E^\M_t(\varphi_1 \vert \varphi_2) = 0$ for $t \geq \mathe^{D^\M_\mathrm{max}(\varphi_1 \vert \varphi_2)}$, where $D^\M_\mathrm{max}(\varphi_1 \vert \varphi_2) = \inf \{ \lambda \in \mathbb{R} \colon \varphi_1 \leq \mathe^\lambda \varphi_2 \} \in [-\infty,\infty]$ is the max-relative entropy of $\varphi_1$ and $\varphi_2$.
\end{enumerate}
\end{corollary}
\begin{proof}
\ref{item:orthogonality} The orthogonality $\varphi_1 \perp \varphi_2 \iff \varphi_1 \perp t \varphi_2$ for $t > 0$ is equivalent to the Jordan decomposition $(\varphi_1 - t \varphi_2)_+ = \varphi_1$, and this implies $E^\M_t(\varphi_1 \vert \varphi_2) = \varphi_1(\1)$. If $E^\M_{t_0}(\varphi_1 \vert \varphi_2) = \varphi_1(\1)$ for some $t_0 > 0$, then $\psi \defeq \varphi_1 - t_0 \varphi_2$ satisfies
\begin{equation}
\norm{ \varphi_1 } = \norm{ \psi_+ } = \varphi_1\left( s^\M(\psi_+) \right) - t_0 \varphi_2\left( s^\M(\psi_+) \right) \leq \norm{ \varphi_1 } - t_0 \varphi_2\left( s^\M(\psi_+) \right) \eqend{.}
\end{equation}
This implies $\varphi_2\left( s^\M(\psi_+) \right) = 0$ and thus $\varphi_1\left( s^\M(\psi_+) \right) = \varphi_1(\1)$, hence $\varphi_1 \perp \varphi_2$.

\ref{item:phi12} We have $\varphi_1 \ll \varphi_2 \iff s_1 \leq s_2 \iff 0 = \varphi_1\left( s_2^\perp \right) = \lim\limits_{t \to \infty} E^\M_t(\varphi_1 \vert \varphi_2)$. Requiring also $\varphi_2 \ll \varphi_1$ gives \ref{item:phi121}.

\ref{item:multiples} The case $\varphi_1 = 0$ (which also implies $\varphi_2 = 0$) is trivial, as is the case $\varphi_1 \neq 0$ and $c = 0$. If $\varphi_2 = c \, \varphi_1 \neq 0$ for some $c > 0$, then $c = \varphi_2(\1) / \varphi_1(\1)$, and the equality
\begin{equation}
E^\M_t(\varphi_1 \vert \varphi_2) = \Bigl( \varphi(\1) - t \varphi_2(\1) \Bigr)_+
\end{equation}
follows directly. Conversely, if the equality holds for all $t$, it holds in particular for $t = \varphi_1(\1) / \varphi_2(\1)$. Then the right hand side vanishes, hence we obtain
\begin{equation}
0 = \left( \varphi_1 - \frac{\varphi_1(\1)}{\varphi_2(\1)} \varphi_2 \right)_+(\1) \implies \norm{ \varphi_1 - \frac{\varphi_1(\1)}{\varphi_2(\1)} \varphi_2 } = 0 \eqend{,}
\end{equation}
with the implication following from \eqref{eq:EFormulaDiffNorm}. \ref{item:equality} follows by specializing to $c = 1$.

\ref{item:submultiples} Domination follows directly from the definition of $E^\M_t(\varphi_1 \vert \varphi_2)$.
\end{proof}

The limit $\lim\limits_{t \to \infty} E^\M_t(\varphi_1 \vert \varphi_2) = \varphi_1(s_2^\perp)$ obtained in \propref[item:EBoundsLimit]{prop:hockeystick_properties} can be interpreted as the maximal probability of an event occuring in the state $\varphi_1$ that has probability zero for occuring in the state $\varphi_2$. For example, consider two vector states on $\BH$, given by unit vectors $\Omega_1$ and $\Omega_2$. Then $\varphi_1(s_2^\perp) = 1 - \abs{ \langle \Omega_1, \Omega_2 \rangle}^2$, the complement of the transition probability.

The following example illustrates the hockeystick divergence in simple situations with modular flows with non-trivial fixed points, and shows that the rate at which the $t \to \infty$ limit is attained can be arbitrary.
\begin{example}
Let $\M \subset \BH$ with cyclic separating unit vector $\Omega_1$, and let $h \geq 0$ be a positive operator\footnote{For example, $h$ could be the range projection of an isometry $V$ that behaves like $\sigma_t(V) = \lambda^{\mathi t} V$ under the modular flow $\sigma$ of $(\M,\Omega_1)$, for some $0 < \lambda \leq 1$.} affiliated with the centralizer $\M_{\omega_1}$ of the vector state $\omega_1$ given by $\Omega_1$, such that $\omega_1(h) = 1$. Then $\omega_2(x) \defeq \omega_1(h x) = \omega_1\left( h^\half x h^\half \right)$ is a normal state of $\M$, and the Jordan decomposition of $\varphi \defeq \omega_1 - t \omega_2$ is given by $\varphi_\pm(x) = \omega_1\left( (1-th)_\pm x \right)$. Denoting by $\mu(\cdot) \defeq \langle \Omega_1, E^h(\cdot) \Omega_1 \rangle$ the spectral measure $E^h$ of $h$ in $\Omega_1$, it follows that
\begin{splitequation}
\label{eq:IntegralELimit1}
E^\M_t(\omega_1 \vert \omega_2) &= \omega_1\left( (1-th)_+ \right) = \int_0^\infty (1 - \lambda t)_+ \total \mu(\lambda) = \int_0^\infty \int_0^1 \chi_{[\lambda,\infty)}\left( \tfrac{s}{t} \right) \total s \total \mu(\lambda) \\
&= \int_0^1 \mu\left( \left[ 0, \tfrac{s}{t} \right] \right) \total s \ \xrightarrow{t \to \infty} \ \mu\left( \{0\} \right)
\end{splitequation}
and analogously
\begin{equation}
\label{eq:IntegralELimit2}
E^\M_t(\omega_2 \vert \omega_1) = \omega_1\left( (h-t)_+ \right) = \int_t^\infty \mu\left( (s,\infty) \right) \total s \ \xrightarrow{t \to \infty} \ 0 \eqend{.}
\end{equation}
\end{example}
The hockeystick divergence limits computed in this example are attained for finite large enough $t$ for measures with spectral gap at $0$ and $\infty$, respectively (this is of course always the case in finite dimensions). In general, they can however be attained at arbitrary rates depending on the weight distribution of the measure $\mu$.

As we will in particular be interested in the parameter dependence of the hockeystick divergence, i.e., in the functions
\begin{equation}
\label{eq:Et}
\mathbb{R}_+ \ni t \mapsto E^\M_t(\varphi_1 \vert \varphi_2) \in \mathbb{R}_+
\end{equation}
for fixed $\varphi_1, \varphi_2 \in \M\stp$, we emphasize that these functions are continuous, monotonically decreasing, non-negative, bounded by $\varphi_1(s_2^\perp) \leq E^\M_t(\varphi_1 \vert \varphi_2) \leq \varphi_1(\1) = \norm{ \varphi_1 }$, and convex, hence almost everywhere twice differentiable with non-negative second derivative. In addition, \eqref{eq:Et} is Lipschitz continuous with Lipschitz constant $\varphi_2(\1) = \norm{ \varphi_2 }$. To prove this, let $t' \geq t$. Using the subadditivity and positive homogeneity of $P^\M$, we obtain
\begin{splitequation}
\abs{ E^\M_t(\varphi_1 \vert \varphi_2) - E^\M_{t'}(\varphi_1 \vert \varphi_2) } &= E^\M_t(\varphi_1 \vert \varphi_2) - E^\M_{t'}(\varphi_1 \vert \varphi_2) \\
&= P^\M\left( \varphi_1 - t' \varphi_2 + (t'-t) \varphi_2 \right) - P^\M\left( \varphi_1 - t' \varphi_2 \right) \\
&\leq (t'-t) P^\M(\varphi_2) = \abs{ t' - t } \norm{ \varphi_2 } \eqend{.}
\end{splitequation}

As a consequence of \corref{cor:Jordan2}, hockeystick divergences also exhibit several interesting properties relating to pairs of von Neumann algebras.
\begin{proposition}
\label{prop:hockeystick_properties2}
Let $\M, \N$ be von Neumann algebras, $\varphi, \varphi_1, \varphi_2 \in \M\stp$, $\psi, \psi_1, \psi_2 \in \N\stp$ and $t \geq 0$.
\begin{enumerate}
\item \emph{(Direct sums)} $E^{\M \oplus \N}_t = E^\M_t \oplus E^\N_t$, i.e.,
\begin{equation}
E^{\M \oplus \N}_t\left( \varphi_1 \oplus \psi_1 \vert \varphi_2 \oplus \psi_2 \right) = E^\M_t(\varphi_1 \vert \varphi_2) + E^\N_t(\psi_1 \vert \psi_2) \eqend{.}
\end{equation}

\item \label{item:ETensorProducts} \emph{(Tensor products)} We have
\begin{equation}
E^{\M \otimes \N}_t\left( \varphi \otimes \psi_1 \vert \varphi \otimes \psi_2 \right) = \varphi(\1) E^\N_t(\psi_1 \vert \psi_2) \eqend{.}
\end{equation}

\item \emph{(Data processing inequality)} Let $T \colon \N \to \M$ be a normal positive contraction. Then
\begin{equation}
E^\N_t\left( \varphi \circ T \vert \psi \circ T \right) \leq E^\M_t(\varphi \vert \psi) \eqend{.}
\end{equation}

\item \label{item:Hockey_prop2_ConditionalExpectations} \emph{(Conditional expectations)} Let $\M \subset \N$ and $C \colon \N \to \M$ a normal conditional expectation. Then
\begin{align}
E^\N_t\left( \varphi \circ C \vert \psi \circ C \right) = E^\M_t(\varphi \vert \psi) \eqend{.}
\end{align}

\item \label{item:Hockey_prop2_Compressions} \emph{(Compressions)} Let $\N = p \M p$ for a projection $p \in \PM$. Then $E^\N_t(\varphi \vert \psi) = E^\M_t(\varphi \vert \psi)$ for all $\varphi, \psi$ with $s^\M(\varphi), s^\M(\psi) \leq p$.

\item \label{item:Hockey_prop2_Martingale} \emph{(Martingale convergence)} Let $(\M_i)_{i \in I}$ with $M_i \subset \M$ for all $i \in I$ be an increasing net of von Neumann algebras such that $\bigcup_{i \in I} \M_i$ is $\sigma(\M, \M_*)$-dense in $\M$. Then it holds that
\begin{equation}
\lim_{i \in I} E^{\M_i}_t\left( \varphi_1 \rvert_{\M_i} \Big\vert \varphi_2 \rvert_{\M_i} \right) = E^\M_t(\varphi_1 \vert \varphi_2) \eqend{.}
\end{equation}
\end{enumerate}
\end{proposition}
\begin{proof}
All claims are immediate consequences of \corref{cor:Jordan2}.
\end{proof}
We refer to \cite[Lemma~4]{sharmawarsi2012} for a similar proof of the data processing inequality for matrix algebras.

At this point, it is apparent that the hockeystick divergence at fixed likelihood ratio, i.e., the map $E^\M_t \colon \M\stp \times \M\stp \to \mathbb{R}_+$, exhibits various properties characteristic of relative entropies: it is non-negative, jointly (lower semi-) continuous, jointly convex, satisfies the data processing inequality (in particular, $E^\M_t$ is monotone under inclusions and invariant under automorphisms $\alpha \in \operatorname{Aut}(\M)$), and distinguishes states, namely $E^\M_t(\varphi_1 \vert \varphi_2) = 0 \iff \varphi_1 = \varphi_2$ for $t \geq 1$. On the other hand, the hockeystick divergence also has properties that distinguishes it from relative entropies: $E^\M_t(\varphi_1 \vert \varphi_2)$ is always finite (also for $\varphi_1 \not\ll \varphi_2$), and the hockeystick divergence is not additive under tensor products (note that in \propref[item:ETensorProducts]{prop:hockeystick_properties2}, the states on the second factor were assumed to coincide).

The latter two properties will change when we pass to $f$-divergences by integrating in the likelihood ratio $t$.

\subsection{\texorpdfstring{$f$}{f}-divergences}
\label{sec:hockeystick:fdivergence}

As explained in the introduction, the hockeystick divergences are the basic building blocks of our main object of interest, $f$-divergences in von Neumann algebras that generalize the classical Csiszár divergences~\cite{csiszar1963}.

The main idea is to consider integrals of hockeystick divergences in the likelihood ratio parameter $t$ as in \eqref{eq:DIntegralIntro}
\begin{equation}
D_f(\mu_1 \vert \mu_2) = \int_1^\infty \left[ f''(t) E_t(\mu_1 \vert \mu_2) + t^{-3} f''\left( t^{-1} \right) E_t(\mu_2 \vert \mu_1) \right] \total t \eqend{.}
\end{equation}
Note that the integral starts at $t = 1$, not at $t = 0$; the small $t$-values are accounted for by the second term using the state exchange relation, \propref[item:EStateExchange]{prop:hockeystick_properties}. Since $f$ is supposed to be convex, the function in front of this term is also natural: the map
\begin{equation}
\label{eq:fStarInvolution}
f \mapsto f_* \eqend{,} \qquad f_*(t) \defeq t f\left( t^{-1} \right)
\end{equation}
is an involution on the set of convex functions $(0, \infty) \to \mathbb{R}$, and $(f_*)''(t) = t^{-3} f''\left( t^{-1} \right)$. This leads to the definition of $f$-divergences, in complete analogy to the definition given in~\cite{hirchetomamichel2024} for matrix algebras.
\begin{definition}
\label{def:fdivergence}
Let $\M$ be a von Neumann algebra $\M$ and $f \colon (0, \infty) \to \mathbb{R}$ convex. The \emph{$f$-divergence} of $\M$ is the map
\begin{splitequation}
\label{eq:Df}
D^\M_f &\colon \M\stp \times \M\stp \to (-\infty,\infty] \eqend{,} \\
D^\M_f(\varphi_1 \vert \varphi_2) &\defeq \int_1^\infty \left[ f''(t) E^\M_t(\varphi_1 \vert \varphi_2) + (f_*)''(t) E^\M_t(\varphi_2 \vert \varphi_1) \right] \total t + \varphi_1(\1) - \varphi_2(\1) \eqend{.}
\end{splitequation}
\end{definition}
Since $t \mapsto E^\M_t(\varphi_1 \vert \varphi_2)$ is continuous in $t$ by \propref{prop:hockeystick_properties}, and convex functions are almost everywhere twice differentiable with $f''(t) \geq 0$, the integral is well-defined as a Lebesgue integral with values in the extended positive reals $[0,\infty]$. Because of the usual correction term $\varphi_1(\1) - \varphi_2(\1)$ for non-normalized functionals, in general the $f$-divergence takes values in $(-\infty,\infty]$. We also note that since $D^\M_f$ depends on $f$ only via its second derivative, $D^\M_f = D^\M_g$ if $f-g$ is affine linear.

As a case of particular interest, we point out that for the information function $f = f_0$ with $f_0(t) = t \ln t$, the $f$-divergence takes the form
\begin{equation}
\label{eq:fDivergenceInformationFunction}
D^\M_{f_0}(\varphi_1 \vert \varphi_2) = \int_1^\infty \left[ \frac{1}{t} E^\M_t(\varphi_1 \vert \varphi_2) + \frac{1}{t^2} E^\M_t(\varphi_2 \vert \varphi_1) \right] \total t + \varphi_1(\1) - \varphi_2(\1) \eqend{.}
\end{equation}

Before we study the properties of $D^\M_f$, we indicate that they are more closely related to relative entropies than hockeystick divergences by considering two very special test cases.
\begin{lemma}
\label{lemma:DfCommutingDensities}
Let $f \colon (0,\infty) \to \mathbb{R}$ be convex and assume that $f$ is differentiable at $1$.
\begin{enumerate}
\item \label{item:DfC} For the trivial von Neumann algebra $\M = \mathbb{C}$ with positive linear functionals $a_1, a_2 \in \mathbb{R}_+$, we have\footnote{We here adopt the convention $0 \cdot \infty = 0$.}
\begin{equation}
\label{eq:DfC}
D^\mathbb{C}_f(a_1 \vert a_2) = a_2 f\left( \frac{a_1}{a_2} \right) - a_2 f(1) + (a_2 - a_1) \left[ f'(1) - 1 \right] \eqend{,} \quad a_1, a_2 > 0 \eqend{,}
\end{equation}
and the cases $a_1 = 0$ and $a_2 = 0$ can be obtained as limits.

\item \label{item:DfCommutingDensities} Let $\M$ be a von Neumann algebra with a faithful normal finite trace $\tau$, and consider positive normal functionals of the form $\varphi_j(x) = \tau(h_j x)$ with $h_j \in \M_+$, \exref[item:TypeIIJordan]{example:Jordan}. If the densities $h_1$ and $h_2$ commute and $h_2$ is invertible, then
\begin{equation}
\label{eq:fDivergenceCommutingDensities}
D^\M_f(\varphi_1 \vert \varphi_2) = \tau\left( h_2 f\left( h_1 h_2^{-1} \right) \right) - \tau(h_2) f(1) + \tau\left( h_1 - h_2 \right) \left[ f'(1) - 1 \right] \eqend{.}
\end{equation}

\item \label{item:ProbabilityMeasures} Let $X$ be a measurable space with two probability measures $\mu_1 \ll \mu_2$. Then for $f(1) = 0$ it holds that
\begin{equation}
\label{eq:DfProbabilityMeasures}
D^{L^\infty(X,\mu_2)}_f(\mu_1 \vert \mu_2) = D^{L^\infty(X,\mu_2)}_{f_*}(\mu_2 \vert \mu_1) = \int_X f\left( \frac{\total \mu_1}{\total \mu_2} \right) \total \mu_2 \eqend{.}
\end{equation}
\end{enumerate}
\end{lemma}
Observe that for the information function $f = f_0$, the $f$-divergence \eqref{eq:fDivergenceCommutingDensities} simplifies to
\begin{equation}
D^\M_{f_0}(\varphi_1 \vert \varphi_2) = \tau\left( h_1 \ln h_1 - h_1 \ln h_2 \right) \eqend{,}
\end{equation}
the relative entropy of $\varphi_1$ and $\varphi_2$. Moreover, the formula \eqref{eq:DfProbabilityMeasures} is the definition of $f$-divergence in classical probability \eqref{eq:DfIntro}. Hence the $f$-divergences considered here generalize their classical counterparts.
\begin{proof}
\ref{item:DfC} By definition of the $f$-divergence, we have
\begin{equation}
D^\mathbb{C}_f(a_1 \vert a_2) = \mathord{\underbrace{\int_1^\infty f''(t) (a_1 - t a_2)_+ \total t}_{\eqdef I_1}} + \mathord{\underbrace{\int_1^\infty f_*''(t) (a_2 - t a_1)_+ \total t}_{\eqdef I_2}} + a_1 - a_2 \eqend{.}
\end{equation}
Let $a_2 > 0$ and $r \defeq a_1/a_2$. The integrand of $I_1$ vanishes for $t \geq r$, so $I_1 = 0$ for $r \leq 1$. For $r > 1$, we have
\begin{equation}
I_1 = a_2 \int_1^\infty f''(t) (r-t)_+ \total t = a_2 \int_1^r f''(t) (r-t) \total t = a_2 \left[ f(r) - f(1) - (r-1) f'(1) \right] \eqend{.}
\end{equation}
The integrand of $I_2$ vanishes for $t \geq 1/r$, so $I_2 = 0$ for $r \geq 1$. For $r < 1$, we use $f_*''(t) = t^{-3} f''(t^{-1})$ and set $s = t^{-1}$ to obtain
\begin{splitequation}
I_2 &= a_2 \int_1^\infty t^{-3} f''(t^{-1}) (1 - r t)_+ \total t = a_2 \int_r^1 f''(s) (s-r) \total s \\
&= a_2 \left[ f(r) - f(1) + (1-r) f'(1) \right] \eqend{.}
\end{splitequation}
Adding both integrals gives the claimed result for $a_2 > 0$. That the formula for $a_1 = 0$ or $a_2 = 0$ can be obtained as a limit is immediate.

\ref{item:DfCommutingDensities} As $h_1$ and $h_2$ commute, they have a joint spectral calculus, and we denote by $\mu^{1,2}_\tau$ their joint spectral measure in the tracial state $\tau$, supported on their joint spectrum $\sigma(h_1, h_2) \subset (\mathbb{R}_+)^2$. Using this spectral measure, Tonelli's Theorem, and part~\ref{item:DfC}, we obtain
\begin{splitequation}
&D^\M_f(\varphi_1 \vert \varphi_2) = \int_1^\infty \left[ f''(t) \tau\left( (h_1 - t h_2)_+ \right) + f_*''(t) \tau\left( (h_2 - t h_1)_+ \right) \right] \total t + \tau\left( h_1 - h_2 \right) \\
&\quad= \int_1^\infty \int_{(\mathbb{R}_+)^2} \Bigl[ f''(t) (\lambda_1 - t \lambda_2)_+ + f_*''(t) (\lambda_2 - t \lambda_1)_+ \Bigr] \total \mu_\tau^{1,2}(\lambda_1,\lambda_2) \total t + \tau\left( h_1 - h_2 \right) \\
&\quad= \int_{(\mathbb{R}_+)^2} \int_1^\infty \Bigl[ f''(t) (\lambda_1 - t \lambda_2)_+ + f_*''(t) (\lambda_2 - t \lambda_1)_+ \Bigr] \total t \total \mu_\tau^{1,2}(\lambda_1,\lambda_2) + \tau\left( h_1 - h_2 \right) \\
&\quad= \int_{(\mathbb{R}_+)^2} \Bigl[ D^\mathbb{C}_f(\lambda_1 \vert \lambda_2) - \lambda_1 + \lambda_2 \Bigr] \total \mu_\tau^{1,2}(\lambda_1,\lambda_2) + \tau\left( h_1 - h_2 \right) \\
&\quad= \int_{(\mathbb{R}_+)^2} D^\mathbb{C}_f(\lambda_1 \vert \lambda_2) \total \mu_\tau^{1,2}(\lambda_1,\lambda_2) \\
&\quad= \int_{(\mathbb{R}_+)^2} \left[ \lambda_2 f\left( \frac{\lambda_1}{\lambda_2} \right) - \lambda_2 f(1) + (\lambda_2 - \lambda_1) \left[ f'(1) - 1 \right] \right] \total \mu_\tau^{1,2}(\lambda_1,\lambda_2) \\
&\quad= \tau\left( h_2 f\left( h_1 h_2^{-1} \right) \right) - \tau(h_2) f(1) + \tau\left( h_1 - h_2 \right) \left[ f'(1) - 1 \right] \eqend{.}
\end{splitequation}

\ref{item:ProbabilityMeasures} Integration against $\mu_2$ defines a faithful normal finite trace on $\M = L^\infty(X,\mu_2)$, and the states $\varphi_1, \varphi_2 \in \M\stp$ given by the densities $h_1 \in L^1(X,\mu_2)$, $h_2 = 1$ satisfy $h_1/h_2 = \total \mu_1 / \total \mu_2$ with the Radon--Nikod{\'y}m derivative. Hence \eqref{eq:DfProbabilityMeasures} follows by specialization from~\ref{item:DfCommutingDensities}, compare \exref[item:AbelianJordan]{example:Jordan}.
\end{proof}

As for the hockeystick divergences, the $f$-divergences enjoy various useful properties.
\begin{theorem}
\label{thm:fdivergence_properties}
Let $\M$ be a von Neumann algebra, $\varphi_1, \varphi_2 \in \M\stp$ and $f \colon (0,\infty) \to \mathbb{R}$ convex with $f_*$ defined in~\eqref{eq:fStarInvolution}. Then the $f$-divergence $D^\M_f$ has the following properties:
\begin{enumerate}
\item \label{item:DBounds} \emph{(Bounds and state discrimination)}
\begin{equation}
\label{eq:DBounds}
\varphi_1(\1) - \varphi_2(\1) \leq D^\M_f(\varphi_1 \vert \varphi_2) \leq \varphi_1(\1) \norm{ f'' \bigr\rvert_{[1,\infty)} }_{L^1} + \varphi_2(\1) \norm{ (f_*)'' \bigr\rvert_{[1,\infty)} }_{L^1} \eqend{.}
\end{equation}
If $f''(t) > 0$ for almost all $t > 0$, then
\begin{equation}
\label{eq:DDistinguishesStates}
D^\M_f(\varphi_1 \vert \varphi_2) = \varphi_1(\1) - \varphi_2(\1) \iff \varphi_1 = \varphi_2 \eqend{.}
\end{equation}
Furthermore, for general convex $f$,
\begin{equations}[eq:DInfinite]
\varphi_1 &\not\ll \varphi_2 \eqend{,} \quad \norm{ f'' \bigr\rvert_{[1,\infty)} }_{L^1} = \infty \implies D^\M_f(\varphi_1 \vert \varphi_2) = \infty \eqend{,} \\
\varphi_2 &\not\ll \varphi_1 \eqend{,} \quad \norm{ (f_*)'' \bigr\rvert_{[1,\infty)} }_{L^1} = \infty \implies D^\M_f(\varphi_1 \vert \varphi_2) = \infty \eqend{.}
\end{equations}
\item \label{item:DContinuity} \emph{(Lower semicontinuity)} $D^\M_f \colon \M\stp \times \M\stp \to (-\infty,\infty]$ is jointly lower semicontinuous in the $\sigma(\M_*, \M)$-topology, and in particular in the norm topology of $\M_*$. For $\norm{ f'' \bigr\rvert_{[1,\infty)} }_{L^1} + \norm{ (f_*)'' \bigr\rvert_{[1,\infty)} }_{L^1} < \infty$, the $f$-divergence takes only finite values and is continuous in the norm topology of $\M_*$.
\item \label{item:DConvexity} \emph{(Convexity)} $D^\M_f$ is jointly convex in the functionals.
\item \label{item:DStateExchange} \emph{(State exchange)} $D^\M_f(\varphi_2 \vert \varphi_1) = D^\M_{f_*}(\varphi_1 \vert \varphi_2) + 2 \varphi_2(\1) - 2 \varphi_1(\1)$.
\item \label{item:DAlternativeFormula} \emph{(Alternative representation)} If $f$ is differentiable at $1$, we have
\begin{splitequation}
\label{eq:f-divergenceFullIntegralFormula}
D^\M_f(\varphi_1 \vert \varphi_2) &= \lim_{\epsilon \searrow 0} \Biggl[ \int_\epsilon^\infty f''(t) E^\M_t(\varphi_1 \vert \varphi_2) \total t + \varphi_2(\1) [ f(\epsilon) - f(1) + f'(\epsilon) - \epsilon f'(\epsilon) ] \\
&\hspace{6em}+ [ \varphi_1(\1) - \varphi_2(\1) ] [ 1 + f'(\epsilon) - f'(1) ] \Biggr] \eqend{.}
\end{splitequation}
\end{enumerate}
Furthermore, if $\N$ is another von Neumann algebra,
\begin{enumerate}
\setcounter{enumi}{5}
\item \label{item:DDPI} \emph{(Data processing inequality)} Let $T \colon \N \to \M$ be a normal linear unital map (i.e., $T(\1_N) = \1_\M$). Then
\begin{align}
D^\N_f\left( \varphi_1 \circ T \vert \varphi_2 \circ T \right) \leq D^\M_f(\varphi_1 \vert \varphi_2) \eqend{.}
\end{align}
If $\M \subset \N$ is a subalgebra and $T$ a normal conditional expectation, equality holds.
\item \label{item:DSumsAndProducts} \emph{(Direct sums and tensor products)} We have for $\varphi, \varphi_1, \varphi_2 \in \M\stp$ and $\psi_1,\psi_2 \in \N\stp$
\begin{equations}
D^{\M \oplus \N}_f\left( \varphi_1 \oplus \psi_1 \vert \varphi_2 \oplus \psi_2 \right) &= D^\M_f(\varphi_1 \vert \varphi_2) + D^\N_f(\psi_1 \vert \psi_2) \eqend{,} \\
D^{\M \otimes \N}_f\left( \varphi \otimes \psi_1 \vert \varphi \otimes \psi_2 \right) &= \varphi\left( \1_\M \right) D^\N_f(\psi_1 \vert \psi_2) \eqend{.}
\end{equations}
\item \label{item:f-divergence_Compressions} \emph{(Compressions)} Let $\N = p \M p$ for a projection $p \in \PM$. Then $D^\N_f(\varphi \vert \psi) = D^\M_f(\varphi \vert \psi)$ for all $\varphi, \psi$ with $s^\M(\varphi), s^\M(\psi) \leq p$.
\item \label{item:f-divergenceMartingale} \emph{(Martingale convergence)} Let $(\M_i)_{i \in I}$ with $M_i \subset \M$ for all $i \in I$ be an increasing net of von Neumann algebras such that $\bigcup_{i \in I} \M_i$ is $\sigma(\M, \M_*)$-dense in $\M$. Then it holds that
\begin{equation}
\lim_{i \in I} D^{\M_i}_f\left( \varphi_1 \rvert_{\M_i} \Big\vert \varphi_2 \rvert_{\M_i} \right) = D^\M_f(\varphi_1 \vert \varphi_2) \eqend{.}
\end{equation}
\end{enumerate}
\end{theorem}
\begin{proof}
\ref{item:DBounds} The bounds \eqref{eq:DBounds} follow directly from positivity of the hockeystick divergences and the upper bound in \propref[item:EBoundsLimit]{prop:hockeystick_properties}. We note that the second $L^1$-norm can be written as $\norm{ (f_*)'' \bigr\rvert_{[1,\infty)} }_{L^1} = \int_0^1 s f''(s) \total s$. If the lower bound is attained, $D^\M_f(\varphi_1 \vert \varphi_2) = \varphi_1(\1) - \varphi_2(\1)$, and $f'' > 0$ almost everywhere, positivity of the integrand of $D^\M_f$ implies $E^\M_t(\varphi_1 \vert \varphi_2) = 0$ for almost all $t$, hence for all $t$ by continuity. \corref[item:submultiples]{cor:EStateComparison} then implies $\varphi_1 = \varphi_2$.

The implications \eqref{eq:DInfinite} follow because $\varphi_1 \not\ll \varphi_2$ implies that $E^\M_t(\varphi_1 \vert \varphi_2) \geq \varphi_1(s_2^\perp) > 0$ for all $t$, and analogously for $\varphi_2 \not\ll \varphi_1$.

\ref{item:DContinuity} The lower semicontinuity of $(\varphi_1,\varphi_2) \mapsto E^\M_t(\varphi_1 \vert \varphi_2)$ given by \propref[item:EContinuity]{prop:hockeystick_properties} passes on to $D^\M_f$ by Fatou's lemma. If the specified $L^1$-norm bound holds, $D^\M_f(\varphi_1 \vert \varphi_2) < \infty$ by \ref{item:DBounds}. The norm continuity of the hockeystick divergence given by \propref[item:EContinuity]{prop:hockeystick_properties} implies the norm continuity of $D^\M_f$ by dominated convergence.

All remaining properties follow from the corresponding properties of the hockeystick divergence, either directly or by straightforward computations. Note that in the data processing inequality, we have to assume $T$ to be unital only because of the additional term $\varphi_1(\1) - \varphi_2(\1)$.
\end{proof}
As expected of a general $f$-divergence, $D^\M_f$ is not additive on tensor products for general $f$. In the special case of the information function we will have additivity; this will be investigated later.

Relative entropies and $f$-divergences are typically hard to compute explicitly, so it is important to have methods for estimating them. On the one hand, any subalgebra $\N \subset \M$ gives lower bounds because by the data processing inequality, the inclusion $\iota \colon \N \to \M$ yields $D^\M_f(\varphi_1 \vert \varphi_2) \geq D^\N_f\left( \varphi_1 \rvert_\N \Big\vert \varphi_2 \rvert_\N \right)$. For instance,
\begin{equation}
D^\M_f(\varphi_1 \vert \varphi_2) \geq D^\mathbb{C}_f(\varphi_1(\1) \vert \varphi_2(\1)) \eqend{.}
\end{equation}
Sharper bounds of this form depend on a suitable choice of $\N$, and we consider two-dimensional subalgebras.
\begin{lemma}
\label{lemma:2dLowerBounds}
Let $\M$ be a von Neumann algebra, $f \colon (0,\infty) \to \mathbb{R}$ convex and differentiable at $1$, and $\varphi_1,\varphi_2 \in \M\stp$ states. For any projection $p \in \PM$ such that $\varphi_2(p) \neq 0$, $\varphi_2(p^\perp) \neq 0$, it holds that
\begin{equation}
\label{eq:LowerBound2}
D^\M_f(\varphi_1 \vert \varphi_2) \geq \varphi_2(p) f\left( \frac{\varphi_1(p)}{\varphi_2(p)} \right) + \varphi_2(p^\perp) f\left( \frac{\varphi_1(p^\perp)}{\varphi_2(p^\perp)} \right) - f(1) \eqend{.}
\end{equation}
For the information function $f = f_0$, we obtain Pinsker's inequality
\begin{equation}
\label{eq:DBound3}
D^\M_{f_0}(\varphi_1 \vert \varphi_2) \geq 2 \left[ E^\M_1(\varphi_1 \vert \varphi_2) \right]^2 = \frac{1}{2} \norm{ \varphi_1 - \varphi_2 }^2 \eqend{.}
\end{equation}
\end{lemma}
\begin{proof}
For $p \in \PM$, we consider the (one- or) two-dimensional subalgebra $\mathbb{C} \oplus \mathbb{C} \cong \N \subset \M$ generated by $p$. On $\mathbb{C} \oplus \mathbb{C}$, we have the non-normalized faithful trace $\tau = \Tr_{M_2}$, with respect to which the restrictions of the functionals $\varphi_j$ have the densities $h_j \defeq \varphi_j(p) \oplus \varphi_j(p^\perp)$, $j=1,2$. By assumption, $h_2$ is invertible. Hence we may apply \eqref{eq:fDivergenceCommutingDensities} to get the claimed bound \eqref{eq:LowerBound2}.

We now specialize to $f = f_0$. Setting $w_j \defeq \varphi_j(p)$ and $\delta \defeq w_1 - w_2$, the bound \eqref{eq:LowerBound2} simplifies to
\begin{splitequation}
&w_1 \ln\left( \frac{w_1}{w_2} \right) + (1-w_1) \ln\left( \frac{1-w_1}{1-w_2} \right) \\
&\quad= (w_2+\delta) \ln\left( 1 + \frac{\delta}{w_2} \right) + (1-w_2-\delta) \ln\left( 1 - \frac{\delta}{1-w_2} \right) \defeq F(\delta) \eqend{.}
\end{splitequation}
As a function of $\delta$ for fixed $w_2$, one checks that $F(0) = F'(0) = 0$ and $F''(\delta) = (w_2 + \delta)^{-1} + (1-w_2-\delta)^{-1} \geq 4$, the inequality following since $\delta + w_2 = w_1 \in (0,1)$. This implies $F(\delta) \geq 2 \delta^2 = 2 \bigl[ \varphi_1(p) - \varphi_2(p) \bigr]^2$. We now consider the support projection $p$ of $\psi_+$, where $\psi \defeq \varphi_1 - \varphi_2$. If $\varphi_2(p) \not\in \{ 0,1 \}$, we may apply the bound, which equals $2 \left[ E^\M_1(\varphi_1 \vert \varphi_2) \right]^2 = \frac{1}{2} \norm{ \varphi_1 - \varphi_2 }^2$ as claimed, with the equality following from the third equality in \eqref{eq:varphiSup}, taking into account that $\varphi_1(\1) = \varphi_2(\1) = 1$.

It remains to consider the two cases $\varphi_2(p) = 0$ and $\varphi_2(p) = 1$. Since $\psi_+(\1) = \frac{1}{2} \norm{ \varphi_1 - \varphi_2 }$, in the first case $\varphi_2(p) = 0$ we obtain $\varphi_1(p) = \psi(p) = \psi_+(\1) = \frac{1}{2} \norm{ \varphi_1 - \varphi_2 }$. For $\varphi_1 = \varphi_2$, \eqref{eq:DBound3} is trivial, and for $\varphi_1(p) > 0$, we have $\varphi_1 \not\ll \varphi_2$, hence $D^\M_{f_0}(\varphi_1 \vert \varphi_2) = \infty$ by \thmref[item:DBounds]{thm:fdivergence_properties}. So \eqref{eq:DBound3} is trivial also in this case.

For $\varphi_2(p) = 1$, we also obtain $\varphi_1 = \varphi_2$ and again triviality of \eqref{eq:DBound3}.
\end{proof}

\section{Relative modular theory and relative entropies}
\label{sec:relentropy}

In this section we recall some elements of relative modular theory. We will work with a general von Neumann algebra $\M$ in a standard representation on a Hilbert space $\Hil$, where every positive normal functional $\varphi \in \M\stp$ is given as $\varphi(x) = \langle \Omega_\varphi, x \, \Omega_\varphi\rangle$ for $x \in \M$, with a unique vector $\Omega_\varphi \in \Hil$ in the natural positive cone.
\begin{definition}
\label{def:RelativeS}
Let $\M \subset \BH$ be a von Neumann algebra and $\varphi_j \in \M\stp$ be given by vectors $\Omega_j$, $j = 1, 2$. The \emph{relative Tomita operator} $S_{\varphi_2, \varphi_1}$ is the closure of
\begin{equations}[eq:RelativeS]
S^0_{\varphi_2,\varphi_1} \colon \M \Omega_1 + (\M\Omega_1)^\perp &\to \Hil \eqend{,} \\
S^0_{\varphi_2,\varphi_1}(x \, \Omega_1 + \Psi) &\defeq s^\M(\varphi_1) x^* \Omega_2 \eqend{,}
\end{equations}
and the \emph{relative modular operator} is its modulus squared:
\begin{equation}
\Delta_{\varphi_2,\varphi_1} \defeq S_{\varphi_2,\varphi_1}^* S_{\varphi_2,\varphi_1} \eqend{.}
\end{equation}
\end{definition}
The implicit claim that $S^0_{\varphi_2,\varphi_1}$ is closable is proven in \cite[Lemma~2.2]{araki1977}. Note that the support projection $s^\M(\varphi_1)$ in \eqref{eq:RelativeS} is necessary to make $S^0_{\varphi_2,\varphi_1}$ well-defined for non-faithful states. We will not need the antiunitary part of the polar decomposition of $S_{\varphi_2,\varphi_1}$, but only the positive part, the relative modular operator $\Delta_{\varphi_2,\varphi_1}$. Some of its basic properties are the following \cite{araki1976,araki1977}:
\begin{lemma}
\label{lemma:RelativeDelta}
The relative modular operator $\Delta_{\varphi_2, \varphi_1}$ of two normal positive functionals $\varphi_1, \varphi_2 \in \M\stp$ on a von Neumann algebra $\M$ satisfies:
\begin{enumerate}
\item $\Delta_{\varphi_2,\varphi_1}$ is a selfadjoint positive operator,
\item \label{item:RelativeDeltaKernel} The projection onto $\ker \Delta_{\varphi_2, \varphi_1}$ is $\1 - s^{\M'}(\varphi_1) s^\M(\varphi_2)$,
\item \label{item:RelativeDeltaScaling} $\Delta_{c_2 \varphi_2, c_1 \varphi_1} = \dfrac{c_2}{c_1} \Delta_{\varphi_2, \varphi_1}$ for $c_1, c_2 > 0$.
\end{enumerate}
\end{lemma}
In~\ref{item:RelativeDeltaKernel}, $s^{\M'}(\varphi_1)$ is the support projection of $\varphi_1$ for the commutant $\M'$ of $\M$, namely the orthogonal projection onto $\overline{\M \Omega_{1}}$.

With the help of relative modular operators, the relative entropy has been defined by Araki~\cite{araki1977} for general von Neumann algebras\footnote{Note that Ref.~\cite{araki1977} writes $S(\psi/\varphi)$ for $S^\M_\mathrm{rel}(\varphi \vert \psi)$.} as follows:
\begin{definition}
\label{def:RelativeEntropy}
Let $\M$ be a von Neumann algebra. The \emph{relative entropy} is the map
\begin{splitequation}
\label{eq:Srel}
S^\M_\mathrm{rel} &\colon \M\stp \times \M\stp \to [0,\infty] \eqend{,} \\
S^\M_\mathrm{rel}(\varphi \vert \psi) &= \begin{cases} - \langle \Omega_\varphi, \ln \Delta_{\psi,\varphi} \, \Omega_\varphi \rangle = \langle \Omega_\varphi, \ln \Delta_{\varphi,\psi} \, \Omega_\varphi \rangle & \varphi \ll \psi \eqend{,} \\ + \infty & \text{otherwise.} \end{cases}
\end{splitequation}
\end{definition}
In the case of matrix algebras or only semifinite von Neumann algebras, this definition coincides with the familiar trace expression over densities as shown in \corref{cor:semifinite_relative_entropy}. The case distinction in \eqref{eq:Srel} is a consequence of \lemmaref[item:RelativeDeltaKernel]{lemma:RelativeDelta}, which implies
\begin{equation}
\Omega_1 \perp \ker \Delta_{\varphi_2,\varphi_1} \iff \Omega_1 = s^{\M'}(\varphi_1) s^\M(\varphi_2) \Omega_1 = s^\M(\varphi_2) \Omega_1 \iff \varphi_1 \ll \varphi_2 \eqend{,}
\end{equation}
i.e. for $\varphi_1 \not\ll \varphi_2$, the expectation value $- \langle \Omega_{\varphi_1}, \ln \Delta_{\varphi_2,\varphi_1} \Omega_{\varphi_1} \rangle$ is ill-defined/infinite. We remark that also for $\varphi_1 \ll \varphi_2$, the relative entropy may be infinite.

Araki's relative entropy is well-studied, and we refer to \cite{araki1976,araki1977,ohyapetz1993} for a detailed discussion of all its properties. Here we recall only those properties of $S^\M_\mathrm{rel}$ that we will need in the following.
\begin{proposition}
\label{prop:srel_properties}
Let $\M$ be a von Neumann algebra, and $\varphi, \psi \in \M\stp$.
\begin{enumerate}
\item \label{item:Srel_independence} $S^\M_\mathrm{rel}( \varphi \vert \psi )$ is independent of the choice of vector representatives of $\varphi$ and $\psi$.
\item \emph{(Joint convexity and bounds)} $S^\M_\mathrm{rel}$ is jointly convex and satisfies
\begin{equation}
S^\M_\mathrm{rel}( \varphi \vert \psi ) \geq - \varphi(\1) \ln\left[ \frac{\psi\left( s^\M(\varphi) \right)}{\varphi(\1)} \right] \eqend{.}
\end{equation}
If $\varphi(\1) = \psi(\1) > 0$, it holds that $S^\M_\mathrm{rel}( \varphi \vert \psi ) \geq 0$, and equality is obtained if and only if $\varphi = \psi$. If $\psi_1 \geq \psi_2$, it holds that $S^\M_\mathrm{rel}( \varphi \vert \psi_1 ) \leq S^\M_\mathrm{rel}( \varphi \vert \psi_2 )$.
\item \emph{(Scaling)} For $\lambda, \mu > 0$ one has the scaling relation
\begin{equation}
S^\M_\mathrm{rel}( \lambda \varphi \vert \mu \psi ) = \lambda S^\M_\mathrm{rel}( \varphi \vert \psi ) - \lambda \varphi(\1) \ln \frac{\mu}{\lambda} \eqend{.}
\end{equation}
\item \label{item:domination} \emph{(Dominated and monotone convergence)} Consider two families $\{ \varphi_n \}_{n \in \mathbb{N}}$, $\{ \psi_n \}_{n \in \mathbb{N}}$ of normal positive linear functionals converging in norm to $\varphi$ and $\psi$. If $\psi_n$ uniformly dominates $\varphi_n$, i.e., if $\varphi_n \leq \lambda \psi_n$ for a fixed $\lambda > 0$ and all $n \in \mathbb{N}$, it holds that
\begin{equation}
\lim_{n \to \infty} S^\M_\mathrm{rel}( \varphi_n \vert \psi_n ) = S^\M_\mathrm{rel}( \varphi \vert \psi ) \eqend{.}
\end{equation}
If $\psi_n$ is monotonically decreasing, it holds that
\begin{equation}
\lim_{n \to \infty} S^\M_\mathrm{rel}( \varphi \vert \psi_n ) = S^\M_\mathrm{rel}( \varphi \vert \psi ) \eqend{.}
\end{equation}
\item \label{item:Srel_ConditionalExpectation} \emph{(Conditional expectation)} Let $\M \subset \N$ be a von Neumann subalgebra and $T \colon \N \to \M$ a normal conditional expectation. Then it holds that
\begin{equation}
S^\N_\mathrm{rel}( \varphi \circ T \vert \psi \circ T ) = S^\M_\mathrm{rel}( \varphi \vert \psi ) \eqend{.}
\end{equation}
\item \label{item:Srel_Martingale} \emph{(Martingale convergence)} Let $(\M_i)_{i \in I}\subset \M$ be an increasing net of von Neumann algebras such that $\bigcup_{i \in I} \M_i$ is $\sigma(\M,\M_*)$-dense in $\M$. Then it holds that
\begin{equation}
\lim_{i \in I} S^{\M_i}_\mathrm{rel}\left( \varphi \rvert_{\M_i} \Big\vert \psi \rvert_{\M_i} \right) = S^\M_\mathrm{rel}( \varphi \vert \psi ) \eqend{.}
\end{equation}
\end{enumerate}
\end{proposition}
\begin{proof}
\ref{item:Srel_independence} -- \ref{item:domination} are proven in~\cite[Lemma~3.2 and Thms.~3.6 -- 3.8]{araki1977}. \ref{item:Srel_ConditionalExpectation} follows from~\cite[Thm.~4]{petz1986c}, and \ref{item:Srel_Martingale} was shown in~\cite[Thm.~11]{petz1985} and~\cite[Thm.~4.1~(vi)]{kosaki1986}.
\end{proof}

\subsection{Modular bounds on hockeystick and \texorpdfstring{$f$}{f}-divergences}
\label{sec:ModularBounds}

Our analysis of hockeystick and $f$-divergences up to this point rests entirely on the Jordan decomposition and its consequences, whereas Araki's relative entropy is formulated with the help of relative modular theory. Despite this difference in technique, it is apparent that the two concepts are related: In addition to the various properties shared by $f$-divergences and relative entropy given in \thmref{thm:fdivergence_properties}, the lower bounds of \lemmaref{lemma:2dLowerBounds} and the special cases considered in \lemmaref{lemma:DfCommutingDensities} indicate close connections, in particular for the information function $f = f_0$.

In this section, we make this connection explicit by showing how relative modular operators can be used to bound hockeystick divergences and $f$-divergences from above and below. We will make the following standing assumption:
\begin{assumption}
\label{ass:modularbounds}
Let $\M \subset \BH$ be a von Neumann algebra, and consider positive normal functionals $\varphi_1, \varphi_2 \in \M\stp$ represented by vectors $\Omega_1, \Omega_2 \in \Hil$, namely $\varphi_k(x) = \langle \Omega_k, x \, \Omega_k \rangle$ for $x \in \M$ and $\norm{ \Omega_k }^2 = \varphi_k(\1)$, $k = 1, 2$. Denote by $s_k = s^\M(\varphi_k)$, $k = 1, 2$ their support projections in $\M$, and by $\Delta_{\varphi_2,\varphi_1}$ the relative modular operator. Let furthermore $\mu_{12} \defeq \langle \Omega_1, E^{\Delta_{\varphi_2,\varphi_1}} \Omega_1 \rangle$ denote the spectral measure of the relative modular operator in $\Omega_1$, which is a finite measure of total mass $\int_0^\infty \total \mu_{12}(\lambda) = \norm{ \Omega_1 }^2 = \varphi_1(\1)$.
\end{assumption}
The following basic lemma shows that relative modular operators appear naturally in the setting of the hypothesis testing considered earlier.
\begin{lemma}
For $t > 0$, the hockeystick divergence (\defref{def:hockeystick}) can be expressed as
\begin{equation}
E^\M_t(\varphi_1 \vert \varphi_2) = \varphi_1(\1) - \inf_{p \in \PM} \left[ \norm{ \Omega_1 - p \, \Omega_1 }^2 + t \norm{ \Delta_{\varphi_2,\varphi_1}^\half p \, \Omega_1 }^2 + t \norm{ s_1^\perp \, p \, \Omega_2 }^2 \right] \eqend{.}
\end{equation}
\end{lemma}
\begin{proof}
For any $x \in \M$, we have $S_{\varphi_2,\varphi_1} x \, \Omega_1 = s_1 x^* \Omega_2$ by \defref{def:RelativeS} of the relative Tomita operator, and in particular $\norm{ s_1 x^* \Omega_2 } = \norm{ S_{\varphi_2,\varphi_1} x \, \Omega_1 } = \norm{ \Delta^\half_{\varphi_2,\varphi_1} x \, \Omega_1 }$. In case $x = p \in \PM$ is a projection, this implies
\begin{equation}
\varphi_2(p) = \varphi_2(p s_1 p) + \varphi_2(p s_1^\perp \, p) = \norm{ \Delta_{\varphi_2,\varphi_1}^\half p \, \Omega_1 }^2 + \norm{ s_1^\perp \, p \, \Omega_2 }^2 \eqend{,}
\end{equation}
and therefore
\begin{splitequation}
\Bigl[ \varphi_1(p) - t \varphi_2(p) \Bigr] &= \varphi_1(\1) - \varphi_1(p^\perp) - t \norm{ \Delta_{\varphi_2,\varphi_1}^\half p \, \Omega_1 }^2 - t \norm{ s_1^\perp \, p \, \Omega_2 }^2 \\
&= \varphi_1(\1) - \left[ \norm{ \Omega_1 - p \, \Omega_1 }^2 + t \norm{ \Delta_{\varphi_2,\varphi_1}^\half p \, \Omega_1 }^2 + t \norm{ s_1^\perp \, p \, \Omega_2 }^2 \right] \eqend{.}
\end{splitequation}
Taking the supremum over $p \in \PM$ yields the hockeystick divergence $E^\M_t(\varphi_1 \vert \varphi_2)$~\eqref{eq:SuccessProbability}.
\end{proof}

Dropping the last term $t \norm{ s_1^\perp \, p \, \Omega_2 }^2$ (which vanishes for faithful $\varphi_1$), we obtain the useful upper bound
\begin{equation}
\label{eq:EUpperModularBound}
E^\M_t(\varphi_1 \vert \varphi_2) \leq \varphi_1(\1) - \inf_{p \in \PM} \left[ \norm{ \Omega_1 - p \, \Omega_1 }^2 + t \norm{ \Delta_{\varphi_2,\varphi_1}^\half p \, \Omega_1 }^2 \right] \eqend{.}
\end{equation}
The following general lemma is the basic tool for estimating this infimum. Similar versions can be found in \cite{kosaki1986,ohyapetz1993,jaksicogatapilletseiringer2012}, but the graph projection perspective taken below makes its geometric content particularly clear.
\begin{lemma}
\label{lemma:GraphBound}
Let $A$ be a positive selfadjoint operator on a Hilbert space, and $\xi \in \Hil$. Then
\begin{equation}
\left\langle \xi, \frac{A}{\1+A} \xi \right\rangle = \inf_{\eta \in \dom A^\half} \left( \norm{ \xi - \eta }^2 + \norm{ A^\half \eta }^2 \right) \eqend{,}
\end{equation}
and the infimum is attained at the unique vector $\eta = (\1 + A)^{-1} \xi$.
\end{lemma}
\begin{proof}
The infimum coincides with the squared distance $d^2$ of $\tilde\xi \defeq \xi \oplus 0$ to the graph $\Gamma\left( A^\half \right)$ of $A^\half$, which is given by $d^2 = \norm{ P_A^\perp \tilde\xi }^2$ with the orthogonal projection $P_A$ onto the graph $\Gamma\left( A^\half \right)$. As
\begin{equation}
P_A = \begin{pmatrix} (\1 + A)^{-1} & A^\half (\1 + A)^{-1} \\ A^\half (\1 + A)^{-1} & A (\1 + A)^{-1} \end{pmatrix} \eqend{,}
\end{equation}
we have $d^2 = \langle \tilde\xi, (\1 - P_A) \tilde\xi \rangle = \langle \xi, A (\1 + A)^{-1} \xi\rangle$ as claimed. From the graph description, it is clear that the infimum is a minimum, and the minimum is attained at the unique $\eta \in \dom A^\half$ such that $P_A \tilde\xi = \eta \oplus A^\half \eta$, namely $\eta = (\1 + A)^{-1} \xi$.
\end{proof}

This lemma leads to an upper bound on hockeystick divergences in the following proposition, see \cite[Lemma~5.8, 5.9]{ohyapetz1993}, \cite[Thm.~6.1 (3)]{jaksicogatapilletseiringer2012} for closely related ideas. We also include a lower bound, which is based on the generalization of the Powers--St{\o}rmer inequality~\cite{powersstormer1970} to arbitrary exponents due to Ogata~\cite{ogata2011}, namely the inequality
\begin{equation}
\label{eq:PSO}
\frac{1}{2} \Bigl[ \varphi_1(\1) + \varphi_2(\1) - \norm{ \varphi_1 - \varphi_2 } \Bigr] \leq \langle \Omega_1, \Delta_{\varphi_2,\varphi_1}^r \Omega_1 \rangle \eqend{,}
\end{equation}
valid for any $\varphi_1,\varphi_2 \in \M\stp$ and any $r \in [0,1]$.
\begin{proposition}
\label{prop:EModularEstimate}
For $t > 0$, the hockeystick divergence satisfies the upper and lower bounds
\begin{equation}
\left\langle \Omega_1, \left( \1 - t^r \, \Delta_{\varphi_2,\varphi_1}^r \right) \Omega_1 \right\rangle \leq E^\M_t(\varphi_1 \vert \varphi_2) \leq \left\langle \Omega_1, \left( \1 + t \Delta_{\varphi_2,\varphi_1} \right)^{-1} \Omega_1 \right\rangle
\end{equation}
for any $r \in [0,1]$. The upper bound holds with equality in the limits $t \to 0$ and $t \to \infty$, and for $0 < t < \infty$, it holds with equality if and only if $\varphi_1 \perp \varphi_2$, i.e., $s_1 s_2 = 0$. For $r \in (0,1]$, the lower bound holds with equality in the limit $t \to 0$.
\end{proposition}
\begin{proof}
The upper bound follows by using \eqref{eq:EUpperModularBound} and \lemmaref{lemma:GraphBound} applied to $A = t \Delta_{\varphi_2,\varphi_1}$ and $\xi = \Omega_1$, and noting that $\eta \defeq p \, \Omega_1 \in \dom \Delta_{\varphi_2,\varphi_1}^\half$ for all $p \in \PM$:
\begin{splitequation}
\label{eq:EstimateUpper1}
E^\M_t(\varphi_1 \vert \varphi_2) &\leq \varphi_1(\1) - \inf_{\eta \in \dom \Delta_{\varphi_2,\varphi_1}^\half} \left[ \norm{ \Omega_1 - \eta }^2 + t \norm{ \Delta_{\varphi_2,\varphi_1}^\half \eta }^2 \right] \\
&= \varphi_1(\1) - \left\langle \Omega_1, \frac{t \Delta_{\varphi_2,\varphi_1}}{\1 + t \Delta_{\varphi_2,\varphi_1}} \Omega_1 \right\rangle = \left\langle \Omega_1, \left( \1 + t \Delta_{\varphi_2,\varphi_1} \right)^{-1} \Omega_1 \right\rangle \eqend{.}
\end{splitequation}
The lower bound is derived from \eqref{eq:PSO}. According to \eqref{eq:EFormulaDiffNorm}, we have
\begin{splitequation}
E^\M_t(\varphi_1 \vert \varphi_2) &= \frac{1}{2} \Bigl[ \varphi_1(\1) - t \varphi_2(\1) + \norm{ \varphi_1 - t \varphi_2 } \Bigr] \\
&= \varphi_1(\1) - \frac{1}{2} \Bigl[ \varphi_1(\1) + t \varphi_2(\1) - \norm{ \varphi_1 - t \varphi_2 } \Bigr] \\
&\geq \varphi_1(\1) - \left\langle \Omega_1, \Delta_{t \varphi_2,\varphi_1}^r \Omega_1 \right\rangle \eqend{,}
\end{splitequation}
which coincides with the claimed lower bound because $\Delta_{t \varphi_2,\varphi_1} = t \Delta_{\varphi_2,\varphi_1}$ by \lemmaref[item:RelativeDeltaScaling]{lemma:RelativeDelta}.

We now consider the limiting cases $t \to 0$ and $t \to \infty$. For $t \to 0$, the lower bound goes to $\varphi_1(\1)$. Also $E^\M_t(\varphi_1 \vert \varphi_2) = \norm{ (\varphi_1 - t \varphi_2)_+ } \to \norm{ (\varphi_1)_+ } = \varphi_1(\1)$, and the upper bound goes to $\left\langle \Omega_1, \left( \1 + t \Delta_{\varphi_2,\varphi_1} \right)^{-1} \Omega_1 \right\rangle = \int_0^\infty (1 + t \lambda)^{-1} \total \mu_{12}(\lambda) \to \int_0^\infty \total \mu_{12}(\lambda) = \varphi_1(\1)$ by dominated convergence. Similarly, in the limit $t \to \infty$ we have $E^\M_t(\varphi_1 \vert \varphi_2) \to \varphi_1(s_2^\perp)$ by \propref[item:EBoundsLimit]{prop:hockeystick_properties}, and for the upper bound $\int_0^\infty (1 + t \lambda)^{-1} \total \mu_{12}(\lambda) \to \mu_{12}\left( \{0\} \right) = \varphi_1(s_2^\perp)$.

It remains to show that for $0 < t < \infty$, the upper bound is not sharp unless $\varphi_1 \perp \varphi_2$. So assume that the upper bound holds with equality. Then the two estimates that we did, the inequality \eqref{eq:EUpperModularBound} and the inequality \eqref{eq:EstimateUpper1}, must hold as equalities. The infimum defining the hockeystick divergence is attained at the support projection $p_t$ of $( \varphi_1 - t \varphi_2 )_+$ by \propref{prop:hockeystick_properties}, and the upper bound in \eqref{eq:EUpperModularBound} derives from \lemmaref{lemma:GraphBound}, which is attained at the unique vector $\eta = \left( \1 + t \Delta_{\varphi_2,\varphi_1} \right)^{-1} \Omega_1$ coming from the projection onto the graph of $\Delta_{t \varphi_2,\varphi_1}^\half = t^\half \Delta_{\varphi_2,\varphi_1}^\half$. Hence this implies $p_t \, \Omega_1 = \left( \1 + t \Delta_{\varphi_2,\varphi_1} \right)^{-1} \Omega_1$. As $p_t$ is a projection, this further implies
\begin{splitequation}
0 &= \langle \Omega_1, p_t \, \Omega_1 \rangle - \norm{ p_t \, \Omega_1 }^2 = \left\langle \Omega_1, \left( \1 + t \Delta_{\varphi_2,\varphi_1} \right)^{-1} \Omega_1 \right\rangle - \norm{ \left( \1 + t \Delta_{\varphi_2,\varphi_1} \right)^{-1} \Omega_1 }^2 \\
&= \int_0^\infty \left[ \frac{1}{1 + t \lambda} - \frac{1}{(1 + t \lambda)^2} \right] \total \mu_{12}(\lambda) = \int_0^\infty \frac{t \lambda}{(1 + t \lambda)^2} \total \mu_{12}(\lambda) \eqend{.}
\end{splitequation}
Since the integrand is strictly positive for $t > 0$, we conclude $\mu_{12}((0,\infty)) = 0$, which is equivalent to $\Omega_1 \in \ker \Delta_{\varphi_2,\varphi_1}$ and $\varphi_1(s_2) = \varphi_1(s_1 s_2) = 0$, i.e. $\varphi_1 \perp \varphi_2$.
\end{proof}

This proposition establishes upper and lower bounds on hockeystick divergences, but also shows that these bounds are not sharp in general. As is apparent from the arguments given above, the error made in the upper estimate is due to replacing the infimum over $\dom \Delta_{\varphi_2,\varphi_1}^\half$, describing the distance of $\Omega_1$ to the graph of $t^\half \Delta_{\varphi_2,\varphi_1}^\half$, by an infimum over the smaller set of vectors given by projections $p \in \PM$ acting on $\Omega_1$. This would introduce no error if the projections in \eqref{eq:EUpperModularBound} were replaced by general elements $x \in \M$ (for cyclic $\Omega_1$). But for general $x \in \M$,
\begin{equation}
\norm{ \Omega_1 - x \, \Omega_1 }^2 + t \norm{ \Delta_{\varphi_2,\varphi_1}^\half x \, \Omega_1 }^2 = \varphi_1\Bigl( (\1 - x)^* (\1 - x) \Bigr) + t \, \varphi_2( x x^* )
\end{equation}
is no longer of hockeystick form because of the quadratic expressions $\abs{\1 - x }^2$ and $\abs{ x^* }^2$. Rather, it is the main ingredient into Kosaki's variational formula for relative entropies \cite{kosaki1986,ohyapetz1993}. It therefore seems difficult to establish the equality $D^\M_{f_0}(\varphi_1 \vert \varphi_2) = S^\M_\mathrm{rel}(\varphi_1 \vert \varphi_2)$ with these methods, and we will prove it via a different route in \secref{sec:operatorint}.

Nonetheless, interesting upper and lower bounds on $f$-divergences can be obtained. We first point out the modular bounds on $f$-divergences that follow from the hockeystick bounds in \propref{prop:EModularEstimate} by integration.
\begin{lemma}
\label{lemma:fdivergence_upperbound}
Let $f \colon (0,\infty) \to \mathbb{R}$ be a convex function, and assume that for $\lambda > 0$, the integral
\begin{equation}
F_\epsilon(\lambda) \defeq \int_\epsilon^\infty \frac{f''(t)}{1 + t \lambda} \total t
\end{equation}
is finite. Then we have the bound
\begin{splitequation}
\label{eq:DModularBound}
D^\M_f(\varphi_1 \vert \varphi_2) \leq \lim_{\epsilon \searrow 0} \Biggl( \left\langle \Omega_1, F_\epsilon\left( \Delta_{\varphi_2,\varphi_1} \right) \Omega_1 \right\rangle &+ \varphi_2(\1) \Bigl[ f(\epsilon) - f(1) + f'(\epsilon) - \epsilon f'(\epsilon) \Bigr] \\
&+ \Bigl[ \varphi_1(\1) - \varphi_2(\1) \Bigr] \Bigl[ 1 + f'(\epsilon) - f'(1) \Bigr] \Biggr) \eqend{.}
\end{splitequation}
\end{lemma}
\begin{proof}
We use the upper bound from \propref{prop:EModularEstimate} and the formula from \thmref[item:DAlternativeFormula]{thm:fdivergence_properties} to obtain
\begin{splitequation}
D^\M_f(\varphi_1 \vert \varphi_2) &\leq \lim_{\epsilon \searrow 0} \Biggl[ \int_\epsilon^\infty f''(t) \left\langle\Omega_1, \left( \1 + t \Delta_{\varphi_2,\varphi_1} \right)^{-1} \Omega_1 \right\rangle \total t \\
&\hspace{4em}+ \varphi_2(\1) \Bigl[ f(\epsilon) - f(1) + f'(\epsilon) - \epsilon f'(\epsilon) \Bigr] \\
&\hspace{4em}+ \Bigl[ \varphi_1(\1) - \varphi_2(\1) \Bigr] \Bigl[ 1 + f'(\epsilon) - f'(1) \Bigr] \Biggr] \eqend{.}
\end{splitequation}
Using Tonelli's theorem, the integral can be rewritten as
\begin{equation}
\int_\epsilon^\infty f''(t) \left\langle \Omega_1, \left( \1 + t \Delta_{\varphi_2,\varphi_1} \right)^{-1} \Omega_1 \right\rangle \total t = \int_\epsilon^\infty \int_0^\infty \frac{f''(t)}{1 + t \lambda} \total \mu_{12}(\lambda) \total t = \left\langle \Omega_1, F_\epsilon\left( \Delta_{\varphi_2,\varphi_1} \right) \Omega_1 \right\rangle \eqend{,}
\end{equation}
which proves \eqref{eq:DModularBound}.
\end{proof}

We now concentrate on the most interesting case of the information function $f = f_0$. For $\varphi_1 \not\ll \varphi_2$, by \thmref[item:DBounds]{thm:fdivergence_properties} we have $D^\M_{f_0}(\varphi_1 \vert \varphi_2) = \infty$ because $t \mapsto f_0''(t) = t^{-1}$ is not integrable over $[1,\infty)$, and since in this case also $S^\M_\mathrm{rel}(\varphi_1 \vert \varphi_2) = \infty$ \eqref{eq:Srel}, we can restrict to the case $\varphi_1 \ll \varphi_2$.
\begin{theorem}
\label{thm:ModularBoundsDf0}
Assume in addition that $0 \neq \varphi_1 \ll \varphi_2$.
\begin{enumerate}
\item \label{item:DModBounds1} For any $r \in (0, 1)$, we have the lower bound
\begin{equation}
\label{eq:DModBounds1_lower}
- \frac{\delta_0}{r} \ln\left( \frac{\delta_r}{\delta_0} \right) - \frac{1-r}{r} \left( \delta_0 - \delta_1^{-\frac{r}{1-r}} \delta_r^\frac{1}{1-r} \right) \leq D^\M_{f_0}(\varphi_1 \vert \varphi_2) \eqend{,} \quad \delta_r \defeq \langle \Omega_1, \Delta_{\varphi_2,\varphi_1}^r \Omega_1 \rangle \eqend{,}
\end{equation}
as well as the upper bound
\begin{equation}
\label{eq:DModBounds1_upper}
D^\M_{f_0}(\varphi_1 \vert \varphi_2) \leq S^\M_\mathrm{rel}(\varphi_1 \vert \varphi_2) + \varphi_1(\1) \eqend{.}
\end{equation}

\item \label{item:DModBounds2} For $\varphi_1(\1) = \varphi_2(\1) = 1$, one has
\begin{equation}
\lim_{n \to \infty} \frac{1}{n} D^{\M^{\otimes n}}_{f_0}\left( \varphi_1^{\otimes n} \vert \varphi_2^{\otimes n} \right) = S^\M_\mathrm{rel}(\varphi_1 \vert \varphi_2) \eqend{.}
\end{equation}
\end{enumerate}
\end{theorem}
\begin{proof}
\ref{item:DModBounds1} We begin with the upper bound \eqref{eq:DModBounds1_upper}, which follows from \lemmaref{lemma:fdivergence_upperbound}. For the information function $f = f_0$, we have $f''(t) = t^{-1}$ and therefore, by explicit calculation, $F_\epsilon(\lambda) = - \ln \lambda - \ln \epsilon + \ln( 1 + \epsilon \lambda )$. Inserting this into \eqref{eq:DModularBound} and using monotone convergence yields
\begin{splitequation}
D^\M_{f_0}(\varphi_1 \vert \varphi_2) &\leq \lim_{\epsilon \searrow 0} \int_0^\infty \left[ - \ln \lambda - \ln \epsilon + \ln( 1 + \epsilon \lambda ) \right] \total\mu_{12} + \varphi_1(\1) ( 1 + \ln \epsilon ) \\
&= S^\M_\mathrm{rel}(\varphi_1 \vert \varphi_2) + \varphi_1(\1) \eqend{.}
\end{splitequation}
The lower bound \eqref{eq:DModBounds1_lower} is based on the lower bound in \propref{prop:EModularEstimate}, namely
\begin{equation}
D_{f_0}^\M(\varphi_1 \vert \varphi_2) \geq \lim_{\epsilon \searrow 0} \left[ \int_\epsilon^\infty \frac{1}{t} \left\langle \Omega_1, \left( \1 - t^{r(t)} \, \Delta_{\varphi_2,\varphi_1}^{r(t)} \right) \Omega_1 \right\rangle \total t + \varphi_1(\1) ( 1 + \ln \epsilon ) \right] \eqend{,}
\end{equation}
where $r \colon \mathbb{R}_+ \to [0,1]$ has to be chosen appropriately. To obtain a finite lower bound, one needs to choose $r(t) \to 0$ for $t \to \infty$. For fixed $t$, we want to minimize $t^{r(t)} \left\langle \Omega_1, \Delta_{\varphi_2,\varphi_1}^{r(t)}) \Omega_1 \right\rangle$, which suggests to choose $r(t) \to 1$ for $t \to 0$. We here make the simple choice
\begin{equation}
r(t) = \begin{cases} 1 & t \in [0, T_1] \eqend{,} \\ r & t \in (T_1, T_2] \eqend{,} \\ 0 & t \in (T_2, \infty) \eqend{,} \end{cases}
\end{equation}
where $0 < r < 1$ and $0 < T_1 \leq T_2 \in \mathbb{R}$ are parameters.

Inserting this choice of $r(t)$ in the above integral gives by straightforward calculation
\begin{equation}
D^\M_{f_0}(\varphi_1 \vert \varphi_2) \geq \varphi_1(\1) (1 + \ln T_2) - \varphi_2(s_1) T_1 - (T_2^r - T_1^r) \frac{\delta_r}{r}
\end{equation}
with $\delta_r$ defined in \eqref{eq:DModBounds1_lower}. Since we assumed $\varphi_1 \ll \varphi_2$, we have $s_2 \, \Omega_1 \neq 0$ and therefore $\delta_r > 0$. Furthermore, $\delta_1 = \varphi_2(s_1)$ and $\delta_0 = \varphi_1(s_2) = \varphi_1(\1)$ because $\varphi_1 \ll \varphi_2$.

As a function of $T_1,T_2 > 0$, the maximum of this function (at fixed $r$) is at
\begin{equation}
T_1 = \left[ \frac{\delta_r}{\varphi_2(s_1)} \right]^\frac{1}{1-r} = \left( \frac{\delta_r}{\delta_1} \right)^\frac{1}{1-r} \eqend{,} \quad T_2 = \left[ \frac{\delta_r}{\varphi_1(\1)} \right]^{-\frac{1}{r}} = \left( \frac{\delta_0}{\delta_r} \right)^\frac{1}{r} \eqend{.}
\end{equation}
Because $r \mapsto \delta_r$ is log-convex (this follows from the spectral theorem and H{\"o}lder's Inequality), we have $\delta_r \leq \delta_0^{1-r} \delta_1^r$ and thus $T_1 \leq \delta_0/\delta_1 \leq T_2$. So these values of $T_1, T_2$ are admissible, and inserting them gives the claimed bound \eqref{eq:DModBounds1_lower}.

\ref{item:DModBounds2} We now proceed to tensor powers of $\M$ and $\varphi_1,\varphi_2$ under the assumption $\delta_0 = \varphi_1(\1) = 1$ and $\delta_1 = \varphi_2(s_1) \leq \varphi_2(\1) = 1$. Since $\Delta_{\varphi_2^{\otimes n},\varphi_1^{\otimes n}} = \Delta_{\varphi_2,\varphi_1}^{\otimes n}$, we have the bound
\begin{equation}
\label{eq:AnotherEstimate}
D^{\M^{\otimes n}}_{f_0}\left( \varphi_1^{\otimes n} \vert \varphi_2^{\otimes n} \right) \geq - \frac{n}{r} \ln \delta_r - \frac{1-r}{r} \left( 1 - \delta_1^{-\frac{r n}{1-r}} \delta_r^\frac{n}{1-r} \right) \eqend{.}
\end{equation}
By log-convexity of $\delta_r$, we have $\delta_r \leq \delta_0^{1-r} \delta_1^r = \delta_1^r$, which implies that the last term in brackets in \eqref{eq:AnotherEstimate} lies in $[0,1]$ for all $n \in \mathbb{N}$ and all $r \in [0, 1)$.

We now choose $r = r_n \defeq n^{-\half}$, and divide by $n$, resulting in
\begin{equation}
\frac{1}{n} D^{\M^{\otimes n}}_{f_0}\left( \varphi_1^{\otimes n} \vert \varphi_2^{\otimes n} \right) \geq - \frac{1}{r_n} \ln \delta_{r_n} - \left( \frac{1}{\sqrt{n}} - \frac{1}{n} \right) \left( 1 - \delta_1^{-\frac{n r_n}{1-r_n}} \delta_{r_n}^{\frac{n}{1-r_n}} \right) \eqend{.}
\end{equation}
By the boundedness mentioned above, the second term converges to $0$ as $n \to \infty$. As $\delta_0 = 1$, the first term converges to the derivative
\begin{splitequation}
\lim_{n \to \infty} \left[ - \frac{1}{r_n} \ln \left\langle \Omega_1, \Delta_{\varphi_2,\varphi_1}^{r_n} \Omega_1 \right\rangle \right] &= - \frac{\total}{\total r} \Bigr\rvert_{r = 0} \ln \left\langle \Omega_1, \Delta_{\varphi_2,\varphi_1}^r \Omega_1 \right\rangle \\
&= - \left\langle \Omega_1, \ln \Delta_{\varphi_2,\varphi_1} \Omega_1 \right\rangle = S^\M_\mathrm{rel}(\varphi_1 \vert \varphi_2) \eqend{.}
\end{splitequation}

On the other hand, the upper bound for tensor powers is
\begin{equation}
\frac{1}{n} D^{\M^{\otimes n}}_{f_0}\left( \varphi_1^{\otimes n} \vert \varphi_2^{\otimes n} \right) \leq - \frac{1}{n} \left\langle \Omega_1^{\otimes n}, \ln \Delta_{\varphi_2,\varphi_1}^{\otimes n} \Omega_1^{\otimes n} \right\rangle + \frac{1}{n} = S^\M_\mathrm{rel}(\varphi_1 \vert \varphi_2) + \frac{1}{n} \eqend{.}
\end{equation}
The combination of these two bounds gives the claimed result for the limit $n \to \infty$.
\end{proof}

As mentioned in the introdcution, passing to tensor powers means that in the context of hypothesis testing, one is handed multiple $(n)$ copies of the two states of the system and may now test the two hypotheses (system in state $\varphi_1$ or $\varphi_2$) using test projections from the larger algebra $\M^{\otimes n}$. Limits of the form $\lim_{n \to \infty} D^{\M^{\otimes n}}_{f_0}\left( \varphi_1^{\otimes n} \vert \varphi_2^{\otimes n} \right)$ often describe asymptotic decay rates of testing errors of such ``multi-shot'' tests, and are also called regularizations of $D^\M_{f_0}(\varphi_1 \vert \varphi_2)$. We have thus shown that the multi-shot regularization of the $f_0$-divergence coincides with Araki's relative entropy. Later we will see by different methods that these two quantities coincide also without regularization.

The lower bound in \thmref[item:DModBounds1]{thm:ModularBoundsDf0} bounds the $f_0$-divergence by R{\'e}nyi-type relative entropies. One may check that for states, the lower bound becomes trivial (converges to $0$) in the limits $r \to 0$ and $r \to 1$. It was chosen here mainly because of its good behaviour under tensor powers, which allowed us to identify the regularization of the $f_0$-divergence. We note that the technique of producing lower bounds underlying this theorem can also be applied to general $f$-divergences, but do not provide details here.

In \cite{hirchetomamichel2024}, a related analysis is carried out in the matrix case for R{\'e}nyi entropies and their multi-shot regularizations. In that case, another function $f \neq f_0$ appears whose integrability properties of $f''$ are better than the ones of $f_0''$, so that one can work with a fixed power $r$ (depending on the R{\'e}nyi parameter $\alpha$) instead of the $t$-dependent $r$ used in our proof. In our approach, different choices of the function $r(t)$ can be used to obtain lower bounds different from the one in \thmref[item:DModBounds1]{thm:ModularBoundsDf0}, but we refrain from giving more examples here.

\subsection{Semifinite von Neumann algebras}
\label{sec:semifinite}

In this section, we consider a setting which encompasses all examples given in \exref{example:Jordan}: a von Neumann algebra $\M$ with a faithful normal semifinite tracial weight.

\begin{definition}
Let $\M$ be a von Neumann algebra. A weight is a function $\varphi \colon \M_+ \to [0,\infty]$ satisfying $\varphi(x + \lambda y) = \varphi(x) + \lambda \varphi(y)$ for all $x,y \in \M_+$ and $\lambda \geq 0$. Furthermore, a weight is said to be:
\begin{enumerate}
\item Faithful if $\varphi(x) = 0 \implies x = 0$,
\item Normal if $\varphi\left( \sup_{i \in I} x_i \right) = \sup_{i \in I} \varphi(x_i)$ for every bounded increasing net $(x_i)_{i \in I} \subset \M_+$,
\item Semifinite if for any $0 \neq x \in \M_+$ there exists $0 < y \leq x$ with $\varphi(y) < \infty$, i.e., the set $\operatorname{span}\{ x \in \M_+ \colon \varphi(x) < \infty \}$ is $\sigma$-weakly dense in $\M$, and
\item Tracial if $\varphi\left( u x u^* \right) = \varphi(x)$ for all $x \in \M_+$ and all unitary $u \in \M$.
\end{enumerate}
\end{definition}
For a normal weight, the characterization $\varphi(x) = \sup \{ \varphi(y) \colon \ 0 \leq y \leq x \eqend{,} \ \varphi(y) < \infty \}$ is equivalent to semifiniteness \cite[Cor.~I.6.3]{dixmier1981}. We extend $\varphi$ by linearity to $\M_\varphi \defeq \operatorname{span} \{ x \in \M_+ \colon \varphi(x) < \infty \} \subset \M$.

Given a normal faithful semifinite tracial weight, we define noncommutative $L^p$ spaces as the natural generalization of classical $L^p$ spaces of measurable functions, with the trace playing the role of the measure~\cite[Thm.~IX.2.13]{takesaki2003}:
\begin{definition}
\label{def:noncommutative_lp_space}
A semifinite von Neumann algebra $(\M, \tau)$ is a von Neumann algebra $\M$ with a faithful normal semifinite tracial weight $\tau$. For $1 \leq p < \infty$, we denote by $L^p(\M,\tau)$ the completion of $\M_\tau \defeq \operatorname{span}\{ x \in \M_+ \colon \tau(x) < \infty \} \subset \M$ with respect to the norm $\norm{ x }_{L^p(\M,\tau)} \defeq \tau\left( \abs{x}^p \right)^\pinv$.
\end{definition}
Note that $L^p(\M,\tau)$ also contains unbounded operators affiliated with $\M$; for a detailed description of these spaces in terms of affiliated $\tau$-measurable operators, see~\cite{terp1981} and~\cite{correadasilva2018}. Moreover, for positive self-adjoint operators affiliated with $\M$ (not necessarily in $L^p(\M,\tau)$) we have~\cite[Ex.~IX.4.5 and Cor.~IX.4.9]{takesaki2003}
\begin{lemma}
\label{lemma:affiliated_tau_convergence}
Consider the extended positive cone $\widehat{\M}_+$ of $\M$ consisting of positive operators affiliated with $\M$. Any normal weight $\varphi \in \M\stp$ has a unique extension to $\widehat{\M}_+$ given by $\varphi(h) \defeq \int_0^\infty \lambda \total \varphi\left( E^h(\lambda) \right) \in [0,\infty]$. If $( h_i )_{i \in I}$, $h_i \in \widehat{\M}_+$ for all $i \in I$ is an increasing net, in the sense that $( \varphi(h_i) )_{i \in I}$ is increasing for all weights $\varphi$, it holds that $\varphi( \sup_{i \in I} h_i ) = \sup_{i \in I} \varphi(h_i)$.
\end{lemma}

It is clear that $L^p(\M,\tau)$ is invariant under taking adjoints, polar decomposition, and multiplication by elements from $\M$ from the left and right, the latter following again using the noncommutative H{\"o}lder inequality~\cite[Thm.~2.1.4]{correadasilva2018}. In addition, also due to the noncommutative H{\"o}lder inequality, it is easy to see that $x \in \M_\tau$ implies $x\in L^p(\M,\tau)$ for all $p \geq 1$ and $\lim_{p \to \infty} \norm{x}_{L^p(\M,\tau)} = \norm{ x }$. We therefore define $L^\infty(\M,\tau) \defeq \M$ with $\norm{ x }_{L^\infty(\M,\tau)} \defeq \norm{ x }$, due to the fact that clearly $\tau$ extends to $L^1(\M,\tau)$, and it holds that~\cite[Thm.~VIII.3.14, Lemma~IX.2.12]{takesaki2003}, \cite[Prop.~5.3.11]{pedersen2018}, \cite[Thm.~2.2.21]{correadasilva2018}:
\begin{theorem}
\label{thm:noncommutativeRadon-Nikodym}
Let $(\M,\tau)$ be a semifinite von Neumann algebra.
\begin{enumerate}
\item \label{item:RNIsomorphism} The map
\begin{equation}
L^1(\M,\tau) \to \M_* \eqend{,} \quad h \mapsto \varphi_h \quad\text{with}\quad \varphi_h(x) \defeq \tau(h x)
\end{equation}
is an isometric isomorphism. The support projection of $\varphi_h$ in $\M$ is given by $s^\M(\varphi_h) = \1 - E^h_{\{0\}}$.
\item \label{item:RNIsomorphismJordan} A functional $\varphi_h \in \M_*$ is selfadjoint (positive) if and only if $h$ is selfadjoint (positive). For positive $h$, it holds that
\begin{align}
\varphi_h(x) = \tau(h x) = \tau\left( h^\half x h^\half \right) \eqend{.}
\end{align}
For selfadjoint $h$, let $h = h^+ - h^-$ be the spectral decomposition of $h$ into positive and negative parts defined by $h^+ \defeq E^h_{(0,\infty)} h$ and $h^- \defeq - E^h_{(-\infty,0)} h$. The Jordan decomposition of $\varphi_h$ is then given by
\begin{equation}
\varphi_h = \varphi_h^+ - \varphi_h^- \quad\text{with}\quad \varphi_h^\pm(x) \defeq \tau( h^\pm x ) \eqend{.}
\end{equation}
\end{enumerate}
\end{theorem}
\begin{proof}
For completeness, we give a proof of~\ref{item:RNIsomorphismJordan}. Given $h \in L^1(\M,\tau)$ and $x \in \M_+$, we directly see that $\varphi_{h^*}(x) = \tau(h^* x) = \varphi_h^*(x)$. The positivity claims are clear, and the formula $\varphi_h(x) = \tau\left( h^\half x h^\half \right)$ follows from the cyclicity of $\tau$. Finally, by \thmref{theorem:JordanDecomposition} the Jordan decomposition is uniquely determined by the orthogonality of the corresponding support projections $s^\M(\varphi_h^\pm) = \1 - E^{h^\pm}_{\{0\}}$ of $\varphi_h^\pm$, which follows from
\begin{equation}
\left( \1 - E^{h^+}_{\{0\}} \right) \left( \1 - E^{h^-}_{\{0\}} \right) = \left( \1 - E^h_{(-\infty,0]} \right) \left( \1 - E^h_{\mathbb{R}_+} \right) = \1 - E^h_{(-\infty,\infty)} = 0 \eqend{.}
\end{equation}
\end{proof}

\paragraph{Notation.} For the remaining part of this section, for a positive operator $h \geq 0$ we will denote $h^{-p} \defeq \int_{(0,\infty)} \lambda^{-p} \total E^h_\lambda$ for any $p > 0$ and $\ln h \defeq \int_{(0,\infty)} \ln(\lambda) \total E^h_\lambda$. From the definition, it is clear that $h h^{-1} = 1 - E^h_{\{0\}} = s^\M(\varphi_h)$.

For a semifinite von Neumann algebra, the relative modular operator from \defref{def:RelativeS} has an explicit form (see, e.g., \cite[Thm.~8']{luczakpodsedkowskawieczorek2021}):
\begin{theorem}
\label{thm:semifinite_relative_modular_operator}
Let $(\M,\tau)$ be a semifinite von Neumann algebra, and $\varphi,\psi \in \M\stp$. Then the relative modular operator $\Delta_{\varphi,\psi}$ satisfies
\begin{equation}
\Delta_{\varphi,\psi}^\half \eta = L\left( h_\varphi^\half \right) R\left( h_\psi^{-\half} \right) \eta = h_\varphi^\half \eta \, h_\psi^{-\half}
\end{equation}
for all $\eta \in \dom \Delta_{\varphi,\psi}^\half$, where $L$ and $R$ denote left and right multiplication.
\end{theorem}
\begin{proof}
We consider the noncommutative $L^p$ space $L^2(\M,\tau)$ from \defref{def:noncommutative_lp_space}, which becomes a Hilbert space when equipped with the scalar product $\left\langle x, y \right\rangle_{L^2(\M,\tau)} \defeq \tau( x^* y )$. $\M$ is represented on $L^2(\M,\tau)$ by left multiplication $L(x) y \defeq x y$, and an antirepresentation is obtained by right multiplication $R(x) y \defeq y x$, both for $x \in \M$ and $y \in L^2(\M,\tau)$. Both are normal and faithful and it holds that $L(\M)' = R(\M)$~\cite[Thm.~V.2.22]{takesaki2001}. The standard representation of $\M$ on $L^2(\M,\tau)$ is the one by left multiplication, the natural positive cone consists of positive operators $x \in L^2(\M,\tau)$, and the antilinear conjugation $J$ exchanges left and right multiplication: $J L(x) J = R(x^*)$ for $x \in \M$. By \thmref{thm:noncommutativeRadon-Nikodym}, there exist positive operators $h_\varphi, h_\psi \in L^1(\M,\tau)$ such that the functionals $\varphi,\psi$ are given by $\varphi(x) = \tau\left( h_\varphi^\half x h_\varphi^\half \right)$ and $\psi(x) = \tau\left( h_\psi^\half x h_\psi^\half \right)$. Since $h_\varphi^\half, h_\psi^\half \in L^2(\M,\tau)$, it follows that the representing vectors for $\varphi$ and $\psi$ in the natural positive cone are given by $h_\varphi^\half$ and $h_\psi^\half$. Therefore, the relative Tomita operator (\defref{def:RelativeS}) is given by the closure of
\begin{equation}
S^0_{\varphi,\psi} \left( L(x) h_\psi^\half + \xi \right) = s^{L(\M)}(\psi) L(x^*) h_\varphi^\half \eqend{,} \quad x \in \M \eqend{,} \quad \xi \in \left[ L(\M) h_\psi^\half \right]^\perp \eqend{,}
\end{equation}
where we observe that $s^{L(\M)}(\psi) = L(s^\M(\psi))$.

Now on one hand, for any $x \in \M$ we have $L(x) h_\psi^\half \in \dom\left( L\left( h_\varphi^\half \right) R\left( h_\psi^{-\half} \right) \right)$, and for any $0 \leq h \in L^1(\M,\tau)$ we obtain
\begin{splitequation}
\label{eq:modularcalc1}
\left\langle h^\half, L\left( h_\varphi^\half \right) R\left( h_\psi^{-\half} \right) x h_\psi^\half \right\rangle_{L^2(\M,\tau)} &= \left\langle h^\half, h_\varphi^\half x s^\M(\psi) \right\rangle_{L^2(\M,\tau)} \\
&= \tau\left( h^\half h_\varphi^\half x s^\M(\psi) \right) \eqend{.}
\end{splitequation}
In addition, for all $\xi \in \left[ L(\M) h_\psi^\half \right]^\perp$, $x \in \M_\tau$ and $n \in \mathbb{N}$ we have that
\begin{splitequation}
&\left\langle L\left( h_\varphi^\half \right) R\left( h_\psi^{-\half} \right) \xi, E^{h_\varphi}_{[0,n]} x E^{h_\psi}_{[n^{-1},\infty)} \right\rangle_{L^2(\M,\tau)} \\
&\quad= \tau\left( h_\psi^{-\half} \xi^* h_\varphi^\half E^{h_\varphi}_{[0,n]} x E^{h_\psi}_{[n^{-1},\infty)} \right) = \left\langle \xi, h_\varphi^\half E^{h_\varphi}_{[0,n]} x E^{h_\psi}_{[n^{-1},\infty)} h_\psi^{-\half} \right\rangle_{L^2(\M,\tau)}\\
&\quad= \left\langle \xi, h_\varphi^\half E^{h_\varphi}_{[0,n]} x h_\psi^{-1} E^{h_\psi}_{[n^{-1},\infty)} h_\psi^\half \right\rangle_{L^2(\M,\tau)} = 0 \eqend{,}
\end{splitequation}
which implies $L\left( h_\varphi^\half \right) R\left( h_\psi^{-\half}\right) \xi = 0$ since $\bigcup_{n \in \mathbb{N}} E^{h_\varphi}_{[0,n]} \M_\tau E^{h_\psi}_{[n^{-1},\infty)}$ is dense in $L^2(\M,\tau)$. On the other hand, since $h^\half$ lies in the natural positive cone, we also obtain
\begin{splitequation}
\label{eq:modularcalc2}
\left\langle h^\half, \Delta_{\varphi,\psi}^\half L(x) h_\psi^\half \right\rangle_{L^2(\M,\tau)} &= \left\langle J \Delta_{\varphi,\psi}^\half L(x) h_\psi^\half, J h^\half \right\rangle_{L^2(\M,\tau)} \\
&= \left\langle S_{\varphi,\psi} L(x) h_\psi^\half, h^\half \right\rangle_{L^2(\M,\tau)} \\
&= \left\langle s^{L(\M)}(\psi) L(x^*) h_\varphi^\half, h^\half \right\rangle_{L^2(\M,\tau)} \\
&= \left\langle h_\varphi^\half, x s^\M(\psi) h^\half \right\rangle_{L^2(\M,\tau)} \\
&= \tau\left( h_\varphi^\half x s^\M(\psi) h^\half \right) = \tau\left( h^\half h_\varphi^\half x s^\M(\psi) \right)
\end{splitequation}
for all $x \in \M$. The coincidence of \eqref{eq:modularcalc1} and \eqref{eq:modularcalc2}, together with the fact that $L^2(\M,\tau)$ is spanned by the elements in the natural positive cone, yields
\begin{equation}
\Delta_{\varphi,\psi}^\half \eta = L\left( h_\varphi^\half \right) R\left( h_\psi^{-\half} \right) \eta
\end{equation}
for all $\eta \in \dom S^0_{\varphi,\psi}$. Since $\dom S^0_{\varphi,\psi}$ is a core for both $\Delta_{\varphi,\psi}^\half$ and $L\left( h_\varphi^\half \right) R\left( h_\psi^{-\half} \right)$, the equality of operators follows.
\end{proof}
\begin{corollary}
\label{cor:semifinite_relative_entropy}
Let $\M$ be a von Neumann algebra with a faithful normal semifinite tracial weight $\tau$, and $\varphi,\psi \in \M\stp$ two normal positive linear functionals. If $s^\M(\psi) \geq s^\M(\varphi)$, the relative entropy between $\varphi$ and $\psi$ is given by Umegaki's definition~\cite{umegaki1962}
\begin{equation}
S^\M_\mathrm{rel}(\varphi \vert \psi) = \tau\left( h_\varphi \ln h_\varphi - h_\varphi \ln h_\psi \right) \eqend{,}
\end{equation}
where $h_\varphi, h_\psi \in L^1(\M,\tau)$ are the operators implementing the functionals by \thmref[item:RNIsomorphism]{thm:noncommutativeRadon-Nikodym}.
\end{corollary}
\begin{proof}
For $s^\M(\psi) \geq s^\M(\varphi)$, Araki's relative entropy~\eqref{eq:Srel} is given by
\begin{equation}
S^\M_\mathrm{rel}(\varphi \vert \psi) = \int_0^\infty \ln t \total \norm{ E^{\Delta_{\varphi,\psi}}(t) \Phi }^2 \eqend{,}
\end{equation}
where $\Phi$ is the implementing vector of $\varphi$ in the natural positive cone. By \thmref{thm:semifinite_relative_modular_operator}, we have $\Delta_{\varphi,\psi}^\half = L\left( h_\varphi^\half \right) R\left( h_\psi^{-\half} \right)$. Since left and right multiplication commute, they have a joint spectral calculus and we denote by $\mu$ their joint spectral measure, supported on their joint spectrum $\sigma\left( L(h_\varphi), R(h_\psi) \right) \subset (\mathbb{R}_+)^2$. Using this spectral measure, it holds that
\begin{equation}
L\left( h_\varphi^\half \right) = \int_{(\mathbb{R}_+)^2} \kappa^\half \total \mu(\kappa,\lambda) \eqend{,} \quad R\left( h_\psi^{-\half} \right) = \int_{\mathbb{R}_+ \times (0,\infty)} \lambda^{-\half} \total \mu(\kappa,\lambda) \eqend{.}
\end{equation}
Hence, we obtain
\begin{splitequation}
S^\M_\mathrm{rel}(\varphi \vert \psi) &= \int_0^\infty \ln t \total \left\langle h_\varphi^\half, E^{\Delta_{\psi,\varphi}}(t) h_\varphi^\half \right\rangle_{L^2(\M,\tau)} \\
&= \int_{\mathbb{R}_+ \times (0,\infty)} \ln\left( \kappa \lambda^{-1} \right) \total \norm{ \mu(\kappa,\lambda) h_\varphi^\half }_{L^2(\M,\tau)}^2 \\
&= \tau\left( h_\varphi^\half \left[ L(\ln h_\varphi) - R(\ln h_\psi) \right] h_\varphi^\half \right) \\
&= \tau\left( h_\varphi \ln h_\varphi - h_\varphi \ln h_\psi \right) \eqend{,}
\end{splitequation}
and the conclusion.
\end{proof}

\subsection{Recovering the relative entropy}
\label{sec:Df=Srel}

We now set out to prove $D_{f_0}^\M = S^\M_\mathrm{rel}$ by a method different from the one in \secref{sec:ModularBounds}. We will first consider semifinite von Neumann algebras (\thmref{thm:D=Ssemifinite}) and then extend the result to general von Neumann algebras (\thmref{thm:D=Sgeneral}). In the case of semifinite $(\M, \tau)$, recall that positive normal functionals $\varphi \in \M\stp$ can be identified with their densities $h_\varphi \in L^1(\M,\tau)$ by \thmref[item:RNIsomorphism]{thm:noncommutativeRadon-Nikodym}, and that by \thmref[item:RNIsomorphismJordan]{thm:noncommutativeRadon-Nikodym} the hockeystick divergence takes the form $E^\M_t(\varphi \vert \psi) = \tau\left( (h_\varphi - t \, h_\psi)_+ \right)$.

In the interest of a reasonably short proof of the following result, we have isolated a few operator and trace integral results that are used in this proof in \secref{sec:operatorint}.
\begin{theorem}
\label{thm:D=Ssemifinite}
Let $(\M, \tau)$ be a semifinite von Neumann algebra, $\varphi, \psi \in \M\stp$ and $f_0(t) = t \ln t$ the information function. Then it holds that
\begin{equation}
D^\M_{f_0}(\varphi \vert \psi) = S^\M_\mathrm{rel}(\varphi \vert \psi) \eqend{,}
\end{equation}
where $S^\M_\mathrm{rel}(\varphi \vert \psi)$ is the relative entropy between $\varphi$ and $\psi$ (\defref{def:RelativeEntropy}).
\end{theorem}
\begin{proof}
If $\varphi \not\ll \psi$, by \thmref[item:DBounds]{thm:fdivergence_properties} we have $D_{f_0}^\M(\varphi \vert \psi) = +\infty$ because $\norm{ f_0'' \bigr\rvert_{[1,\infty)} }_{L^1(\M,\tau)} = +\infty$, which agrees with $S^\M_\mathrm{rel}(\varphi \vert \psi) = +\infty$ (\defref{def:RelativeEntropy}). We may therefore restrict to $\varphi \ll \psi$. Restricting to the corner $s^\M(\psi) \M s^\M(\psi)$ then changes neither the relative entropy nor the $f_0$-divergence by \thmref[item:f-divergence_Compressions]{thm:fdivergence_properties}, i.e., $S^\M_\mathrm{rel}(\varphi \vert \psi) = S^{s^\M(\psi) \M s^\M(\psi)}_\mathrm{rel}(\varphi \vert \psi)$ and $D^\M_{f_0}(\varphi \vert \psi) = D^{s^\M(\psi) \M s^\M(\psi)}_{f_0}(\varphi \vert \psi)$, because $\varphi \ll \psi$ is supported on this corner. As $h_\psi$ is injective on this corner, we may without loss of generality assume $h_\psi$ to be injective.

As an approximation to $f_0$, consider the family of functions $f_a \colon \mathbb{R}_+ \to \mathbb{R}_+$ with $a \in (0,1)$, defined by
\begin{equation}
\label{eq:fa}
f_a(t) \defeq \frac{(t+a) \ln(t+a)}{1-a^2} - \frac{(1+a t) \ln(1+a t)}{a (1-a^2)} + t \eqend{.}
\end{equation}
One easily checks the pointwise limit $\lim_{a \searrow 0} f_a(t) = t \ln t = f_0(t)$ and the second derivatives of $f_a$ and its conjugate $(f_a)_*$
\begin{equation}
f''_a(t) = \frac{1}{(a+t) (1+a t)} \eqend{,} \quad (f_a)_*''(t) = \frac{1}{t} f_a''(t) \eqend{.}
\end{equation}
As $a \searrow 0$, we have $f_a''(t) \to t^{-1}$ and $(f_a)_*''(t) \to t^{-2}$ pointwise on $(0,\infty)$ and monotonically from below. By the monotone convergence theorem applied to the integral representation \eqref{eq:Df}, we thus obtain
\begin{equation}
\lim_{a \to 0^+} D^\M_{f_a}(\varphi \vert \psi) = D^\M_{f_0}(\varphi \vert \psi) \eqend{.}
\end{equation}
Since $f_a'', (f_a)_*'' \in L^1([1,\infty))$, \thmref[item:DContinuity]{thm:fdivergence_properties} states that $D^\M_{f_a} \colon \M_* \times \M_* \to \mathbb{R}$ is norm continuous. We may therefore pass to sequences $\{ \varphi_n \}_{n \in \mathbb{N}}$, $\{ \psi_n \}_{n \in \mathbb{N}}$ of normal positive linear functionals converging in norm to $\varphi$ and $\psi$.

We choose
\begin{equation}
\label{eq:varphi_psi_n_def}
\varphi_n(x) \defeq \tau\left( x \, E^{h_\varphi}_{[0,n]} h_\varphi \right) \eqend{,} \quad \psi_n(x) \defeq \tau\left( x \, E^{h_\psi}_{[n^{-1},n]} h_\psi \right)
\end{equation}
and note that the densities $h_{\varphi,n}$ and $h_{\psi,n}$ are bounded, and $h_{\psi,n}^{-1}$ is bounded as well. Since $\varphi(x) - \varphi_n(x) = \tau\left( x \, E^{h_\varphi}_{(n,\infty)} h_\varphi \right) = \tau\left( x^\frac{1}{2} E^{h_\varphi}_{(n,\infty)} h_\varphi \, x^\frac{1}{2} \right) \geq 0$ for all $x \in \M_+$, $\varphi - \varphi_n$ is a normal positive linear functional and it holds that $\norm{ \varphi - \varphi_n } = \varphi(\1) - \varphi_n(\1)$. Using \lemmaref{lemma:affiliated_tau_convergence}, we obtain
\begin{equation}
\lim_{n \to \infty} \varphi_n(\1) = \lim_{n \to \infty} \tau\left( E^{h_\varphi}_{[0,n]} h_\varphi \right) = \tau(h_\varphi) = \varphi(\1) \eqend{.}
\end{equation}
The same argument applies to $\psi_n$, using that $\lim_{n \to \infty} E^{h_\psi}_{[n^{-1},n]} h_\psi = E^{h_\psi}_{(0,\infty)} h_\psi = h_\psi$ since $h_\psi$ is assumed to be injective. Thus $\varphi_n \to \varphi$ and $\psi_n \to \psi$ in norm.

Using formula~\eqref{eq:f-divergenceFullIntegralFormula}, we obtain
\begin{splitequation}
D^\M_{f_a}(\varphi_n \vert \psi_n) &= \lim_{\epsilon \searrow 0} \Biggl[ \int_\epsilon^\infty f_a''(t) E^\M_t(\varphi_n \vert \psi_n) \total t + \psi_n(\1) [ f_a(\epsilon) - f_a(1) + f_a'(\epsilon) - \epsilon f_a'(\epsilon) ] \\
&\hspace{6em}+ [ \varphi_n(\1) - \psi_n(\1) ] [ 1 + f_a'(\epsilon) - f_a'(1) ] \Biggr] \\
&= \int_0^\infty f_a''(t) \tau\left( ( h_{\varphi_n} - t h_{\psi_n} )_+ \right) \total t + \varphi_n(\1) \left[ 1 + \frac{\ln a}{1-a^2} \right] \\
&\quad+ \psi_n(\1) \left[ \frac{\ln(1+a)}{a} + \frac{a \ln a}{1-a^2} - 1 \right] \eqend{.}
\end{splitequation}
We now make use of \propref{prop:operatorint_ihab_jhab_taux_f} (with $x = \1$), according to which the above hockeystick integral equals
\begin{splitequation}
&\int_0^\infty f_a''(t) E^\M_t(\varphi_n \vert \psi_n) \total t = \int_0^\infty f_a''(t) \tau\left( ( h_{\varphi_n} - t h_{\psi_n} )_+ \right) \total t \\
&\hspace{6em}= \int_0^\infty \tau\left( \sqrt{ \frac{s}{h_{\psi_n} + s} } \, \left[ C_{h_{\psi_n},h_{\varphi_n}}(s) \right]^2 f_a''\left( C_{h_{\psi_n},h_{\varphi_n}}(s) \right) \sqrt{ \frac{s}{h_{\psi_n} + s} } \right) \total s
\end{splitequation}
with $C_{A,B}(s) = ( A + s )^{-\frac{1}{2}} B ( A + s )^{-\frac{1}{2}}$. Using resolvent identities, the latter integral can be written in the form
\begin{splitequation}
&\frac{1}{a (1-a^2)} \int_0^\infty s \, \tau\left( \frac{1-a^2}{h_{\psi_n} + s} + \frac{a^2}{h_{\psi_n} + a^{-1} h_{\varphi_n} + s} - \frac{1}{h_{\psi_n} + a h_{\varphi_n} + s} \right) \total s \\
&= \frac{1}{1-a^2} \int_0^\infty \tau\left( h_{\varphi_n} \left[ \frac{1}{h_{\psi_n} + a h_{\varphi_n} + s} - \frac{1}{h_{\psi_n} + a^{-1} h_{\varphi_n} + s} \right] \right) \total s \\
&\quad- \frac{1}{a (1-a^2)} \int_0^\infty \tau\left( h_{\psi_n} \left[ \frac{1}{h_{\psi_n} + s} - \frac{1}{h_{\psi_n} + a h_{\varphi_n} + s} \right] \right) \total s \\
&\quad+ \frac{a}{1-a^2} \int_0^\infty \tau\left( h_{\psi_n} \left[ \frac{1}{h_{\psi_n} + s} - \frac{1}{h_{\psi_n} + a^{-1} h_{\varphi_n} + s} \right] \right) \total s \eqend{.}
\end{splitequation}
For $a \in (0,1)$, the argument of each $\tau$ is a positive operator, and we can therefore use \lemmaref{lemma:operator_log_derivative} to obtain
\begin{splitequation}
&\frac{1}{a (1-a^2)} \int_0^\infty s \, \tau\left( \frac{1-a^2}{h_{\psi_n} + s} + \frac{a^3}{a h_{\psi_n} + a s + h_{\varphi_n}} - \frac{1}{h_{\psi_n} + s + a h_{\varphi_n}} \right) \total s \\
&\quad= \frac{1}{1-a^2} \, \tau\left( h_{\varphi_n} \left[ \ln\left( h_{\psi_n} + a^{-1} h_{\varphi_n} \right) - \ln\left( h_{\psi_n} + a h_{\varphi_n} \right) \right] \right) \\
&\qquad- \frac{1}{a (1-a^2)} \, \tau\left( h_{\psi_n} \Bigl[ \ln\left( h_{\psi_n} + a h_{\varphi_n} \right) - \ln h_{\psi_n} \Bigr] \right) \\
&\qquad+ \frac{a}{1-a^2} \, \tau\left( h_{\psi_n} \left[ \ln\left( h_{\psi_n} + a^{-1} h_{\varphi_n} \right) - \ln h_{\psi_n} \right] \right) \eqend{.}
\end{splitequation}
It follows that
\begin{splitequation}
D^\M_{f_a}(\varphi_n \vert \psi_n) &= \frac{1}{1-a^2} \, \tau\left( h_{\varphi_n} \left[ \ln\left( a h_{\psi_n} + h_{\varphi_n} \right) - \ln\left( h_{\psi_n} + a h_{\varphi_n} \right) \right] \right) \\
&\quad- \frac{1}{a (1-a^2)} \, \tau\left( h_{\psi_n} \Bigl[ \ln\left( h_{\psi_n} + a h_{\varphi_n} \right) - \ln h_{\psi_n} \Bigr] \right) \\
&\quad+ \frac{a}{1-a^2} \, \tau\left( h_{\psi_n} \left[ \ln\left( a h_{\psi_n} + h_{\varphi_n} \right) - \ln h_{\psi_n} \right] \right) + \varphi_n(\1) + \psi_n(\1) \left[ \frac{\ln(1+a)}{a} - 1 \right] \\
&= \frac{1}{1-a^2} S^\M_\mathrm{rel}( \varphi_n \vert \psi_n + a \varphi_n ) - \frac{1}{1-a^2} S^\M_\mathrm{rel}( \varphi_n \vert a \psi_n + \varphi_n ) \\
&\quad+ \varphi_n(\1) + \psi_n(\1) \left[ \frac{\ln(1+a)}{a} - 1 \right] + \frac{1}{a} S^\M_\mathrm{rel}( \psi_n \vert \psi_n + a \varphi_n ) \\
&\quad+ \frac{a}{1-a^2} S^\M_\mathrm{rel}( \psi_n \vert \psi_n + a \varphi_n ) - \frac{a}{1-a^2} S^\M_\mathrm{rel}( \psi_n \vert a \psi_n + \varphi_n ) \eqend{,}
\end{splitequation}
where we expressed the various trace expressions using the relative entropy. Since for all $a \in (0,1)$ the second functional of the various relative entropies always dominates the first, we can use the dominated convergence property of \propref[item:domination]{prop:srel_properties} (with $\lambda = a^{-1}$), and it follows that
\begin{splitequation}
D^\M_{f_a}(\varphi \vert \psi) &= \lim_{n \to \infty} D^\M_{f_a}(\varphi_n \vert \psi_n) \\
&= \frac{1}{1-a^2} S^\M_\mathrm{rel}( \varphi \vert \psi + a \varphi ) - \frac{1}{1-a^2} S^\M_\mathrm{rel}( \varphi \vert a \psi + \varphi ) + \varphi(\1) + \psi(\1) \left[ \frac{\ln(1+a)}{a} - 1 \right] \\
&\quad+ \frac{1}{a} S^\M_\mathrm{rel}( \psi \vert \psi + a \varphi ) + \frac{a}{1-a^2} S^\M_\mathrm{rel}( \psi \vert \psi + a \varphi ) - \frac{a}{1-a^2} S^\M_\mathrm{rel}( \psi \vert a \psi + \varphi ) \eqend{.}
\end{splitequation}
In the limit $a \to 0^+$, the second functional of the various relative entropies is always monotonically decreasing, and using the monotone convergence property of \propref[item:domination]{prop:srel_properties}, we obtain (with $S^\M_\mathrm{rel}(\varphi \vert \varphi) = 0$)
\begin{equation}
D^\M_{f_0}(\varphi \vert \psi) = \lim_{a \to 0^+} D^\M_{f_a}(\varphi \vert \psi) = S^\M_\mathrm{rel}( \varphi \vert \psi ) + \varphi(\1) + \lim_{a \to 0^+} \frac{1}{a} S^\M_\mathrm{rel}( \psi \vert \psi + a \varphi ) \eqend{.}
\end{equation}
To compute the limit of the last term, we write
\begin{equation}
\lim_{a \to 0^+} \frac{1}{a} S^\M_\mathrm{rel}( \psi \vert \psi + a \varphi ) = \lim_{a \to 0^+} \frac{1}{a} \tau\left( h_\psi \ln h_\psi - h_\psi \ln( h_\psi + a h_\varphi ) \right) = - \tau( h_\varphi ) = - \varphi(\1) \eqend{,}
\end{equation}
where we used \lemmaref{lemma:operator_log_derivative} in the second equality, and the conclusion follows.
\end{proof}

Having established $D^\M_{f_0} = S^\M_\mathrm{rel}$ for semifinite $\M$, we now extend it to general von Neumann algebras. The idea is to approximate a general von Neumann algebra $\M$ by suitably chosen semifinite or even finite ones (Haagerup reduction). We gratefully acknowledge that Lauritz van Luijk has shown us this argument.

Recall that given a von Neumann algebra $\M$ and an abelian locally compact group $G$, we can associate to each $\varphi \in \M\stp$ a canonical \emph{dual functional} $\widehat{\varphi} \in (\M \rtimes G)\stp$ \cite[Def.~3.1]{haagerup1979}, \cite[Def.~9.23]{goldsteinlabuschagne2025}. Furthermore, the map $\varphi \mapsto \widehat{\varphi}$ is a bijection between $\M\stp$ and those positive normal functionals on the crossed product $\M \rtimes G$ that are invariant under the dual action \cite[Thm.~9.24]{goldsteinlabuschagne2025}. If $G$ is discrete, the dual weight construction also yields a faithful normal conditional expectation $C$ from $\M \rtimes G$ onto $\M$ such that $\widehat{\varphi} = \varphi \circ C$ for all $\varphi\in \M\stp$. These ideas lie at the heart of Haagerup's reduction theorem. Given a von Neumann algebra $\M$, one considers the crossed product $\M \rtimes \mathbb{Q}_D$, where $\mathbb{Q}_D$ denotes the group of dyadic rationals endowed with the discrete topology, and the action of $\mathbb{Q}_D$ on $\M$ is given by the restriction of the modular automorphism group of some normal faithful semifinite weight. Defining $\widehat\M \defeq \M \rtimes \mathbb{Q}_D$ and constructing suitable von Neumann subalgebras $\widehat\M_n \subset \widehat\M$, the reduction theorem contains the following result \cite[Thm.~13.1]{goldsteinlabuschagne2025}, \cite[Thm.~2.1]{haagerupjungexu2009}:
\begin{theorem}
\label{thm:HaagerupReduction}
Let $\M$ be a von Neumann algebra. Then there exists a larger von Neumann algebra $\widehat\M \supset \M$ and a sequence $\widehat\M_n \subset \widehat\M$ of von Neumann subalgebras such that
\begin{enumerate}
\item there exists a faithful normal conditional expectation $C \colon \widehat\M \to \M$,
\item each $\widehat\M_n$ is finite, and
\item $\bigcup_{n \in \mathbb{N}} \widehat\M_n \subset \widehat\M$ is dense in the $\sigma$-strong-${}^*$ topology.
\end{enumerate}
\end{theorem}

Note that this theorem is proven for $\sigma$-finite $\M$ in \cite{haagerupjungexu2009}, but extends to the non-$\sigma$-finite case \cite{goldsteinlabuschagne2025}. However, in our application to $D^\M_{f_0}(\varphi \vert \psi)$, using \thmref[item:f-divergence_Compressions]{thm:fdivergence_properties} we could also restrict to $\sigma$-finite $\M$ by cutting down with the support projection of $\varphi + \psi$.
\begin{theorem}
\label{thm:D=Sgeneral}
Let $\M$ be a von Neumann algebra, and $\varphi, \psi \in \M\stp$. Then
\begin{equation}
D^\M_{f_0}(\varphi \vert \psi) = S^\M_\mathrm{rel}(\varphi \vert \psi) \eqend{.}
\end{equation}
\end{theorem}
\begin{proof}
We know that $D_{f_0}^\M$ is invariant under normal conditional expectations in the sense of \propref[item:DDPI]{thm:fdivergence_properties}, and by \propref[item:f-divergenceMartingale]{thm:fdivergence_properties} converges for increasing $\sigma(\M, \M_*)$-dense nets of subalgebras. By \propref[item:Srel_ConditionalExpectation]{prop:srel_properties}, \ref{item:Srel_Martingale} also $S^\M_\mathrm{rel}$ has these two properties. Because $\sigma(\M, \M_*)$- and $\sigma$-strong-${}^*$ closures of ${}^*$-subalgebras agree, we may use \thmref{thm:D=Ssemifinite} and \thmref{thm:HaagerupReduction} to conclude
\begin{splitequation}
D^\M_{f_0}\left( \varphi \vert \psi \right) &= D^{\widehat\M}_{f_0}\left( \varphi \circ C \vert \psi \circ C \right) = \lim_{n \to  \infty} D^{\widehat\M_n}_{f_0}\left( (\varphi \circ C) \rvert_{\widehat\M_n} \Big\vert (\psi\circ C) \rvert_{\widehat\M_n} \right) \\
&= \lim_{n \to \infty} S^{\widehat\M_n}_\mathrm{rel}\left( (\varphi \circ C) \rvert_{\widehat\M_n} \Big\vert (\psi\circ C) \rvert_{\widehat\M_n} \right) = S^{\widehat\M}_\mathrm{rel}\left( \varphi \circ C \vert \psi \circ C \right) = S^\M_\mathrm{rel}(\varphi \vert \psi) \eqend{.}
\end{splitequation}
\end{proof}

\section{Operator integrals}
\label{sec:operatorint}

Two key steps in the proof of \thmref{thm:D=Ssemifinite} rely on results that will be proven only in this section:

\begin{enumerate}
\item[(i)] The equality of the two operator integrals $I_f(A, B)$ and $J_f(A, B)$ defined in \eqref{eq:operatorint_ihab_taux_def}.
\item[(ii)] The identification of $J_f(A,B)$ with the semifinite relative entropy formula for the function $f=f_a$.
\end{enumerate}

This section is organized as follows: In \secref{sec:operatorint_welldefined}, we show that the two integrals are well-defined as Bochner integrals. In \secref{sec:operatorint_equality}, we prove their equality (step (i), see \propref{prop:operatorint_ihab_jhab_taux_f}): First for constant $f$, then for general continuous $f$. The proof for constant $f$ is the essential one. We remark that the proofs in this section use many of the same techniques that were employed by Cheng and Liu~\cite{chengliu2025} in the matrix algebra case. In \secref{sec:operatorint_technical}, we provide the integral representation of the logarithm (step (ii), \lemmaref{lemma:operator_log_derivative}) and collect some technical lemmas that are used repeatedly.

\subsection{Well-definedness of the integrals}
\label{sec:operatorint_welldefined}

\begin{lemma}
\label{lemma:integrability}
Let $A, B \in L^1(\M,\tau) \cap \M_+$, $A \geq a > 0$, and $f \colon \mathbb{R}_+ \to \mathbb{R}$ continuous, and define
\begin{equations}
F^I_{A,B,f}, F^J_{A,B,f} &\colon \mathbb{R}_+ \to L^1(\M,\tau) \cap \M \eqend{,} \\
F^I_{A,B,f}(t) &\defeq f(t) \, (B - t A)_+ \eqend{,} \\
F^J_{A,B,f}(s) &\defeq \sqrt{ \frac{s}{A + s} } \, \left[ C_{A,B}(s) \right]^2 f\left( C_{A,B}(s) \right) \sqrt{ \frac{s}{A + s} } \eqend{,}
\end{equations}
where $C_{A,B}(s) \defeq ( A + s )^{-\half} B ( A + s )^{-\half} \geq 0$.

\begin{enumerate}
\item \label{item:Integrands} The functions $F^I_{A,B,f}$ and $F^J_{A,B,f}$ are well-defined, continuous in $L^1(\M,\tau)$ norm and in the norm of $\M$, and integrable over $\mathbb{R}_+$ in both norms.

\item \label{item:BochnerIntegrals} The integrals
\begin{equation}
\label{eq:operatorint_ihab_taux_def}
I_f(A, B) \defeq \int_0^\infty F^I_{A,B,f}(t) \total t \eqend{,} \quad J_f(A, B) \defeq \int_0^\infty F^J_{A,B,f}(s) \total s
\end{equation}
exist as Bochner integrals in $L^1(\M,\tau)$ and $\M$. They depend only on the restriction of $f$ to the compact interval
\begin{align}
\I_{A,B} \defeq \left[ 0, \norm{ B } \norm{ A^{-1} } \right] \eqend{,}
\end{align}
and depend continuously on $f \in C^0(\I_{A,B})$ (with supremum norm).
\end{enumerate}
\end{lemma}
\begin{proof}
\ref{item:Integrands} For $F^I_{A,B,f}$, we first recall that the map $L^1(\M,\tau)_\mathrm{sa} \ni X \mapsto X_+ \in L^1(\M,\tau)$ is continuous in $L^1(\M,\tau)$ norm: We have $X_+ = \left( X + \abs{X} \right)/2$, and the absolute value is norm continuous as a map on $L^1(\M,\tau)$ because $L^1(\M,\tau) \cong \M_*$ and $\varphi \mapsto \abs{ \varphi }$ is norm continuous on $\M_*$ \cite{kosaki1984b}. As $t \mapsto (B - t A) \in L^1(\M,\tau)$ is clearly continuous in $L^1(\M,\tau)$ norm, and $f$ is continuous as well, the continuity claim follows. Analogous but simpler arguments show the continuity in the norm of $\M$.

Since $B$ and $A^{-1}$ are bounded ($A \geq a > 0$ with $a = \norm{ A^{-1} }^{-1}$), we have $B - t A \leq \norm{ B } - t a$, and therefore $( B - t A )_+ = 0$ for $t > \norm{ B } \norm{ A^{-1} }$. Thus $F^I_{A,B,f}$ is a compactly supported function continuous in both the $L^1(\M,\tau)$ norm and the norm of $\M$, and hence integrable in both norms.

For $F^J_{A,B,f}$, we first consider the function $C_{A,B}$, abbreviated to $C$ in this proof. Since $A \geq a > 0$, we have for  $s, s' \geq 0$
\begin{equation}
\norm{ (A+s)^{-\frac{1}{2}} - (A+s')^{-\frac{1}{2}} } \leq \sup_{\lambda \geq a} \abs{ (\lambda + s)^{-\frac{1}{2}} - (\lambda + s')^{-\frac{1}{2}} } \leq \frac{1}{2} \norm{ A^{-1} }^\frac{3}{2} \abs{ s-s' } \eqend{,}
\end{equation}
which shows that $\mathbb{R}_+ \ni s \mapsto (A+s)^{-\half} \in \M$ is norm continuous and bounded. By the continuity of left and right multiplication of $L^1(\M,\tau)$ by elements from $\M$ in $L^1(\M,\tau)$ norm, which follows from the noncommutative H{\"o}lder inequality, it follows that $C$ is continuous in $L^1(\M,\tau)$ norm and uniformly bounded by $\norm{ C(s) }_{L^1(\M,\tau)} \leq \norm{ B }_{L^1(\M,\tau)} (a+s)^{-1} \leq \norm{ B }_{L^1(\M,\tau)} \norm{ A^{-1} }$.

This bound holds in the same way for the norm in $\M$, which implies that $\sigma(C(s)) \subset \I_{A,B}$ for all $s \geq 0$. On the compact interval $\I_{A,B}$, using Weierstraß' Theorem we may approximate the continuous function $f$ uniformly by polynomials. For $f$ a polynomial, continuity of $s \mapsto f(C(s))$ is clear by the norm continuity of the algebraic operations, and a standard $\epsilon/3$-argument shows that $\mathbb{R}_+ \ni s \mapsto f(C(s)) \in \M$ is continuous in the norm of $\M$. We now view $F^J_{A,B,f}$ as the product of various norm continuous functions $\mathbb{R}_+ \to \M$ (the exterior factors with the roots, $f \circ C$, and one factor $C$) and the $L^1(\M,\tau)$-norm continuous function $C \colon \mathbb{R}_+ \to L^1(\M,\tau)$. This shows the continuity of $F^J_{A,B,f}$ in $L^1(\M,\tau)$ norm. Analogously, one checks the continuity in the norm of $\M$.

In the estimate
\begin{splitequation}
\norm{ F^J_{A,B,f}(s) }_{L^1(\M,\tau)} &\leq \norm{ \sqrt{ \frac{s}{A + s} } }^2 \norm{ C(s) } \norm{ C(s) }_{L^1(\M,\tau)} \norm{ f(C(s)) } \\
&\leq \frac{s}{a+s} \frac{\norm{ B }_{L^1(\M,\tau)} \norm{ B }}{(a+s)^2} \sup_{\lambda \in \I_{A,B}} \abs{ f(\lambda) } \eqend{,}
\end{splitequation}
following again from the noncommutative H{\"o}lder inequality, the supremum is finite by continuity of $f$, and the upper bound is integrable. Hence $F^J_{A,B,f}$ is integrable over $\mathbb{R}_+$ in $L^1(\M,\tau)$ norm, and by analogous arguments also in the norm of $\M$.

\ref{item:BochnerIntegrals} is a direct consequence of the arguments given in \ref{item:Integrands}.
\end{proof}

\subsection{Equality of the integrals}
\label{sec:operatorint_equality}

We now show that the integrals are actually equal. In this section, we make the following standing assumption:
\begin{assumption}
\label{ass:operatorint}
Let $(\M,\tau)$ be a semifinite von Neumann algebra, $A, B, C, D \in L^1(\M,\tau)$ self-adjoint with $A$, $B$ and $A^{-1}$ positive and bounded, $x \in \M$, and $f \colon \mathbb{R}_+ \to \mathbb{R}_+$ a continuous bounded positive function.
\end{assumption}
Moreover, we employ \lemmaref{lemma:resolvent_bound_imaginary} and \ref{lemma:resolvent_bound_complex} for bounding resolvents, without specifying their use explicitly each time.

\begin{lemma}
$I_f(A,B) = J_f(A,B)$ is equivalent to
\begin{equation}
\label{eq:IntegralEqualityWeak}
I^{\tau,x}_f(A,B) \defeq \int_0^\infty \tau\left( x^* F^I_{A,B,f}(t) \, x \right) \total t = \int_0^\infty \tau\left( x^* F^J_{A,B,f}(s) \, x \right) \total s \eqdef J^{\tau,x}_{f}(A,B)
\end{equation}
for all $x \in \M$.
\end{lemma}
\begin{proof}
Since the Bochner integrals exist in $L^1(\M,\tau)$, and $\tau$ is continuous in the $L^1(\M,\tau)$ norm, by Hille's theorem~\cite[Thm.~II.~2.6]{diesteluhl1977} we have $I^{\tau,x}_f(A,B) = \tau\left( x^* I_f(A,B) \, x \right)$, and analogously for the other integral. Thus \eqref{eq:IntegralEqualityWeak} implies $0 = \tau\left( y \Bigl[ I_f(A,B) - J_f(A,B) \Bigr] \right)$ for all $y \in \M_+$, hence by linearity for all $y \in \M$ and in particular for $y = \Bigl[ I_f(A,B) - J_f(A,B) \Bigr]^*$. Faithfulness of $\tau$ then implies $I_f(A,B) = J_f(A,B)$. The other implication is clear.
\end{proof}

The first step is to establish the equality \eqref{eq:IntegralEqualityWeak} for the constant function $f = 1$, which is \propref{prop:operatorint_ihab_jhab_taux_const}. This requires a number of preparatory steps: In \propref{prop:operatorint_gh_lipschitz}, we approximate a step function underlying the hockeystick divergence by smooth functions. In \lemmaref{lemma:operatorint_kab_analytic} we prove analyticity of various resolvent-type functions that appear in the integrals, see in particular the integral representation for $I_1^{\tau,x}(A,B)$ through these functions in \propref{prop:operatorint_ihab_resolvent_form}. In \propref{prop:operatorint_ihab_jhab_taux_const}, we then make use of analytic continuation to show the equality of the two integrals for constant $f = 1$, and this is generalized to general continuous $f$ in \propref{prop:operatorint_ihab_jhab_taux_f}.

\begin{proposition}
\label{prop:operatorint_gh_lipschitz}
For $\epsilon > 0$ let
\begin{equation}
\label{eq:operatorint_ihab_jhab_taux_gh_def}
g^{A,B,x}_s(t) \defeq \frac{2}{\pi} \tau\left( x^* \left[ \frac{(B - t A)^2}{s^2 + (B - t A)^2} \right] x \right) \eqend{,} \quad h^{A,B,x}_\epsilon(t) \defeq \int_0^\infty g^{A,B,x}_{s+\epsilon}(t) \total s \eqend{.}
\end{equation}
Then $g^{A,B,x}_s$ (for $s > 0$) and $h^{A,B,x}_\epsilon$ are locally Lipschitz continuous in $t$, and it holds that
\begin{equation}
I^{\tau,x}_1(A, B) = \int_0^\infty \tau\left( x^* (B - t A)_+ x \right) \total t = \frac{1}{4} \lim_{R \to \infty} \lim_{\epsilon \searrow 0} \int_{-R}^R \left( 1 - t \partial_t \right) h^{A,B,x}_\epsilon(t) \total t \eqend{.}
\end{equation}
\end{proposition}
\begin{proof}
The local Lipschitz continuity of $g^{A,B,x}_s$ follows from
\begin{splitequation}
&\pi \frac{g^{A,B,x}_s(t) - g^{A,B,x}_s(t')}{t-t'} \\
&= \tau\left( x^* \left[ \frac{\mathi s}{s - \mathi (B - t' A)} A \frac{1}{s - \mathi (B - t A)} - \frac{\mathi s}{s + \mathi (B - t A)} A \frac{1}{s + \mathi (B - t' A)} \right] x \right) \\
&= - \tau\biggl( x^* \biggl[ \frac{B - t A}{s - \mathi (B - t A)} A \frac{1}{s + \mathi (B - t' A)} + \frac{B - t A}{s + \mathi (B - t A)} A \frac{1}{s + \mathi (B - t' A)} \\
&\qquad\qquad+ \frac{s}{s - \mathi (B - t' A)} A \frac{1}{s + \mathi (B - t' A)} \left[ 2 B - (t+t') A \right] \frac{1}{s - \mathi (B - t A)} \\
&\qquad\qquad+ (t-t') \frac{s}{s - \mathi (B - t A)} A \frac{1}{s - \mathi (B - t' A)} A \frac{1}{s + \mathi (B - t' A)} \biggr] x \biggr)
\end{splitequation}
which is obtained using resolvent identities, and the bounds
\begin{equation}
\abs{ \tau\left( x^* f(s) A g(s) x \right) } \leq \norm{ x }^2 \norm{ f(s) } \norm{ g(s) } \tau(A) \eqend{,} \quad \norm{ \frac{1}{s \pm \mathi (B-t A)} } \leq \frac{1}{s} \eqend{,}
\end{equation}
which show that $\abs{ g^{A,B,x}_s(t) - g^{A,B,x}_s(t') } \leq \abs{t-t'} \, c \, s^{-2}$ for some $c = c(A,B,t,t',x)$. Since this bound (with $s \to s + \epsilon$) is integrable in $s$ for any $\epsilon > 0$, also $h^{A,B,x}_\epsilon$ is locally Lipschitz continuous. We now want to show that the limit
\begin{equation}
\lim_{\epsilon \searrow 0} h^{A,B,x}_\epsilon(t) = \tau\left( x^* \abs{B - t A} x \right)
\end{equation}
is obtained monotonically. For this, we use \lemmaref{lemma:operatorint_tau_function} with $C = \abs{B - t A}$ and Tonelli's theorem to obtain
\begin{splitequation}
h^{A,B,x}_\epsilon(t) &= \frac{2}{\pi} \int_0^\infty \int_0^\infty \frac{\lambda}{( s + \epsilon )^2 + \lambda^2} \total \mu^x_{\abs{B-tA}}(\lambda) \total s \\
&= \frac{2}{\pi} \int_0^\infty \int_0^\infty \frac{\lambda}{( s + \epsilon )^2 + \lambda^2} \total s \total \mu^x_{\abs{B-tA}}(\lambda) = \frac{2}{\pi} \int_0^\infty \arctan\left( \frac{\lambda}{\epsilon} \right) \total \mu^x_{\abs{B-tA}}(\lambda) \eqend{.}
\end{splitequation}
Since $\mu^x_{\abs{B-tA}}$ is a finite measure and $\arctan\left( \frac{\lambda}{\epsilon} \right)$ converges monotonically to $\frac{\pi}{2}$ for $\lambda \geq 0$, the monotone convergence theorem shows that
\begin{equation}
\lim_{\epsilon \searrow 0} \int_0^\infty \arctan\left( \frac{\lambda}{\epsilon} \right) \total \mu^x_{\abs{B-tA}}(\lambda) = \int_0^\infty \frac{\pi}{2} \total \mu^x_{\abs{B-tA}}(\lambda) = \frac{\pi}{2} \tau\left( x^* \abs{B - t A} x \right) \eqend{.}
\end{equation}
Using again the monotone convergence theorem, it thus follows that
\begin{splitequation}
&\lim_{\epsilon \searrow 0} \int_{-R}^R h^{A,B,x}_\epsilon(t) \total t = \int_{-R}^R \tau\left( x^* \abs{B - t A} x \right) \total t \\
&\quad= 2 \int_0^R \tau\left( x^* (B - t A)_+ x \right) \total t - \int_0^R \tau\left( x^* (B - t A) x \right) \total t \\
&\qquad+ 2 \int_{-R}^0 \tau\left( x^* (B - t A)_- x \right) \total t + \int_{-R}^0 \tau\left( x^* (B - t A) x \right) \total t \\
&\quad= 2 \int_0^R \tau\left( x^* (B - t A)_+ x \right) \total t + R^2 \, \tau(x^* A x) \eqend{,}
\end{splitequation}
where we used that $\abs{C} = 2 C_+ - C = 2 C_- + C$ and that $(B - t A)_- = 0$ for $t \leq 0$. Because $h^{A,B,x}_\epsilon(t)$ is locally Lipschitz continuous, by Rademacher's theorem it is almost everywhere differentiable and its derivative satisfies the fundamental theorem of calculus. Partial integration then shows that
\begin{equation}
\int_{-R}^R \left( 1 - t \partial_t \right) h^{A,B,x}_\epsilon(t) \total t = 2 \int_{-R}^R h^{A,B,x}_\epsilon(t) \total t - R h^{A,B,x}_\epsilon(R) - R h^{A,B,x}_\epsilon(-R) \eqend{,}
\end{equation}
such that
\begin{splitequation}
&\frac{1}{4} \lim_{R \to \infty} \lim_{\epsilon \searrow 0} \int_{-R}^R \left( 1 - t \partial_t \right) h^{A,B,x}_\epsilon(t) \total t \\
&\quad= \lim_{R \to \infty} \left[ \frac{1}{2} \lim_{\epsilon \searrow 0} \int_{-R}^R h^{A,B,x}_\epsilon(t) \total t - \frac{R}{4} \tau\left( x^* \abs{B - R A} x \right) - \frac{R}{4} \tau\left( x^* \abs{B + R A} x \right) \right] \\
&\quad= \lim_{R \to \infty} \Biggl[ \int_0^R \tau\left( x^* (B - t A)_+ x \right) \total t \\
&\hspace{6em}+ \frac{R^2}{2} \, \tau(x^* A x) - \frac{R}{4} \tau\left( x^* \abs{B - R A} x \right) - \frac{R}{4} \tau\left( x^* \abs{B + R A} x \right) \Biggr] \eqend{.}
\end{splitequation}
Since $(B - t A)_+ = 0$ for $t > \norm{ B } \norm{ A^{-1} }$, the limit $R \to \infty$ is trivial for the integral, and we conclude by noting that for $R > \norm{ B } \norm{ A^{-1} }$ we have
\begin{equation}
\tau\left( x^* \abs{B - R A} x \right) + \tau\left( x^* \abs{B + R A} x \right) = \tau\left( x^* (R A - B) x \right) + \tau\left( x^* (B + R A) x \right) = 2 R \, \tau( x^* A x ) \eqend{.}
\end{equation}
\end{proof}

\begin{lemma}
\label{lemma:operatorint_kab_analytic}
Let
\begin{equation}
k_{0,\pm}^{A,B,C}(s,t) \defeq \frac{1}{s \pm \mathi (B - t A)}
\end{equation}
and recursively for $n \in \mathbb{N}$
\begin{equation}
k_{n,\pm}^{A,B,C}(s,t) \defeq k_{n-1,\pm}^{A,B,C}(s,t) C \frac{1}{s \pm \mathi (B - t A)} \eqend{.}
\end{equation}
For any $s, \delta > 0$, the function
\begin{equation}
z \mapsto \tau\left( x^* k_{2,\pm}^{A,\delta B,B}(s,z) \, x \right) \eqend{,} \quad z \in \mathbb{R} \pm \mathi \mathbb{R}_+
\end{equation}
is analytic in the given half-plane. For any $s, \delta > 0$ and $\phi \in \left[ - \frac{\pi}{2}, \frac{\pi}{2} \right]$, the function
\begin{equation}
z \mapsto \tau\left( x^* k_{2,\pm}^{A,\delta B,B}\left( z, \mathi \mathe^{\mathi \phi} \right) \, x \right) \eqend{,} \quad \Re z \geq \epsilon
\end{equation}
is analytic in the shifted right half-plane for any $\epsilon > 0$. For any $s \geq 0$, the function
\begin{equation}
z \mapsto \tau\left( x^* k_{2,\pm}^{A,0,B}(z,1) \, x \right) \eqend{,} \quad z \in \mathbb{R} \mp \mathi \mathbb{R}_+
\end{equation}
is analytic in the given half-plane.
\end{lemma}
\begin{proof}
Using resolvent identities, we can compute formally the derivatives of the given functions with respect to $\Re z$ and $\Im z$, which are of the form $\tau( C'(\Re z) )$ or $\tau( D'(\Im z) )$. Clearly, the Cauchy--Riemann equations are fulfilled and the functions are analytic whenever the derivatives can be justified. Similar straightforward but lengthy computations to the ones used to prove local Lipschitz continuity in \propref{prop:operatorint_gh_lipschitz} show that $C'(\Re z)$ and $D'(\Im z)$ are norm continuous in $\Re z$ and $\Im z$, respectively, with the norm uniformly bounded as long as $z$ lies in the stated half-planes. The main point here is the invertibility of the operators $s \pm \mathi \delta B \mp \mathi z A$, $z \pm \mathi \delta B \pm \mathe^{\mathi \phi} A$, and $z \mp \mathi A$ respectively, such that the norms of their resolvents are bounded. Therefore, \lemmaref{lemma:operatorint_tau_differentiable} is applicable, and the derivatives are indeed given by the expressions obtained by the formal computation.
\end{proof}

\begin{proposition}
\label{prop:operatorint_ihab_resolvent_form}
It holds that
\begin{equation}
I^{\tau,x}_1(A, B) = \frac{1}{2 \pi} \lim_{R \to \infty} \lim_{\epsilon \searrow 0} \int_\epsilon^\infty s \int_{-R}^R \tau\left[ x^* \left( k_{2,+}^{A,B,B}(s,t) + k_{2,-}^{A,B,B}(s,t) \right) x \right] \total t \total s \eqend{.}
\end{equation}
with the functions $k_{n,\pm}^{A,B,C}(s,t)$ defined in \lemmaref{lemma:operatorint_kab_analytic}.
\end{proposition}
\begin{proof}
By \propref{prop:operatorint_gh_lipschitz}, for any $s > 0$ the function $g^{A,B,x}_s(t)$ defined in~\eqref{eq:operatorint_ihab_jhab_taux_gh_def} is locally Lipschitz continuous in $t$, and by Rademacher's theorem almost everywhere differentiable. For its derivative, using resolvent identities we compute formally
\begin{splitequation}
\left( 1 - t \partial_t \right) g^{A,B,x}_s(t) &= \frac{1}{\pi} \tau\biggl( x^* \biggl[ \mathi s k_{1,+}^{A,B,B}(s,t) - \mathi s k_{1,-}^{A,B,B}(s,t) \\
&\hspace{4em}+ 2 \frac{(B - t A)^2 \left[ (B - t A)^2 - s^2 \right]}{\left[ s^2 + (B - t A)^2 \right]^2} \biggr] x \biggr) \eqend{.}
\end{splitequation}
This is of the form $\tau\left( C'(t) \right)$ with $C'(t) \in L^1(\M,\tau)$, and similar straightforward but lengthy computations to the ones used to prove local Lipschitz continuity in \propref{prop:operatorint_gh_lipschitz} show that $C'(t)$ is continuous in $t$ in $L^1(\M,\tau)$ norm with the norm uniformly bounded for $t \in [-R,R]$. Therefore, \lemmaref{lemma:operatorint_tau_differentiable} is applicable and the derivative of $g^{A,B,x}_s(t)$ is equal to the stated expression. More straightforward estimates also give
\begin{equation}
\abs{ \left( 1 - t \partial_t \right) g^{A,B,x}_{s + \epsilon}(t) } \leq \frac{4}{\pi} \norm{ x }^2 \norm{ B - t A } \Bigl[ \tau\left( \abs{ B - t A } \right) + \tau(B) \Bigr] \, (s + \epsilon)^{-2} \eqend{,}
\end{equation}
such that both the dominated convergence theorem and Fubini's theorem are applicable and yield
\begin{equation}
\int_{-R}^R \left( 1 - t \partial_t \right) h^{A,B,x}_\epsilon(t) \total t = \int_\epsilon^\infty \int_{-R}^R \left( 1 - t \partial_t \right) g^{A,B,x}_s(t) \total t \total s \eqend{.}
\end{equation}
for the function $h^{A,B,x}_\epsilon(t)$ defined in~\eqref{eq:operatorint_ihab_jhab_taux_gh_def}. We repeat the procedure, and with the same arguments as before and a lengthy computation obtain
\begin{equation}
t \partial_t \left( 1 - t \partial_t \right) g^{A,B,x}_s(t) = - \frac{2}{\pi} s \, \tau\left( x^* \left[ k_{2,+}^{A,B,B}(s,t) + k_{2,-}^{A,B,B}(s,t) \right] x \right) - \frac{2 \mathi}{\pi} \partial_s \tau\left( x^* \tilde{k}^{A,B}(s,t) \, x \right)
\end{equation}
with
\begin{equation}
\tilde{k}^{A,B}(s,t) \defeq s^2 k_{1,+}^{A,B,B}(s,t) - s^2 k_{1,-}^{A,B,B}(s,t) + \frac{2 \mathi (B - t A)^2 s^3}{\left[ (B - t A)^2 + s^2 \right]^2} \eqend{.}
\end{equation}
Integration by parts then yields
\begin{splitequation}
\int_{-R}^R \left( 1 - t \partial_t \right) g^{A,B,x}_s(t) \total t &= \int_{-R}^R \partial_t \left[ t \left( 1 - t \partial_t \right) g^{A,B,x}_s(t) \right] \total t - \int_{-R}^R t \partial_t \left( 1 - t \partial_t \right) g^{A,B,x}_s(t) \total t \\
&= \left[ t \left( 1 - t \partial_t \right) g^{A,B,x}_s(t) \right]_{t = R} - \left[ t \left( 1 - t \partial_t \right) g^{A,B,x}_s(t) \right]_{t = - R} \\
&\quad- \int_{-R}^R t \partial_t \left( 1 - t \partial_t \right) g^{A,B,x}_s(t) \total t \eqend{,}
\end{splitequation}
and using \propref{prop:operatorint_gh_lipschitz} and combining all of the above we obtain
\begin{splitequation}
&I^{\tau,x}_1(A, B) = \frac{1}{4} \lim_{R \to \infty} \lim_{\epsilon \searrow 0} \int_{-R}^R \left( 1 - t \partial_t \right) h^{A,B,x}_\epsilon(t) \total t \\
&\quad= \lim_{R \to \infty} \lim_{\epsilon \searrow 0} \int_\epsilon^\infty \Biggl[ \frac{1}{4} \left[ t \left( 1 - t \partial_t \right) g^{A,B,x}_s(t) \right]_{t = R} - \frac{1}{4} \left[ t \left( 1 - t \partial_t \right) g^{A,B,x}_s(t) \right]_{t = - R} \\
&\qquad+ \frac{\mathi}{2 \pi} \int_{-R}^R \partial_s \tau\left( x^* \tilde{k}^{A,B}(s,t) \, x \right) \total t + \frac{s}{2 \pi} \int_{-R}^R \tau\left( x^* \left[ k_{2,+}^{A,B,B}(s,t) + k_{2,-}^{A,B,B}(s,t) \right] x \right) \total t \Biggr] \total s \eqend{.}
\end{splitequation}
We thus obtain the conclusion if we show that all except the last term vanish in the limits $R \to \infty$ and $\epsilon \searrow 0$. We start with the middle term involving $\tilde{k}^{A,B}$. Another straightforward computation shows that
\begin{splitequation}
\abs{ \partial_s \tau\left( x^* \tilde{k}^{A,B}(s,t) \, x \right) } &\leq 2 \norm{ x }^2 \norm{ B - t A } \Bigl( 10 \tau(B) + 7 \tau( \abs{ B - t A } ) \Bigr) s^{-2} \eqend{,} \\
\abs{ \tau\left( x^* \tilde{k}^{A,B}(s,t) \, x \right) } &\leq 2 \norm{ x }^2 \Bigl( 2 \tau(B) + \tau( \abs{ B - t A } ) \Bigr) \min\left( 1, \norm{ B - t A } s^{-1} \right) \eqend{,}
\end{splitequation}
and using Fubini's theorem we can interchange the $s$ and $t$ integrations to obtain
\begin{equation}
\int_\epsilon^\infty \int_{-R}^R \partial_s \tau\left( x^* \tilde{k}^{A,B}(s,t) \, x \right) \total t \total s = - \int_{-R}^R \tau\left( x^* \tilde{k}^{A,B}(\epsilon,t) \, x \right) \total t \eqend{.}
\end{equation}
Since the integrand is bounded uniformly in $\epsilon$, the dominated convergence theorem shows that we can interchange the limit $\epsilon \searrow 0$ with the integration in $t$, and the integrand converges pointwise to $0$. For the first term, the same reasoning as before yields
\begin{splitequation}
\left( 1 - t \partial_t \right) g^{A,B,x}_s(t) &= \frac{\mathi s}{\pi} \tau\left( x^* \left[ k_{1,+}^{A,B,B}(s,t) - k_{1,-}^{A,B,B}(s,t) \right] x \right) \\
&\quad+ \frac{1}{\pi} \partial_s \tau\left( x^* \left[ \frac{(B - t A)^2}{s + \mathi (B - t A)} + \frac{(B - t A)^2}{s - \mathi (B - t A)} \right] x \right)
\end{splitequation}
and the bound $\abs{ \tau\left( x^* \left[ \frac{(B - t A)^2}{s \pm \mathi (B - t A)} \right] x \right) } \leq \norm{ x }^2 \norm{ B - t A } \tau\left( \abs{B - t A} \right) \, s^{-1}$, such that
\begin{splitequation}
&\int_\epsilon^\infty \biggl[ \left[ t \left( 1 - t \partial_t \right) g^{A,B,x}_s(t) \right]_{t = R} - \left[ t \left( 1 - t \partial_t \right) g^{A,B,x}_s(t) \right]_{t = - R} \biggr] \total s \\
&\quad= \frac{\mathi R}{\pi} \int_\epsilon^\infty s \, \tau\left( x^* \left[ k_{1,+}^{A,B,B}(s,R) - k_{1,-}^{A,B,B}(s,R) + k_{1,+}^{A,B,B}(s,-R) - k_{1,-}^{A,B,B}(s,-R) \right] x \right) \total s \\
&\qquad- \frac{R}{\pi} \tau\left( x^* \left[ \frac{(B - R A)^2}{\epsilon + \mathi (B - R A)} + \frac{(B - R A)^2}{\epsilon - \mathi (B - R A)} \right] x \right) \\
&\qquad- \frac{R}{\pi} \tau\left( x^* \left[ \frac{(B + R A)^2}{\epsilon + \mathi (B + R A)} + \frac{(B + R A)^2}{\epsilon - \mathi (B + R A)} \right] x \right) \eqend{.}
\end{splitequation}
The bounds $\abs{ \tau\left( x^* k_{1,+}^{A,B,B}(s,\pm R) \, x \right) }, \abs{ \tau\left( x^* k_{1,-}^{A,B,B}(s,\pm R) \, x \right) } \leq \norm{ x }^2 \, \tau(B) \norm{ \frac{1}{B \mp R A} }^2$ show that for $R > \norm{ B } \norm{ A^{-1} }$ the integrand is bounded for all $s \in [0,1]$, such that the limit $\epsilon \searrow 0$ is finite. Rescaling then $s \to R s$, setting $\delta = R^{-1}$ and using resolvent identities, we obtain
\begin{splitequation}
&\frac{\mathi R}{\pi} \int_0^\infty s \, \tau\left( x^* \left[ k_{1,+}^{A,B,B}(s,R) - k_{1,-}^{A,B,B}(s,R) + k_{1,+}^{A,B,B}(s,-R) - k_{1,-}^{A,B,B}(s,-R) \right] x \right) \total s \\
&\quad= \frac{2}{\pi} \int_0^\infty s \, \tau\left( x^* \left[ \kappa^{A,B}_-(s,-\delta) + \kappa^{A,B}_-(s,\delta) + \kappa^{A,B}_+(s,-\delta) + \kappa^{A,B}_+(s,\delta) \right] x \right) \total s
\end{splitequation}
with
\begin{equation}
\kappa^{A,B}_\pm(s,\delta) \defeq \frac{1}{s \pm \mathi (A + \delta B)} B \frac{1}{s \pm \mathi (A - \delta B)} B \frac{1}{s \pm \mathi (A + \delta B)} \eqend{.}
\end{equation}
The bound
\begin{equation}
s \, \abs{ \tau\left( x^* \kappa_\pm^{A,B}(s,\delta) \, x \right) } \leq \norm{ x }^2 \norm{ B } \tau(B) \min\left( s^{-2}, \norm{ \frac{1}{A + \delta B} }^2 \right)
\end{equation}
shows that for sufficiently small $\delta$ (say $\delta < ( 2 \norm{ A^{-1} } \norm{B} )^{-1}$) the integrand is bounded in absolute value by an integrable function of $s$ uniformly in $\delta$. Using the dominated convergence theorem, we can exchange the limit $\delta \to 0^+$ (corresponding to $R \to \infty$) with the integration and obtain
\begin{splitequation}
&\lim_{R \to \infty} \frac{\mathi R}{\pi} \int_0^\infty s \, \tau\left( x^* \left[ k_{1,+}^{A,B,B}(s,R) - k_{1,-}^{A,B,B}(s,R) + k_{1,+}^{A,B,B}(s,-R) - k_{1,-}^{A,B,B}(s,-R) \right] x \right) \total s \\
&\quad= \frac{4}{\pi} \int_0^\infty s \, \tau\left( x^* \left[ \frac{1}{s + \mathi A} B \frac{1}{s + \mathi A} B \frac{1}{s + \mathi A} + \frac{1}{s - \mathi A} B \frac{1}{s - \mathi A} B \frac{1}{s - \mathi A} \right] x \right) \total s \\
&\quad= \frac{4}{\pi} \int_{-\infty}^\infty s \, \tau\left( x^* \frac{1}{s + \mathi A} B \frac{1}{s + \mathi A} B \frac{1}{s + \mathi A} \, x \right) \total s = \frac{4}{\pi} \int_{-\infty}^\infty s \, \tau\left( x^* k_{2,-}^{A,0,B}(s,1) \, x \right) \total s \eqend{.}
\end{splitequation}
By \lemmaref{lemma:operatorint_kab_analytic}, the function $z \mapsto z \, \tau\left( x^* k_{2,-}^{A,0,B}(z,1) \, x \right)$ is analytic for $\Im z \geq 0$, and by Cauchy's theorem, its integral over the rectangle contour with endpoints $\pm R$ and $\pm R + \mathi M$ vanishes. For $u \geq 0$, the bound
\begin{equation}
\abs{ (s + \mathi u) \, \tau\left( x^* k_{2,-}^{A,0,B}(s + \mathi u,1) \, x \right) } \leq \norm{ x }^2 \norm{ B } \tau(B) \left( s^2 + u^2 + \norm{ A^{-1} }^{-2} \right)^{-1}
\end{equation}
shows that the integrals over the straight lines from $\pm R$ to $\pm R + \mathi M$ vanish in the limit $R \to \infty$, such that we obtain
\begin{equation}
\int_{-\infty}^\infty s \, \tau\left( x^* k_{2,-}^{A,0,B}(s,1) \, x \right) \total s = \int_{-\infty}^\infty (s + \mathi M) \, \tau\left( x^* k_{2,-}^{A,0,B}(s + \mathi M,1) \, x \right) \total s
\end{equation}
for any $M \geq 0$. The same bound shows that the integrand is bounded in absolute value by an integrable function of $s$ uniformly in $M$ and vanishes pointwise as $M \to \infty$. Using the dominated convergence theorem we can take the limit $M \to \infty$ inside the integral, such that this term vanishes as well.
\end{proof}

\begin{proposition}
\label{prop:operatorint_ihab_jhab_taux_const}
For any $x \in \M$, the integrals $I^{\tau,x}_1(A, B)$ and $J^{\tau,x}_1(A, B)$~\eqref{eq:IntegralEqualityWeak} are equal.
\end{proposition}
\begin{proof}
By \propref{prop:operatorint_ihab_resolvent_form}, it holds that
\begin{equation}
I^{\tau,x}_1(A, B) = \frac{1}{2 \pi} \lim_{R \to \infty} \lim_{\epsilon \searrow 0} \int_\epsilon^\infty s \int_{-R}^R \tau\left[ x^* \left( k_{2,+}^{A,B,B}(s,t) + k_{2,-}^{A,B,B}(s,t) \right) x \right] \total t \total s \eqend{,}
\end{equation}
and rescaling $\epsilon \to R \epsilon$, $t \to R t$, $s \to R s$ and setting $\delta = R^{-1}$, this can be written as
\begin{equation}
I^{\tau,x}_1(A, B) = \frac{1}{2 \pi} \lim_{\delta \to 0^+} \lim_{\epsilon \searrow 0} \int_\epsilon^\infty s \int_{-1}^1 \tau\left[ x^* \left( k_{2,+}^{A,\delta B,B}(s, t) + k_{2,-}^{A,\delta B,B}(s, t) \right) x \right] \total t \total s \eqend{.}
\end{equation}
Consider now the closed contour consisting of the interval $[-1,1]$ and the half-circle of radius 1 in the upper half-plane. By \lemmaref{lemma:operatorint_kab_analytic}, the function $z \mapsto \tau\left( x^* k_{2,+}^{A,\delta B,B}(s,z) \, x \right)$ is analytic in the upper half-plane, such that by Cauchy's theorem its integral over the contour vanishes and parametrizing the half-circle by $z = \mathi \mathe^{\mathi \phi}$ with $\phi \in \left[ - \frac{\pi}{2}, \frac{\pi}{2} \right]$ we obtain
\begin{equation}
\int_{-1}^1 \tau\left( x^* k_{2,+}^{A,\delta B,B}(s,t) \, x \right) \total t = \int_{- \frac{\pi}{2}}^\frac{\pi}{2} \tau\left( x^* k_{2,+}^{A,\delta B,B}\left( s, \mathi \mathe^{\mathi \phi} \right) \, x \right) \mathe^{\mathi \phi} \total \phi \eqend{.}
\end{equation}
The bound
\begin{equation}
\abs{ \tau\left( x^* k_{2,+}^{A,\delta B,B}\left( s, \mathi \mathe^{\mathi \phi} \right) \, x \right) } \leq \norm{ x }^2 \norm{ B } \tau(B) \, s^{-3}
\end{equation}
shows that the integrand is bounded in absolute value by an integrable function of $s$, and using Fubini's theorem we can interchange the integrations to obtain
\begin{equation}
\int_\epsilon^\infty s \int_{-1}^1 \tau\left( x^* k_{2,+}^{A,\delta B,B}(s, t) \, x \right) \total t \total s = \int_{- \frac{\pi}{2}}^\frac{\pi}{2} \int_\epsilon^\infty s \, \tau\left( x^* k_{2,+}^{A,\delta B,B}\left( s, \mathi \mathe^{\mathi \phi} \right) \, x \right) \total s \, \mathe^{\mathi \phi} \total \phi \eqend{.}
\end{equation}
By \lemmaref{lemma:operatorint_kab_analytic}, the function $z \mapsto z \, \tau\left( x^* k_{2,+}^{A,\delta B,B}\left( z, \mathi \mathe^{\mathi \phi} \right) \, x \right)$ is analytic for $\Re z \geq \epsilon$, and by Cauchy's theorem its integral over the closed contour consisting of the interval $[ \epsilon, R ]$, the straight line from $\mathe^{\mathi \phi} \epsilon$ to $\mathe^{\mathi \phi} R$ and the two arcs connecting $\epsilon$ and $\mathe^{\mathi \phi} \epsilon$ as well as $R$ and $\mathe^{\mathi \phi} R$ vanishes for any $R > \epsilon$. Parametrizing the large arc by $z = \mathe^{\mathi \alpha \phi} R$ with $\alpha \in [0,1]$, the integral over it can be bounded by
\begin{splitequation}
&\abs{ \int_0^1 R^2 \mathe^{2 \mathi \alpha \phi} \, \tau\left( x^* k_{2,+}^{A,\delta B,B}\left( R \mathe^{\mathi \alpha \phi}, \mathi \mathe^{\mathi \phi} \right) \, x \right) \phi \total \alpha } \\
&\quad\leq R^2 \abs{\phi} \norm{ x }^2 \norm{ B } \tau(B) \int_0^1 \norm{ \frac{1}{R \mathe^{\mathi \alpha \phi} + \mathi \delta B + \mathe^{\mathi \phi} A} }^3 \total \alpha \\
&\quad= \frac{\abs{\phi} \norm{ x }^2 \norm{ B } \tau(B)}{R} \int_0^1 \norm{ \frac{1}{\1 + \mathi \delta R^{-1} B \, \mathe^{- \mathi \alpha \phi} + R^{-1} A \, \mathe^{\mathi (1-\alpha) \phi}} }^3 \total \alpha \eqend{,}
\end{splitequation}
and since $A$ and $B$ are bounded, the norm is finite for sufficiently large $R$ and the integral vanishes in the limit $R \to \infty$. For the small arc, we obtain analogously
\begin{splitequation}
&\abs{ \int_0^1 \epsilon^2 \mathe^{2 \mathi \alpha \phi} \, \tau\left( x^* k_{2,+}^{A,\delta B,B}\left( \epsilon \, \mathe^{\mathi \alpha \phi}, \mathi \mathe^{\mathi \phi} \right) \, x \right) \phi \total \alpha } \\
&\quad\leq \epsilon \abs{\phi} \norm{ x }^2 \norm{ B } \tau(B) \int_0^1 \norm{ \frac{1}{\epsilon \, \mathe^{- \mathi (1-\alpha) \phi} + \mathi \delta B \, \mathe^{- \mathi \phi} + A} }^3 \total \alpha \eqend{,}
\end{splitequation}
and for sufficiently small $\epsilon$ and $\delta$, the norm is again finite, uniformly in $\phi$ and $\alpha$. Since the integration in $\phi$ is over the finite interval $\left[ - \frac{\pi}{2}, \frac{\pi}{2} \right]$, for this term we may exchange the limit $\epsilon \searrow 0$ with the integration, and it follows that
\begin{splitequation}
&\lim_{\epsilon \searrow 0} \int_\epsilon^\infty s \int_{-1}^1 \tau\left( x^* k_{2,+}^{A,\delta B,B}(s, t) \, x \right) \total t \total s \\
&\quad= \lim_{\epsilon \searrow 0} \int_{- \frac{\pi}{2}}^\frac{\pi}{2} \int_\epsilon^\infty s \, \tau\left( x^* k_{2,+}^{A,\delta B,B}\left( \mathe^{\mathi \phi} s, \mathi \mathe^{\mathi \phi} \right) \, x \right) \total s \, \mathe^{3 \mathi \phi} \total \phi \eqend{.}
\end{splitequation}
Using the bound
\begin{splitequation}
&\abs{ \mathe^{3 \mathi \phi} \tau\left( x^* k_{2,+}^{A,\delta B,B}\left( \mathe^{\mathi \phi} s, \mathi \mathe^{\mathi \phi} \right) \, x \right) } \\
&\quad= \abs{ \tau\left( x^* \left[ \frac{1}{s + \mathi \delta B \, \mathe^{- \mathi \phi} + A} B \frac{1}{s + \mathi \delta B \, \mathe^{- \mathi \phi} + A} B \frac{1}{s + \mathi \delta B \, \mathe^{- \mathi \phi} + A} \right] x \right) } \\
&\quad\leq \norm{ x }^2 \norm{ B } \tau(B) \norm{ \frac{1}{s + \mathi \delta B \, \mathe^{- \mathi \phi} + A} }^3 \leq \norm{ x }^2 \norm{ B } \tau(B) \begin{cases} s^{-3} & s > 1 \eqend{,} \\ \norm{ \frac{1}{A - \delta B} }^3 & s \in [0,1] \eqend{,} \end{cases}
\end{splitequation}
because $A^{-1}$ is bounded, we see that for sufficiently small $\delta$ (say $\delta < ( 2 \norm{ A^{-1} } \norm{B} )^{-1}$) the integrand is bounded in absolute value by an integrable function of $s$ uniformly in $\phi$, $\epsilon$ and $\delta$. We may therefore interchange the limits $\delta, \epsilon \searrow 0$ with the integration and obtain
\begin{splitequation}
&\lim_{\delta \to 0^+} \lim_{\epsilon \searrow 0} \int_\epsilon^\infty s \int_{-1}^1 \tau\left( x^* k_{2,+}^{A,\delta B,B}(s, t) \, x \right) \total t \total s \\
&\quad= \int_{- \frac{\pi}{2}}^\frac{\pi}{2} \int_0^\infty s \, \tau\left( x^* \left[ \frac{1}{s + A} B \frac{1}{s + A} B \frac{1}{s + A} \right] x \right) \total s \total \phi \eqend{.}
\end{splitequation}
The same arguments show that also
\begin{splitequation}
&\lim_{\delta \to 0^+} \lim_{\epsilon \searrow 0} \int_\epsilon^\infty s \int_{-1}^1 \tau\left( x^* k_{2,-}^{A,\delta B,B}(s, t) \, x \right) \total t \total s \\
&\quad= \pi \int_0^\infty s \, \tau\left( x^* \left[ \frac{1}{s + A} B \frac{1}{s + A} B \frac{1}{s + A} \right] x \right) \total s \eqend{,}
\end{splitequation}
and combining both we obtain the conclusion.
\end{proof}

We now generalize this result to arbitrary continuous functions $f$.
\begin{proposition}
\label{prop:operatorint_ihab_jhab_taux_f}
For any continuous $f \colon \mathbb{R}_+ \to \mathbb{R}$, the integrals $I^{\tau,x}_f(A, B)$ and $J^{\tau,x}_f(A, B)$ \eqref{eq:IntegralEqualityWeak} are equal.
\end{proposition}
\begin{proof}
In view of \lemmaref[item:BochnerIntegrals]{lemma:integrability} and the Weierstraß approximation theorem, it is sufficient to consider $f$ to be a monomial. We already know from \propref{prop:operatorint_ihab_jhab_taux_const} that the statement is true for $f = 1$, i.e., we have
\begin{splitequation}
\int_0^\infty \tau\left( x^* ( B - t A )_+ x \right) \total t &= \int_0^\infty \tau\left( x^* \sqrt{ \frac{s}{A + s} } \, \left[ C_{A,B}(s) \right]^2 \sqrt{ \frac{s}{A + s} } x \right) \total s \\
&= \int_0^\infty s\, \tau\left( x^* \, \frac{1}{A + s} B \frac{1}{A + s} B \frac{1}{A + s} \, x \right) \total s \eqend{.}
\end{splitequation}
Taking $A + r B$ with $r > 0$ instead of $A$, we obtain
\begin{splitequation}
\int_0^\infty \tau\left( x^* ( (1 - r t) B - t A )_+ x \right) \total t &= \int_0^\frac{1}{r} \tau\left( x^* ( (1 - r t) B - t A )_+ x \right) \total t \\
&= \int_0^\frac{1}{r} (1 - r t) \tau\left( x^* \left( B - \frac{t}{1-r t} A \right)_+ x \right) \total t \\
&= \int_0^\infty (1 + r u)^{-3} \, \tau\left( x^* \left( B - u A \right)_+ x \right) \total u \eqend{,}
\end{splitequation}
where the first equality follows because for $r t > 1$ we have $( (1 - r t) B - t A )_+ = 0$, and the third by the change of variables $t = u/(1 + r u)$. Therefore, we obtain
\begin{equation}
\int_0^\infty \frac{1}{(1 + r u)^3} \, \tau\left( x^* \left( B - u A \right)_+ x \right) \total u = \int_0^\infty s\, \tau\left( x^* g_{A,B}(r,s) \, x \right) \total s
\end{equation}
with $g_{A,B}(r,s) \defeq (A + r B + s)^{-1} B (A + r B + s)^{-1} B (A + r B + s)^{-1}$ for all $r > 0$. The same estimates as in the proof of \lemmaref{lemma:integrability} show that both integrals are absolutely convergent, and using the dominated convergence theorem we can interchange derivatives with respect to $r$ with the integration. Furthermore, these estimates show that $g_{A,B}(r,s)$ is norm continuous in $r$ with respect to the $L^1(\M,\tau)$ norm, and differentiable with
\begin{splitequation}
&\partial_r g_{A,B}(r,s) = \lim_{\epsilon \to 0} \frac{g_{A,B}(r + \epsilon,s) - g_{A,B}(r,s)}{\epsilon} \\
&\qquad= \lim_{\epsilon \to 0} \frac{1}{\epsilon} \biggl[ \left( \frac{1}{A + r B + s + \epsilon B} - \frac{1}{A + r B + s} \right) B \frac{1}{A + r B + s + \epsilon B} B \frac{1}{A + r B + s + \epsilon B} \\
&\hspace{6em}+ \frac{1}{A + r B + s} B \left( \frac{1}{A + r B + s + \epsilon B} - \frac{1}{A + r B + s} \right) B \frac{1}{A + r B + s + \epsilon B} \\
&\hspace{6em}+ \frac{1}{A + r B + s} B \frac{1}{A + r B + s} B \left( \frac{1}{A + r B + s + \epsilon B} - \frac{1}{A + r B + s} \right) \biggr] \\
&\qquad= - 3 \frac{1}{A + r B + s} B \frac{1}{A + r B + s} B \frac{1}{A + r B + s} B \frac{1}{A + r B + s} \eqend{,}
\end{splitequation}
where we used the resolvent identity
\begin{equation}
\frac{1}{A + r B + s + \epsilon B} - \frac{1}{A + r B + s} = - \epsilon \frac{1}{A + r B + s + \epsilon B} B \frac{1}{A + r B + s} \eqend{.}
\end{equation}
The same estimates show that also $\partial_r g_{A,B}(r,s)$ is norm continuous in $r$ with respect to the $L^1(\M,\tau)$ norm, and we can apply \lemmaref{lemma:operatorint_tau_differentiable} to obtain
\begin{splitequation}
&- 3 \int_0^\infty \frac{u}{(1 + r u)^4} \, \tau\left( x^* \left( B - u A \right)_+ x \right) \total u = \partial_r \int_0^\infty \frac{1}{(1 + r u)^3} \, \tau\left( x^* \left( B - u A \right)_+ x \right) \total u \\
&\quad= \partial_r \int_0^\infty s\, \tau\left( x^* \, g_{A,B}(r,s) \, x \right) \total s = \int_0^\infty s\, \partial_r \tau\left( x^* \, g_{A,B}(r,s) \, x \right) \total s \\
&\quad= - 3 \int_0^\infty s\, \tau\left( x^* \, \frac{1}{A + r B + s} B \frac{1}{A + r B + s} B \frac{1}{A + r B + s} B \frac{1}{A + r B + s} \, x \right) \total s \eqend{.}
\end{splitequation}
Setting now $r = 0$, we obtain
\begin{splitequation}
\int_0^\infty u \, \tau\left( x^* \left( B - u A \right)_+ x \right) \total u &= \int_0^\infty s\, \tau\left( x^* \, \frac{1}{A + s} B \frac{1}{A + s} B \frac{1}{A + s} B \frac{1}{A + s} \, x \right) \total s \\
&= \int_0^\infty \tau\left( x^* \sqrt{ \frac{s}{A + s} } \, \left[ C_{A,B}(s) \right]^3 \sqrt{ \frac{s}{A + s} } x \right) \total s \eqend{.}
\end{splitequation}
The same arguments apply to higher derivatives, and we obtain
\begin{equation}
\int_0^\infty u^k \, \tau\left( x^* \left( B - u A \right)_+ x \right) \total u = \int_0^\infty \tau\left( x^* \sqrt{ \frac{s}{A + s} } \, \left[ C_{A,B}(s) \right]^{2+k} \sqrt{ \frac{s}{A + s} } x \right) \total s
\end{equation}
for any $k \in \mathbb{N}$.
\end{proof}

\subsection{Trace lemmas}
\label{sec:operatorint_technical}

In this section, we work directly with the unbounded densities and collect some technical lemmas concerning properties and behaviour of the tracial weight $\tau$. The standing assumption~\ref{ass:operatorint} is no longer assumed here.

\begin{lemma}
\label{lemma:operator_log_derivative}
Let $A, B, C \in L^1(\M,\tau)$ be positive self-adjoint operators with $\dom B \subset \dom C$ and $B \geq C$. Then it holds that
\begin{equation}
\label{eq:operator_log_derivative_int1}
\int_0^\infty \tau\left( A \frac{1}{C + s} - A \frac{1}{B + s} \right) \total s = \tau\left( A \ln B - A \ln C \right) \eqend{,}
\end{equation}
where both sides could be equal to $+\infty$, and if $A$ is injective, it also holds that
\begin{equation}
\label{eq:operator_log_derivative_int2}
\int_0^\infty \tau\left( \frac{\sqrt{A}}{A + s} B \frac{\sqrt{A}}{A + s} \right) \total s = \tau(B)
\end{equation}
and that
\begin{equation}
\lim_{\epsilon \searrow 0} \frac{1}{\epsilon} \, \tau\left( A \ln (A + \epsilon B) - A \ln A \right) = \tau(B) \eqend{.}
\end{equation}
\end{lemma}
\begin{proof}
We first show that the integrals are well-defined. For all $s \geq \epsilon > 0$, continuity of the map $s \mapsto (C + s)^{-1} \in \M$ in the norm of $\M$ and the continuity of left and right multiplication of $L^1(\M,\tau)$ by elements from $\M$ in $L^1(\M,\tau)$ norm, which follows from the noncommutative H{\"o}lder inequality, shows that $s \mapsto A (C + s)^{-1} - A (B + s)^{-1}$ is continuous in $L^1(\M,\tau)$ norm and bounded by $\norm{ A (C+s)^{-1} - A (B+s)^{-1} }_{L^1(\M,\tau)} \leq 2 \norm{ A }_{L^1(\M,\tau)} \epsilon^{-1}$. For bounded $A$, the further estimate
\begin{splitequation}
\norm{ A \frac{1}{C + s} - A \frac{1}{B + s} }_{L^1(\M,\tau)} &\leq \norm{ A } \norm{ \frac{1}{C + s} - \frac{1}{B + s} }_{L^1(\M,\tau)} \\
&= \norm{ A } \norm{ \frac{1}{C + s} (B-C) \frac{1}{B + s} }_{L^1(\M,\tau)} \\
&\leq \frac{\norm{A}}{s^2} \norm{ B-C }_{L^1(\M,\tau)}
\end{splitequation}
shows that the integrand is bounded by an integrable function of $s$ for all $s \geq \epsilon > 0$. Therefore, by Bochner's theorem \cite[Thm.~II.~2.2]{diesteluhl1977} the integral $\int_\epsilon^\infty \left[ A (C + s)^{-1} - A (B + s)^{-1} \right] \total s$ exists as a Bochner integral with values in $L^1(\M,\tau)$, and by Hille's theorem \cite[Thm. II.~2.6]{diesteluhl1977} we have
\begin{splitequation}
\int_\epsilon^\infty \tau\left( A \frac{1}{C + s} - A \frac{1}{B + s} \right) \total s &= \tau\left( A \int_\epsilon^\infty \left( \frac{1}{C+s} - \frac{1}{B+s} \right) \total s \right) \\
&= \tau\left( A \Bigl[ \ln(B + \epsilon) - \ln(C + \epsilon) \Bigr] \right) \eqend{.}
\end{splitequation}
Since $B \geq C$, it holds that $\tau\left( A (C + s)^{-1} - A (B + s)^{-1} \right) \geq 0$ and the left-hand side converges monotonically to $\int_0^\infty \tau\left( A (C + s)^{-1} - A (B + s)^{-1} \right) \total s$. Since the function $t \mapsto \ln(t + \epsilon)$ is operator monotone for all $\epsilon > 0$, we have $\ln(B + \epsilon) - \ln(C + \epsilon) \geq 0$ and the net $( \ln(B + \epsilon) - \ln(C + \epsilon) )_{\epsilon \in (0,1)}$ is increasing as $\epsilon \searrow 0$. By \lemmaref{lemma:affiliated_tau_convergence}, we then obtain $\lim_{\epsilon \searrow 0} \tau\left( A \bigl[ \ln(B + \epsilon) - \ln(C + \epsilon) \bigr] \right) = \tau\left( A \ln B - A \ln C \right)$. For unbounded $A$, we consider $A_n \defeq E^A_{[0,n]} A$, and obtain the conclusion again by monotone convergence and \lemmaref{lemma:affiliated_tau_convergence}.

For the integral \eqref{eq:operator_log_derivative_int2}, norm continuity and the reduction to bounded $B_n \defeq E^B_{[0,n]} B$ follow analogously. Employing then \lemmaref{lemma:operatorint_tau_function} with $x = \sqrt{ B_n } \in \M$, $C = A \in L^1(\M,\tau)$ and $g(t) = (t+s)^{-2}$, we obtain
\begin{splitequation}
\int_\epsilon^\infty \tau\left( \frac{\sqrt{A}}{A + s} \sqrt{B_n} \frac{\sqrt{A}}{A + s} \right) \total s &= \int_\epsilon^\infty \int_{(0,\infty)} \frac{1}{(t+s)^2} \total \mu_A^{\sqrt{B_n}}(t) \total s \\
&= \int_{(0,\infty)} \int_\epsilon^\infty \frac{1}{(t+s)^2} \total s \total \mu_A^{\sqrt{B_n}}(t) \\
&= \int_{(0,\infty)} \frac{1}{\epsilon + t} \total \mu_A^{\sqrt{B_n}}(t) = \tau\left( B_n \frac{A}{A + \epsilon} \right) \eqend{,}
\end{splitequation}
where in the second equality we employed Tonelli's theorem since the integrand is positive, and in the last used again \lemmaref{lemma:operatorint_tau_function}, now with $g(t) = (t + \epsilon)^{-1}$. By monotone convergence and normality of $\tau$, we obtain the conclusion as $\epsilon \searrow 0$.

For the final limit, using the first integral and resolvent identities we compute
\begin{splitequation}
&\frac{1}{\epsilon} \, \tau\left( A \ln(A + \epsilon B) - A \ln A \right) \\
&\quad= \frac{1}{\epsilon} \int_0^\infty \tau\left( A \frac{1}{A + s} - A \frac{1}{A + \epsilon B + s} \right) \total s \\
&\quad= \int_0^\infty \tau\left( A^\frac{1}{2} \frac{1}{A + s} B \frac{1}{A + \epsilon B + s} A^\frac{1}{2} \right) \total s \\
&\quad= \frac{1}{2} \int_0^\infty \tau\left( A^\frac{1}{2} \left( \frac{1}{A + s} + \frac{1}{A + \epsilon B + s} \right) B \left( \frac{1}{A + s} + \frac{1}{A + \epsilon B + s} \right) A^\frac{1}{2} \right) \total s \\
&\qquad- \frac{1}{2} \int_0^\infty \tau\left( A^\frac{1}{2} \frac{1}{A + s} B \frac{1}{A + s} A^\frac{1}{2} + A^\frac{1}{2} \frac{1}{A + \epsilon B + s} B \frac{1}{A + \epsilon B + s} A^\frac{1}{2} \right) \total s \eqend{.}
\end{splitequation}
Since each of the integrands is positive and monotonically increasing as $\epsilon \searrow 0$, the monotone convergence theorem and \lemmaref{lemma:affiliated_tau_convergence} show that
\begin{splitequation}
&\lim_{\epsilon \searrow 0} \int_0^\infty \tau\left( A^\frac{1}{2} \left( \frac{1}{A + s} + \frac{1}{A + \epsilon B + s} \right) B \left( \frac{1}{A + s} + \frac{1}{A + \epsilon B + s} \right) A^\frac{1}{2} \right) \total s \\
&\quad= 4 \int_0^\infty \tau\left( A^\frac{1}{2} \frac{1}{A + s} B \frac{1}{A + s} A^\frac{1}{2} \right) \total s = 4 \tau(B)
\end{splitequation}
and
\begin{splitequation}
&\lim_{\epsilon \searrow 0} \int_0^\infty \tau\left( A^\frac{1}{2} \frac{1}{A + \epsilon B + s} B \frac{1}{A + \epsilon B + s} A^\frac{1}{2} \right) \total s \\
&\quad= \int_0^\infty \tau\left( A^\frac{1}{2} \frac{1}{A + s} B \frac{1}{A + s} A^\frac{1}{2} \right) \total s = \tau(B) \eqend{,}
\end{splitequation}
and the conclusion follows.
\end{proof}

\begin{lemma}
\label{lemma:operatorint_tau_differentiable}
Let $C(r) \in L^1(\M,\tau)$ be continuous and continuously differentiable in $L^1(\M,\tau)$ norm for all $r \in [a,b]$, $b > a \geq 0$, with $\sup_{r \in [a,b]} \norm{ C'(r) }_{L^1(\M,\tau)} < \infty$. Then it holds that $\partial_r \tau( C(r) ) = \tau( C'(r) )$.
\end{lemma}
\begin{proof}
By assumption, $\int_s^t \norm{ C'(r) }_{L^1(\M,\tau)} \total r$ is finite for all $b \geq t \geq s \geq a$. Hence, by Bochner's theorem~\cite[Thm.~II.~2.2]{diesteluhl1977} $\int_s^t C'(r) \total r$ exists as Bochner integral with values in $L^1(\M,\tau)$. Since $L^1(\M,\tau)$ is a Banach space, the fundamental theorem of calculus applies to the norm continuous function $C'(r)$, and we obtain using Hille's theorem~\cite[Thm.~II.~2.6]{diesteluhl1977}
\begin{equation}
\tau(C(t)) - \tau(C(s)) = \tau\left( \int_s^t C'(r) \total r \right) = \int_s^t \tau( C'(r) ) \total r \eqend{,}
\end{equation}
from which the conclusion follows by differentiating with respect to $t$.
\end{proof}

\begin{lemma}
\label{lemma:operatorint_tau_function}
For a bounded Borel measurable function $g \colon \mathbb{R}_+ \to \mathbb{R}_+$, a positive operator $C \in L^1(\M,\tau)$ and $x \in \M$, the identity
\begin{equation}
\tau\left( x \, C g(C) x^* \right) = \int_{\mathbb{R}_+} g(t) \total \mu_C^x(t)
\end{equation}
holds, where $\mu_C^x$ is the finite Borel measure with
\begin{equation}
\mu_C^x(M) \defeq \tau\left( x \, C E^C_M x^* \right) \quad\text{for}\quad M \in \mathcal{B}(\mathbb{R}_+) \eqend{.}
\end{equation}
If $C$ is bounded, the support of $\mu^C_x$ is the interval $[0, \norm{C}]$. Furthermore, the tail function $\mu_C^x((t,\infty))$ is right-continuous.
\end{lemma}
\begin{proof}
The $\sigma$-additivity of $\mu_C^x$ is inherited from the one of the spectral measures $E^C_M$, and positivity follows because $\tau$ is a positive weight. Furthermore, $\mu_C^x$ is finite since
\begin{equation}
\mu_C^x(\mathbb{R}_+) = \tau\left( x \, C x^* \right) \leq \norm{ x }^2 \tau(C) < \infty \eqend{,}
\end{equation}
since $\tau(C) < \infty$ by assumption, and the tail function is right-continuous because $\tau$ is normal and for a decreasing sequence $\{ t_n \}_{n \in \mathbb{N}}$ with $\lim_{n \to \infty} t_n = t$ the spectral projections $E^C_{(t_n,\infty)}$ form an increasing sequence strongly converging to $E^C_{(t,\infty)}$~\cite[Thm.~2.7.11]{brattelirobinson1}. The support statements are clear. Approximating then $g$ by a monotonically increasing sequence of simple functions $\{ g_n \}_{n \in \mathbb{N}}$ of the form $g_n(t) = \sum_{i=1}^{s_n} a_{n,i} \, \chi_{M_i}(t)$ with $s_n \in \mathbb{N}$, we obtain
\begin{splitequation}
\int_{\mathbb{R}_+} g(t) \total \mu_C^x(t) &= \lim_{n \to \infty} \sum_{i=1}^{s_n} a_{n,i} \, \mu_C^x\left( M_i \right) = \lim_{n \to \infty} \sum_{i=1}^{s_n} a_{n,i} \, \tau\left( x C E^C_{M_i} x^* \right) \\
&= \lim_{n \to \infty} \tau\left( x \left[ \sum_{i=1}^{s_n} a_{n,i} \, C E^C_{M_i} \right] x^* \right) = \lim_{n \to \infty} \tau\left( x \, C \int_{\mathbb{R}_+} g_n(t) \total E^C(t) \, x^* \right) \\
&= \tau\left( x \, C g(C) x^* \right) \eqend{,}
\end{splitequation}
where we used linearity and \lemmaref{lemma:affiliated_tau_convergence}.
\end{proof}

\begin{lemma}
\label{lemma:resolvent_bound_imaginary}
Let $C$, $D$ be densely defined closed operators on a Hilbert space $\mathcal{H}$ whose numerical ranges are contained in the half-planes $\Im z \geq \gamma$ and $\Im z \leq \delta$, respectively. For every $s$ with $- \mathi s \not\in \sigma(C)$ and $s + \gamma > 0$, we have the bound $\norm{ ( C + \mathi s )^{-1} } \leq (s+\gamma)^{-1}$, and for every $s$ with $- \mathi s \not\in \sigma(D)$ and $s + \delta < 0$, we have the bound $\norm{ ( D + \mathi s )^{-1} } \leq \abs{s+\delta}^{-1}$.
\end{lemma}
\begin{proof}
For $u \in \dom(C)$, define $v \defeq (C + \mathi s) u$. Then we have
\begin{equation}
\Im \left( u, v \right)_\mathcal{H} = \Im \left( u, C u \right)_\mathcal{H} + s \norm{ u }_\mathcal{H}^2 \geq (s+\gamma) \norm{ u }_\mathcal{H}^2 \geq 0
\end{equation}
by the hypothesis on the numerical range of $C$. Using the Cauchy--Schwarz inequality, it follows that
\begin{equation}
(s+\gamma) \norm{ u }_\mathcal{H}^2 \leq \Im \left( u, v \right)_\mathcal{H} \leq \abs{ \left( u, v \right)_\mathcal{H} } \leq \norm{ u }_\mathcal{H} \norm{ v }_\mathcal{H} \eqend{,}
\end{equation}
and hence
\begin{equation}
\norm{ v }_\mathcal{H} = \norm{ (C + \mathi s) u }_\mathcal{H} \geq (s+\gamma) \norm{ u }_\mathcal{H} \eqend{.}
\end{equation}
Since $- \mathi s \not\in \sigma(C)$ and $C$ is densely defined and closed, the inverse $(C + \mathi s)^{-1}$ is bounded and thus defined on all of $\mathcal{H}$. Hence, taking $u = (C + \mathi s)^{-1} w \in \dom(C)$ for any $w \in \mathcal{H}$, we obtain
\begin{equation}
\norm{ w }_\mathcal{H} \geq (s+\gamma) \norm{ (C + \mathi s)^{-1} w }_\mathcal{H} \eqend{,}
\end{equation}
and taking the supremum over all $w$ with $\norm{ w }_\mathcal{H} = 1$ the required bound on $\norm{ (C + \mathi s)^{-1} }$ follows. The bound on $\norm{ ( D + \mathi s )^{-1} } = \norm{ ( - D - \mathi s )^{-1} }$ is then obtained by noting that $-D$ has numerical range contained in the half-plane $\Im z \geq - \delta$.
\end{proof}
\begin{lemma}
\label{lemma:resolvent_bound_complex}
Let $C$, $D$ be densely defined closed symmetric operators on a Hilbert space $\mathcal{H}$ with a common dense domain $\mathcal{D}$, and assume that $C > 0$ is boundedly invertible. Then we have $\norm{ (C + \mathi D)^{-1} } \leq \norm{ C^{-1} }$.
\end{lemma}
\begin{proof}
By~\cite[Prop.~2.8]{schmuedgen2012}, $\norm{ (C + \mathi D - \lambda)^{-1} }$ is bounded from above by the inverse of the distance of $\lambda \in \mathbb{C}$ from the numerical range of $C + \mathi D$. Since for $u \in \mathcal{D}$, we have
\begin{equation}
\abs{ \left( u, (C + \mathi D) u \right)_\mathcal{H} } \geq \abs{ \Re \left( u, (C + \mathi D) u \right)_\mathcal{H} } = \abs{ \left( u, C u \right)_\mathcal{H} } \eqend{,}
\end{equation}
the distance of $0$ from the numerical range of $C + \mathi D$ is bounded from below by the distance of $0$ from the numerical range of $C$. Since $C^{-1}$ is bounded and positive, this distance is exactly $\norm{ C^{-1} }^{-1}$, and the conclusion follows.
\end{proof}

\begin{acknowledgments}
\thmref{thm:D=Sgeneral} has been shown independently and simultaneous to our work by Koßmann, Schwonnek,
Liu, and Cheng \cite{kossmannschwonnekliucheng2026}.

The idea to use Haagerup reduction to extend \thmref{thm:D=Ssemifinite} from the semifinite to the general case in \thmref{thm:D=Sgeneral} is due to Lauritz van Luijk, who has shown us this argument.

This work has been funded by the Deutsche Forschungsgemeinschaft (DFG, German Research Foundation) --- project no. 463380247 within the Heisenberg grant LE 2222/3-1 (``Quantenfelder und Operatoralgebren'').
\end{acknowledgments}

\appendix

\subsection*{Conflict of interest statement}

On behalf of all authors, the corresponding author states that there is no conflict of interest.

\subsection*{Data availability statement}

This manuscript has no associated data.

\bibliographystyle{JHEP}
\bibliography{literature}

\providecommand{\href}[2]{#2}\begingroup\raggedright\begin{thebibliography}{10}

\bibitem{csiszar1963}
I.~Csisz{\'a}r, \emph{{Eine informationstheoretische Ungleichung und ihre
  Anwendung auf den Beweis der Ergodizit{\"a}t von Markoffschen Ketten}},
  {\emph{Magyar Tud. Akad. Mat. Kutat{\'o} Int. K{\"o}zl.} {\bfseries 8} (1963)
  85},
  \href{https://real.mtak.hu/id/eprint/201426}{https://real.mtak.hu/id/eprint/201426}.

\bibitem{sasonverdu2016}
I.~Sason and S.~Verd{\'u}, \emph{{$f$-divergence Inequalities}},
  \href{https://doi.org/10.1109/TIT.2016.2603151}{\emph{IEEE Trans. Inf.
  Theory} {\bfseries 62} (2016) 5973}
  [\href{https://arxiv.org/abs/1508.00335}{{\ttfamily 1508.00335}}].

\bibitem{takesaki2001}
M.~Takesaki, \emph{{Theory of Operator Algebras I}}, vol.~124 of
  \emph{{Encyclopaedia of Mathematical Sciences}}, Springer-Verlag, Berlin,
  Heidelberg, Germany (2003),
  \href{https://doi.org/10.1007/978-1-4612-6188-9}{10.1007/978-1-4612-6188-9}.

\bibitem{watrous2018}
J.~Watrous, \emph{{The Theory of Quantum Information}}, Cambridge University
  Press, Cambridge, UK (2018).

\bibitem{holevo2019}
A.S.~Holevo, \emph{{Quantum Systems, Channels, Information: A Mathematical
  Introduction}}, De Gruyter Brill, Berlin, Germany (2019).

\bibitem{vanluijkstottmeisterwernerwilming2024}
L.~van Luijk, A.~Stottmeister, R.F.~Werner and H.~Wilming, \emph{{Embezzlement
  of entanglement, quantum fields, and the classification of von Neumann
  algebras}},  \href{https://arxiv.org/abs/2401.07299}{{\ttfamily 2401.07299}}.

\bibitem{vanluijkstottmeisterwernerwilming2025}
L.~van Luijk, A.~Stottmeister, R.F.~Werner and H.~Wilming, \emph{{Pure State
  Entanglement and von Neumann Algebras}},
  \href{https://doi.org/10.1007/s00220-025-05465-5}{\emph{Commun. Math. Phys.}
  {\bfseries 406} (2025) 296}
  [\href{https://arxiv.org/abs/2409.17739}{{\ttfamily 2409.17739}}].

\bibitem{emch1972}
G.G.~Emch, \emph{{Algebraic methods in statistical mechanics and quantum field
  theory}}, Wiley Interscience, New York, USA (1972).

\bibitem{haag1996}
R.~Haag, \emph{{Local quantum physics: Fields, particles, algebras}},
  Springer-Verlag, Berlin, Heidelberg, Germany, second revised and enlarged~ed.
  (1996).

\bibitem{puszworonowicz1975}
W.~Pusz and S.L.~Woronowicz, \emph{{Functional calculus for sesquilinear forms
  and the purification map}},
  \href{https://doi.org/10.1016/0034-4877(75)90061-0}{\emph{Rept. Math. Phys.}
  {\bfseries 8} (1975) 159}.

\bibitem{kuboando1980}
F.~Kubo and T.~Ando, \emph{{Means of positive linear operators}},
  \href{https://doi.org/10.1007/BF01371042}{\emph{Math. Ann.} {\bfseries 246}
  (1980) 205}.

\bibitem{araki1976}
H.~Araki, \emph{{Relative Entropy of States of von Neumann Algebras}},
  \href{https://doi.org/10.2977/prims/1195191148}{\emph{Publ. RIMS, Kyoto
  Univ.} {\bfseries 11} (1976) 809}.

\bibitem{araki1977}
H.~Araki, \emph{{Relative Entropy of States of von Neumann Algebras II}},
  \href{https://doi.org/10.2977/prims/1195190105}{\emph{Publ. RIMS, Kyoto
  Univ.} {\bfseries 13} (1977) 173}.

\bibitem{arakimasuda1982}
H.~Araki and T.~Masuda, \emph{{Positive Cones and $L_p$-Spaces for von Neumann
  Algebras}}, \href{https://doi.org/10.2977/prims/1195183577}{\emph{Publ. RIMS,
  Kyoto Univ.} {\bfseries 18} (1982) 339}.

\bibitem{petz1985}
D.~Petz, \emph{{Quasi-entropies for States of a von Neumann Algebra}},
  \href{https://doi.org/10.2977/prims/1195178929}{\emph{Publ. RIMS, Kyoto
  Univ.} {\bfseries 21} (1985) 787}.

\bibitem{petz1986a}
D.~Petz, \emph{{Quasi-entropies for finite quantum systems}},
  \href{https://doi.org/10.1016/0034-4877(86)90067-4}{\emph{Rept. Math. Phys.}
  {\bfseries 23} (1986) 57}.

\bibitem{hiai2018}
F.~Hiai, \emph{{Quantum $f$-divergences in von Neumann algebras. I. Standard
  $f$-divergences}}, \href{https://doi.org/10.1063/1.5039973}{\emph{J. Math.
  Phys.} {\bfseries 59} (2018) 012202}
  [\href{https://arxiv.org/abs/1805.02050}{{\ttfamily 1805.02050}}].

\bibitem{hayashi2017}
M.~Hayashi, \emph{{Quantum Information Theory: Mathematical Foundation}},
  {Graduate Texts in Physics}, Springer-Verlag, Berlin, Heidelberg, Germany,
  second~ed. (2017),
  \href{https://doi.org/10.1007/978-3-662-49725-8}{10.1007/978-3-662-49725-8}.

\bibitem{jaksicogatapilletseiringer2012}
V.~Jak{\v s}i{\'c}, Y.~Ogata, C.A.~Pillet and R.~Seiringer, \emph{{Quantum
  Hypothesis Testing and Non-Equilibrium Statistical Mechanics}},
  \href{https://doi.org/10.1142/S0129055X12300026}{\emph{Rev. Math. Phys.}
  {\bfseries 24} (2012) 1230002}
  [\href{https://arxiv.org/abs/1109.3804}{{\ttfamily 1109.3804}}].

\bibitem{hirche2018}
C.~Hirche, \emph{{From asymptotic hypothesis testing to entropy inequalities}},
  {PhD thesis}, Universitat Aut{\`o}noma de Barcelona, Bellaterra, Spain, 2018.
\newblock \href{https://arxiv.org/abs/1812.05142}{{\ttfamily 1812.05142}}.

\bibitem{helstrom1969}
C.W.~Helstrom, \emph{{Quantum Detection and Estimation Theory}},
  \href{https://doi.org/10.1007/BF01007479}{\emph{J. Stat. Phys.} {\bfseries 1}
  (1969) 231}.

\bibitem{jencova2010}
A.~Jen{\v c}ov{\'a}, \emph{{Quantum Hypothesis Testing and Sufficient
  Subalgebras}}, \href{https://doi.org/10.1007/s11005-010-0398-0}{\emph{Lett.
  Math. Phys.} {\bfseries 93} (2010) 15}
  [\href{https://arxiv.org/abs/0810.4045}{{\ttfamily 0810.4045}}].

\bibitem{sharmawarsi2012}
N.~Sharma and N.A.~Warsi, \emph{{On the strong converses for the quantum
  channel capacity theorems}},
  \href{https://arxiv.org/abs/1205.1712}{{\ttfamily 1205.1712}}.

\bibitem{sharmawarsi2013}
N.~Sharma and N.A.~Warsi, \emph{{Fundamental Bound on the Reliability of
  Quantum Information Transmission}},
  \href{https://doi.org/10.1103/PhysRevLett.110.080501}{\emph{Phys. Rev. Lett.}
  {\bfseries 110} (2013) 080501}
  [\href{https://arxiv.org/abs/1302.5281}{{\ttfamily 1302.5281}}].

\bibitem{hirchetomamichel2024}
C.~Hirche and M.~Tomamichel, \emph{{Quantum R{\'e}nyi and f-Divergences from
  Integral Representations}},
  \href{https://doi.org/10.1007/s00220-024-05087-3}{\emph{Commun. Math. Phys.}
  {\bfseries 405} (2024) 208}
  [\href{https://arxiv.org/abs/2306.12343}{{\ttfamily 2306.12343}}].

\bibitem{frenkel2023}
P.E.~Frenkel, \emph{{Integral formula for quantum relative entropy implies data
  processing inequality}},
  \href{https://doi.org/10.22331/q-2023-09-07-1102}{\emph{Quantum} {\bfseries
  7} (2023) 1102} [\href{https://arxiv.org/abs/2208.12194}{{\ttfamily
  2208.12194}}].

\bibitem{liuhirchecheng2025}
P.-C.~Liu, C.~Hirche and H.-C.~Cheng, \emph{{Layer Cake Representations for
  Quantum Divergences}},  \href{https://arxiv.org/abs/2507.07065}{{\ttfamily
  2507.07065}}.

\bibitem{jencova2024}
A.~Jen{\v c}ov{\'a}, \emph{{Recoverability of quantum channels via hypothesis
  testing}}, \href{https://doi.org/10.1007/s11005-024-01775-2}{\emph{Lett.
  Math. Phys.} {\bfseries 114} (2024) 31}
  [\href{https://arxiv.org/abs/2303.11707}{{\ttfamily 2303.11707}}].

\bibitem{vanluijkwilming2026}
L.~van Luijk and H.~Wilming, \emph{{Sufficiency and Petz recovery for positive
  maps}},  \href{https://arxiv.org/abs/2604.08380}{{\ttfamily 2604.08380}}.

\bibitem{hircherouzefranca2023}
C.~Hirche, C.~Rouz{\'e} and D.S.~Fran{\c c}a, \emph{{Quantum Differential
  Privacy: An Information Theory Perspective}},
  \href{https://doi.org/10.1109/TIT.2023.3272904}{\emph{IEEE Trans. Info.
  Theor.} {\bfseries 69} (2023) 5771}
  [\href{https://arxiv.org/abs/2202.10717}{{\ttfamily 2202.10717}}].

\bibitem{beigi2013}
S.~Beigi, \emph{{Sandwiched R{\'e}nyi divergence satisfies data processing
  inequality}}, \href{https://doi.org/10.1063/1.4838855}{\emph{J. Math. Phys.}
  {\bfseries 54} (2013) 122202}
  [\href{https://arxiv.org/abs/1306.5920}{{\ttfamily 1306.5920}}].

\bibitem{muellerhermesreeb2017}
A.~M{\"u}ller-Hermes and D.~Reeb, \emph{{Monotonicity of the Quantum Relative
  Entropy Under Positive Maps}},
  \href{https://doi.org/10.1007/s00023-017-0550-9}{\emph{Ann. H. Poincar{\'e}}
  {\bfseries 18} (2017) 1777}
  [\href{https://arxiv.org/abs/1512.06117}{{\ttfamily 1512.06117}}].

\bibitem{jencova2018}
A.~Jen{\v c}ov{\'a}, \emph{{R{\'e}nyi Relative Entropies and Noncommutative
  $L_p$-Spaces}}, \href{https://doi.org/10.1007/s00023-018-0683-5}{\emph{Ann.
  H. Poincar{\'e}} {\bfseries 19} (2018) 2513}
  [\href{https://arxiv.org/abs/1609.08462}{{\ttfamily 1609.08462}}].

\bibitem{kosaki1986}
H.~Kosaki, \emph{{Relative entropy of states: a variational expression}},
  {\emph{J. Operator Theory} {\bfseries 16} (1986) 335},
  \href{https://www.theta.ro/jot/archive/1986-016-002/1986-016-002-010.pdf}{https://www.theta.ro/jot/archive/1986-016-002/1986-016-002-010.pdf}.

\bibitem{friedland2026}
S.~Friedland, \emph{{A generalization of Frenkel's formula}},
  \href{https://arxiv.org/abs/2602.10962}{{\ttfamily 2602.10962}}.

\bibitem{liebloss1997}
E.H.~Lieb and M.~Loss, \emph{{Analysis}}, vol.~14 of \emph{{Graduate Studies in
  Mathematics}}, American Mathematical Society, Providence, Rhode Island, USA,
  second~ed. (2001).

\bibitem{ogata2011}
Y.~Ogata, \emph{{A generalization of Powers--St{\o}rmer inequality}},
  \href{https://doi.org/10.1007/s11005-011-0504-y}{\emph{Lett. Math. Phys.}
  {\bfseries 97} (2011) 339}.

\bibitem{powersstormer1970}
R.T.~Powers and E.~St{\o}rmer, \emph{{Free States of the Canonical
  Anticommutation Relations}},
  \href{https://doi.org/10.1007/BF01645492}{\emph{Commun. Math. Phys.}
  {\bfseries 16} (1970) 1}.

\bibitem{chengliu2026}
H.-C.~Cheng and P.-C.~Liu, \emph{{Multiple Quantum Hypothesis Testing: One-Shot
  Pairwise Bounds and Sharp Asymptotics}},
  \href{https://arxiv.org/abs/2606.06246}{{\ttfamily 2606.06246}}.

\bibitem{chengliu2025}
H.-C.~Cheng and P.-C.~Liu, \emph{{Error Exponents for Quantum Packing Problems
  via An Operator Layer Cake Theorem}},
  \href{https://arxiv.org/abs/2507.06232v3}{{\ttfamily 2507.06232v3}}.

\bibitem{chenggourlamiliu2025}
H.-C.~Cheng, G.~Gour, L.~Lami and P.-C.~Liu, \emph{{The operator layer cake
  theorem is equivalent to Frenkel's integral formula}},
  \href{https://arxiv.org/abs/2512.04345}{{\ttfamily 2512.04345}}.

\bibitem{brattelirobinson1}
O.~Bratteli and D.W.~Robinson, \emph{{Operator Algebras and Quantum Statistical
  Mechanics}}, vol.~1 of \emph{{Theoretical and Mathematical Physics}},
  Springer-Verlag, Berlin, Heidelberg, Germany (1987),
  \href{https://doi.org/10.1007/978-3-662-02520-8}{10.1007/978-3-662-02520-8}.

\bibitem{rudin1987}
W.~Rudin, \emph{{Real and complex analysis}}, McGraw-Hill, Singapore, third
  international~ed. (1987).

\bibitem{dixmier1981}
J.~Dixmier, \emph{{Von Neumann Algebras}}, North-Holland, Amsterdam, The
  Netherlands (1981).

\bibitem{ohyapetz1993}
M.~Ohya and D.~Petz, \emph{{Quantum Entropy and its Use}}, {Theoretical and
  Mathematical Physics}, Springer-Verlag, Berlin, Heidelberg, Germany (1993).

\bibitem{petz1986c}
D.~Petz, \emph{{Sufficient Subalgebras and the Relative Entropy of States of a
  von Neumann Algebra}},
  \href{https://doi.org/10.1007/BF01212345}{\emph{Commun. Math. Phys.}
  {\bfseries 105} (1986) 123}.

\bibitem{takesaki2003}
M.~Takesaki, \emph{{Theory of Operator Algebras II}}, vol.~125 of
  \emph{{Encyclopaedia of Mathematical Sciences}}, Springer-Verlag, Berlin,
  Heidelberg, Germany (2003),
  \href{https://doi.org/10.1007/978-3-662-10451-4}{10.1007/978-3-662-10451-4}.

\bibitem{terp1981}
M.~Terp, ``{$L^p$ spaces associated with von Neumann algebras}.'' Report No. 3a
  + 3b, K{\o}benhavns Universitets Matematiske Institut, June, 1981.

\bibitem{correadasilva2018}
R.~{Correa Da Silva}, \emph{{Lecture Notes on Noncommutative {$L_p$}-Spaces}},
  \href{https://arxiv.org/abs/1803.02390}{{\ttfamily 1803.02390}}.

\bibitem{pedersen2018}
G.K.~Pedersen, \emph{{$C^*$-Algebras and Their Automorphism Groups}},  {Pure
  and Applied Mathematics}, Academic Press (2018),
  \href{https://doi.org/10.1016/B978-0-12-814122-9.00005-2}{DOI}.

\bibitem{luczakpodsedkowskawieczorek2021}
A.~{\L}uczak, H.~Pods{\k{e}}dkowska and R.~Wieczorek, \emph{{Relative modular
  operator in semifinite von Neumann algebras and its use}},
  \href{https://doi.org/10.1142/S0129055X22500404}{\emph{Rev. Math. Phys.}
  {\bfseries 35} (2021) 2250040}
  [\href{https://arxiv.org/abs/1912.09633}{{\ttfamily 1912.09633}}].

\bibitem{umegaki1962}
H.~Umegaki, \emph{{Conditional expectation in an operator algebra, IV (Entropy
  and information)}}, \href{https://doi.org/10.2996/kmj/1138844604}{\emph{Kodai
  Math. Sem. Rep.} {\bfseries 14} (1962) 59}.

\bibitem{haagerup1979}
U.~Haagerup, \emph{{On the dual weights for crossed products of von Neumann
  algebras I: Removing separability conditions}},
  \href{https://doi.org/10.7146/math.scand.a-11768}{\emph{Math. Scand.}
  {\bfseries 43} (1978) 99–118}.

\bibitem{goldsteinlabuschagne2025}
S.~Goldstein and L.~Labuschagne, \emph{{Noncommutative Measures and $L^p$ and
  Orlicz Spaces, with Applications to Quantum Physics}}, Oxford University
  Press, Oxford, UK (2025),
  \href{https://doi.org/10.1093/oso/9780198950202.001.0001}{10.1093/oso/9780198950202.001.0001}.

\bibitem{haagerupjungexu2009}
U.~Haagerup, M.~Junge and Q.~Xu, \emph{{A reduction method for noncommutative
  $L_p$-spaces and applications}},
  \href{https://doi.org/10.1090/S0002-9947-09-04935-6}{\emph{Trans. Amer. Math.
  Soc.} {\bfseries 362} (2010) 2125}
  [\href{https://arxiv.org/abs/0806.3635}{{\ttfamily 0806.3635}}].

\bibitem{kosaki1984b}
H.~Kosaki, \emph{{On the Continuity of the Map $\varphi\to
  \left\lvert\varphi\right\rvert$ from the Predual of a $W^*$-Algebra}},
  \href{https://doi.org/10.1016/0022-1236(84)90055-7}{\emph{J. Funct. Analysis}
  {\bfseries 59} (1984) 123}.

\bibitem{diesteluhl1977}
J.~Diestel and J.J.~{Uhl, Jr.}, \emph{{Vector measures}}, vol.~15 of
  \emph{{Mathematical Surveys}}, American Mathematical Society, Providence,
  Rhode Island, USA (1977).

\bibitem{schmuedgen2012}
K.~Schmüdgen, \emph{{Unbounded Self-adjoint Operators on Hilbert Space}},
  Springer Dordrecht, Dordrecht, The Netherlands (2012),
  \href{https://doi.org/10.1007/978-94-007-4753-1}{10.1007/978-94-007-4753-1}.

\bibitem{kossmannschwonnekliucheng2026}
G.~Koßmann, R.~Schwonnek, P.-C.~Liu and H.-C.~Cheng, \emph{{Device-independent
  Quantum Key Distribution in the commuting operator framework}}, {\emph{To
  appear}}.

\end{thebibliography}\endgroup

\end{document}